\documentclass[11pt,reqno]{amsart}
\usepackage[a4paper,left=2.5cm,right=2.5cm,top=3cm,bottom=2cm]{geometry}
\usepackage{amsmath,amssymb,amsthm,mathtools}
\usepackage{enumitem}
\usepackage{mathrsfs}
\usepackage{tikz-cd}
\usepackage{tikz}
\usepackage{longtable}
\usepackage{makecell}
\usepackage{array,tabularx}
\usepackage{graphicx}
\usepackage{pdflscape}
\usepackage{float}
\usepackage{subcaption}
\usepackage{diagbox}
\usetikzlibrary{decorations.pathreplacing,calc}
\usepackage[colorlinks=true, pdfstartview=FitV, linkcolor=blue,%
citecolor=blue, urlcolor=blue]{hyperref}
\newcolumntype{Y}{>{\centering\arraybackslash}X}
\numberwithin{table}{section}
\numberwithin{equation}{section}
\newtheorem{theorem}{Theorem}[section]
\newtheorem{proposition}[theorem]{Proposition}
\newtheorem{lemma}[theorem]{Lemma}
\newtheorem{corollary}[theorem]{Corollary}
\theoremstyle{definition}
\newtheorem{definition}[theorem]{Definition}
\newtheorem{example}[theorem]{Example}
\newtheorem{remark}[theorem]{Remark}

\newcommand{\Z}{\mathbb Z}
\newcommand{\Q}{\mathbb Q}
\newcommand{\wt}{\operatorname{wt}}

\newcommand{\ol}{\overline}

\newcommand{\cl}{\mathrm{cl}}
\newcommand{\ch}{\mathrm{ch}}
\newcommand{\Hom}{\mathrm{Hom}}
\newcommand{\veps}{\varepsilon}
\newcommand{\vphi}{\varphi}
\newcommand{\et}[1]{\widetilde e_{#1}}
\newcommand{\ft}[1]{\widetilde f_{#1}}

\newcommand{\ep}{\varepsilon}

\newcommand{\qfrak}{\mathfrak q}

\allowdisplaybreaks
\title[Closed formulas for energy functions on tensor squares of perfect crystals]{Closed formulas for energy functions on tensor\\ squares of higher-level perfect crystals in\\ classical affine types}
\author{Shaolong Han}
\address{Beijing International Center for Mathematical Research, Peking University, Beijing 100871, P.R. China}
\email{hanshaolong@bicmr.pku.edu.cn}
\begin{document}

\begin{abstract}
	For every \(l\geq1\), we construct explicit closed-form coordinate formulas for the local energy functions on the tensor products \(B_l\otimes B_l\) of level-\(l\) perfect crystals in classical affine types. 
	A single finite-level max-linear formula covers all seven types: its two branches coincide in type \(A_n^{(1)}\), yielding a cyclic maximum of partial sums, whereas in the remaining six types each branch is the maximum of finitely many explicit piecewise-linear expressions in barred coordinates, with type-dependent boundary data.
	We verify the defining local-energy recursion directly on
	the finite crystals and derive equivalent recursive forms, allowing the energy to be evaluated without applying the combinatorial \(R\)-matrix.
	Substitution into the KMN path character formula gives explicit positive coordinate path sums for the characters of all level-\(l\) integrable highest weight modules. After principal specialization, the path exponent can be rewritten as a weighted sum of a position-independent adjacent-pair statistic; comparison with the specialized Weyl--Kac character formula yields a uniform family of level-\(l\) Rogers--Ramanujan-type identities equating these sums with explicit infinite products. For a representative low-rank case at level two in each family, we display the complete adjacent-pair degree matrix and list the resulting identities.
\end{abstract}

\maketitle

\tableofcontents
\section*{Introduction}

Perfect crystals provide finite combinatorial models for integrable highest weight representations of quantum affine algebras. For a perfect crystal of level \(l\) and a dominant weight of the same level, the Kyoto path realization identifies the corresponding highest weight crystal with the set of semi-infinite paths whose tails agree with a prescribed ground-state path \cite{KMN1,KMN2}. The classical weight of a path is determined by the weights of its constituent factors, whereas its affine degree is governed by the local energies associated with adjacent pairs.
Local energy is therefore the finite datum connecting perfect crystals, affine path degrees, and character formulas, and its explicit evaluation is essential for making the positive path-sum character formula computationally effective. 
Perfect crystals for the classical affine types were applied to the study of Demazure modules in \cite{KMOTU98}.

\vskip 2mm

Energy functions arose in the study of solvable lattice models, where the corner transfer matrix method expressed local state probabilities in terms of one-dimensional configuration sums whose normalized limits were identified with affine Lie algebra string functions \cite{DJKMO89}.
In type \(A_{n-1}^{(1)}\), Nakayashiki and Yamada expressed the Lascoux--Sch\"utzenberger charge in terms of sums of local energies, yielding energy-sum formulas for Kostka polynomials and a limit formula for the \(\widehat{\mathfrak{sl}}_n/\mathfrak{sl}_n\) branching functions \cite{NY97}. Lenart and Schilling identified global energy with the negative of the corresponding charge statistic on tensor products of single-column Kirillov--Reshetikhin crystals with weakly decreasing heights in types \(A_{n-1}^{(1)}\) and \(C_n^{(1)}\) \cite{LS13}. Energy functions also play central roles in one-dimensional sums and fermionic formulas \cite{HKOTT02}, virtual-crystal constructions \cite{OSS03}, the \(X=M\) framework \cite{Sch07}, parabolic Lusztig \(q\)-analogues \cite{LOS12}, and the realization of affine Demazure gradings by global energy \cite{ST12}. 
These developments demonstrate the reach of energy functions, but generally take local energy as input or study global statistics built from it. Our problem is to determine the local energy itself directly in finite-level coordinates.

\vskip 2mm

The crystals considered here belong to the broader Kirillov--Reshetikhin crystal framework. 
The existence of Kirillov--Reshetikhin crystals for all nonexceptional affine types was established in \cite{OS08}, and explicit combinatorial realizations were constructed in \cite{FOS09}. 
For a different uniform family of level-one perfect crystals, Benkart, Frenkel, Kang, and Lee gave a construction valid for all affine Lie algebras and determined the corresponding energies from the classical connected components of the tensor square \cite{BFKL06}. 
At finite level, energy can also be computed together with the combinatorial \(R\)-matrix. 
Pairing and piecewise-linear rules are available for symmetric-tensor crystals of type \(A_n^{(1)}\) \cite{HHIKTT01}; insertion and reverse-bumping algorithms for \(B_l\otimes B_k\) were developed for types \(C_n^{(1)}\) and \(A_{2n-1}^{(2)}\) in \cite{HKOT00}, while corresponding column-insertion constructions for types \(B_n^{(1)}\), \(D_n^{(1)}\), \(A_{2n}^{(2)}\), and \(D_{n+1}^{(2)}\) were given in \cite{HKOT02}.
Related constructions for rectangular Kirillov--Reshetikhin crystals appear in \cite{OS10}. 
Although these methods make the energy algorithmically computable and, in several cases, apply to inhomogeneous tensor products \(B_l\otimes B_k\) or more general rectangular Kirillov--Reshetikhin crystals, they do not provide a closed formula for the energy alone that is both expressed in the original finite-level coordinates of the two input factors and uniform across the seven nonexceptional affine families. 
Deriving such a formula for the repeated evaluation of local energy along crystal paths is therefore a distinct problem at finite level.

\vskip 2mm

The closest precursor to our max-linear formulas is the work of Kang,
Kashiwara, and Misra \cite{KKM94}. 
For the seven nonexceptional affine types, they constructed the limit crystals \(B_\infty\) associated with coherent families of perfect crystals and described the corresponding energy functions on \(B_\infty\otimes B_\infty\) as maxima of linear functions. 
A finite crystal \(B_l\), however, is related to \(B_\infty\) through shifted embeddings involving one-element crystals, rather than through a canonical unshifted inclusion. 
Consequently, although the limit-crystal formulas provide natural candidates for finite-level energy, their validity on the original tensor product \(B_l\otimes B_l\) does not follow from a formal restriction: the coordinate shifts, finite coordinate constraints, type-dependent boundary data, and additive normalization must
be taken into account.

\vskip 2mm

We therefore address the finite-level problem of expressing the local energy on the homogeneous tensor products \(B_l\otimes B_l\) directly in the original crystal coordinates, for every \(l\geq1\).
To the best of our knowledge, no such formula has previously been stated and intrinsically verified with the finite boundary data made explicit and all seven families organized within a single max-linear form.

\vskip 2mm

A local energy function is unique up to an additive constant. For every \(l\geq1\), our main theorem fixes a concrete representative on
\(B_l\otimes B_l\) in each of the affine types $A_n^{(1)}$, $A_{2n}^{(2)}$, $D_{n+1}^{(2)}$, $C_n^{(1)}$, $A_{2n-1}^{(2)}$,
$D_n^{(1)}$ and $B_n^{(1)}$.
All seven families are governed by a single finite-level max-linear
formula. 
In type \(A_n^{(1)}\), the two branches of the unified formula coincide, so the energy reduces to a cyclic maximum of partial coordinate sums.
For the remaining six families, the same formula takes a two-branch form, with each branch involving the maximum of a finite collection of explicit linear candidates in the standard barred coordinates. The type dependence is confined to the index sets, coordinate constraints, boundary candidates, root-length scaling factors, and terminal-node behavior; see Theorem~\ref{thm:closed-form-energy-functions}.

\vskip 2mm

We also derive equivalent recursive evaluation formulas involving only addition, subtraction, positive parts, and absolute values; see
Theorem~\ref{thm:unified-max-free-recursive-closed-formulas}. 
These formulas evaluate the energy directly from the two input factors, without first computing the image of the input tensor under the combinatorial \(R\)-matrix. 
The main technical difficulty is to verify the defining local-energy recursion uniformly on the finite crystals. 
At the internal nodes, a dominance argument shows that each linear candidate that can change under a Kashiwara operator is controlled by one that remains unchanged. 
The affine and terminal nodes require separate boundary arguments accounting for the finite coordinate constraints, long- and short-root scaling factors, the parity condition in type \(C_n^{(1)}\), the binary coordinate in types \(B_n^{(1)}\) and \(D_{n+1}^{(2)}\), and the two terminal branches in type \(D_n^{(1)}\). 
Resolving these boundary phenomena within a common framework is the principal technical part of the proof.

\vskip 2mm

The closed energy functions have two applications. First, they give a
coordinate form of the KMN path character formula for every dominant weight of level \(l\). Subtracting the corresponding ground-state local energy from each adjacent contribution removes the additive ambiguity of the energy function. 
Substituting the closed energy functions into the KMN path character formula then expresses the character as an explicit positive sum over coordinate paths that eventually agree with the ground-state path. 
Both the classical weight exponents and the affine degree are given directly by the finite-crystal coordinates. The new point is that
the affine grading in the KMN formula is made explicit as a computable function of the coordinates of adjacent crystal elements; see Proposition~\ref{prop:closed-ch-formula}.

\vskip 2mm

After applying principal specialization to the KMN character formula, a telescoping argument rewrites the resulting path exponent as a weighted sum of an adjacent-pair statistic that depends only on the ordered pair and not on its position in the path.
Consequently, the complete exponent is determined by a finite adjacent-pair matrix. This position-independent local description makes the principally specialized finite-level path sum computationally effective; see Propositions~\ref{prop:adjacent-pair-degree-form}
and~\ref{prop:adjacent-energy-independent-of-k}.

\vskip 2mm

The character-comparison mechanism underlying the second application has direct crystal-theoretic precedents. 
Primc combined the KMN character formula with the principally specialized Weyl--Kac character formula to obtain a colored-partition identity for the basic \(A_2^{(1)}\)-module, with difference conditions determined by the energy of a perfect crystal \cite{Primc99}.
Certain colored-partition bases for basic representations of the
untwisted classical affine Lie algebras were constructed in
\cite{Primc00}, where the difference conditions arising from the
vertex-operator construction were shown to coincide with the energy
functions of the corresponding perfect crystals.
Both works concern basic representations, whereas our finite-level formulas make the path-sum side explicit for every dominant weight of positive level in all classical affine types.

\vskip 2mm
By comparing the principally specialized coordinate path formula with the corresponding Weyl--Kac character formula, we obtain a uniform family of level-\(l\) Rogers--Ramanujan-type identities for all classical affine types; see Theorem~\ref{thm:principal-specialized-path-product-identity} and
Corollary~\ref{cor:level-l-specialization}. 
Here ``Rogers--Ramanujan-type'' refers to an identity equating a positive path sum, graded by an explicit nonnegative integral statistic determined by adjacent pairs, with an explicit infinite product.
The essential new input is the finite-level path-sum side and its grading by position-independent local adjacent-pair contributions. For a representative low-rank case at level two in each family, we compute the complete adjacent-pair degree matrix and list the resulting identities.
Reformulating these path sums as colored-partition generating functions remains a separate problem.

\vskip 2mm
The degree matrices provide natural input for such partition-theoretic reformulations.
Dousse and Konan showed how perfect-crystal paths can be encoded by multi-grounded colored partitions, with suitably normalized local-energy data determining the adjacent-color difference conditions \cite{DK22}. 
Dousse, Hardiman, and Konan subsequently treated \(A_1^{(1)}\) perfect crystals at arbitrary level and derived a family of higher-level partition identities \cite{DHK}.
At level one, the required energy data are often recorded in a small finite table; for \(A_1^{(1)}\), the arbitrary-level case remains tractable because of its rank-one structure.
Our formulas provide analogous finite-level data in higher rank and may support extensions of grounded and multi-grounded
partition constructions to the families considered here.

\vskip 2mm

A further potential application concerns Young wall realizations of highest weight crystals.
Young walls model affine crystals as configurations of colored blocks, with Kashiwara operators acting through local modifications
of columns.
Level-one and higher-level constructions were developed in \cite{KangYW} and \cite{KL06YW,KL09YW}, respectively.
As the level of the dominant weight increases, the descriptions of admissible column configurations, ground-state walls, and compatibility conditions between adjacent columns become increasingly intricate.
In models related to perfect-crystal paths, local energy is a natural candidate for encoding the affine interaction between adjacent columns.
Our closed energy formulas may allow level-dependent local-energy matrices to be replaced by uniform coordinate rules and may help describe compatibility conditions between adjacent columns at higher levels.
Developing this connection, as well as treating unequal-level tensor products and more general rectangular Kirillov--Reshetikhin crystals, is left for future work.

\vskip 2mm

The paper is organized as follows.
Section~\ref{sec:perfect-crystals-quantum-affine-algebras} recalls perfect crystals, tensor products, and local energy functions;
Section~\ref{sec:perfect-crystals-classical-affine} reviews the coordinate realizations of the seven families;
Section~\ref{sec:closed-form-energy-functions} proves the closed and
recursive energy formulas and gives low-rank energy examples;
Section~\ref{sec:effective-character-algorithm} derives the coordinate path character formulas and adjacent-pair statistics; and
Section~\ref{sec:specialized-rogers-ramanujan-type-identities} establishes the Rogers--Ramanujan-type identities obtained by principal specialization, expresses their product sides explicitly as infinite products, and works out a representative low-rank example at level two for each of the seven families.

\section{Perfect crystals of quantum affine algebras}
\label{sec:perfect-crystals-quantum-affine-algebras}
Throughout this paper, let \(I=\{0,1,\dots,n\}\) be the index set of an affine Cartan datum \((A,P,\Pi,P^\vee,\Pi^\vee)\), where
\begin{enumerate}
	\item \(A=(a_{ij})_{i,j\in I}\) is a symmetrizable affine Cartan matrix,
	\item \(P\) is a free abelian group of rank \(n+2\),
	\item \(\Pi=\{\alpha_i\in P\mid i\in I\}\) is a \(\Q\)-linearly independent subset,
	\item \(P^\vee=\Hom_{\Z}(P,\Z)\),
	\item \(\Pi^\vee=\{h_i\in P^\vee\mid i\in I\}\),
\end{enumerate}
subject to the compatibility condition \(\langle h_i,\alpha_j\rangle=a_{ij}\) for \(i,j\in I\).

Set \(\mathfrak h=\Q\otimes_{\Z}P^\vee\). Choose \(d\in P^\vee\) so that \(\alpha_i(d)=\delta_{i0}\) for \(i\in I\). For each \(i\in I\), the fundamental weight \(\Lambda_i\in \mathfrak h^\ast\) is characterized by \(\Lambda_i(h_j)=\delta_{ij}\) and \(\Lambda_i(d)=0\) for \(i,j\in I\).

There exists a nondegenerate symmetric bilinear form \((\ ,\ )\) on
\(\mathfrak h^\ast\), normalized by
\[
r_i:=\frac{(\alpha_i,\alpha_i)}{2}\in\Z_{>0},
\qquad
\langle h_i,\lambda\rangle
=
\frac{2(\alpha_i,\lambda)}{(\alpha_i,\alpha_i)}
\quad (i\in I,\ \lambda\in P).
\]
We also set $Q_+:=\sum_{i=0}^{n}\Z_{\ge0}\alpha_i$.

We refer to \cite{Kac90, HK02} for the general background on affine Cartan data.

Let \(q\) be an indeterminate. 
For \(i\in I\), define
\[
q_i=q^{r_i},
\qquad
[m]_i=\frac{q_i^m-q_i^{-m}}{q_i-q_i^{-1}},
\qquad
[m]_i!=[m]_i[m-1]_i\cdots[1]_i,
\qquad
[0]_i!=1.
\]

\begin{definition}
	The \textit{quantum affine algebra} \(U_q(\mathfrak g)\) associated with
	\((A,P,\Pi,P^\vee,\Pi^\vee)\) is the \(\Q(q)\)-algebra with \(1\) generated by $e_i,\ f_i\ (i\in I),\ q^h\ (h\in P^\vee)$,
	subject to the relations
	\begin{equation*}
		\begin{aligned}
			& q^0=1,\qquad q^h q^{h'}=q^{h+h'} \qquad (h,h'\in P^\vee),\\
			& q^h e_j q^{-h}=q^{\langle h,\alpha_j\rangle}e_j,\\
			& q^h f_j q^{-h}=q^{-\langle h,\alpha_j\rangle}f_j
			\qquad (h\in P^\vee,\ j\in I),\\
			& e_i f_j-f_j e_i=\delta_{ij}\frac{K_i-K_i^{-1}}{q_i-q_i^{-1}},\\
			& \sum_{k=0}^{1-a_{ij}}(-1)^k e_i^{(1-a_{ij}-k)} e_j e_i^{(k)}=0
			\qquad (i\neq j),\\
			& \sum_{k=0}^{1-a_{ij}}(-1)^k f_i^{(1-a_{ij}-k)} f_j f_i^{(k)}=0
			\qquad (i\neq j),
		\end{aligned}
	\end{equation*}
where $K_i=q^{r_i h_i}$, $e_i^{(k)}=\frac{e_i^k}{[k]_i!}$ and $f_i^{(k)}=\frac{f_i^k}{[k]_i!}$.
\end{definition}

The subalgebra of \(U_q(\mathfrak g)\) generated by
\(e_i,\ f_i,\ K_i^{\pm1}\ (i\in I)\) is denoted by \(U_q'(\mathfrak g)\).
It is the derived quantum affine algebra, that is, the quantum affine
algebra without the degree operator \(q^d\).

Let \(c=c_0h_0+\cdots+c_n h_n\) be the canonical central element, and let \(\delta=d_0\alpha_0+\cdots+d_n\alpha_n\) be the null root, where the coefficients \(c_i,d_i\) are as in \cite{Kac90}. Then \(P=\bigoplus_{i\in I}\Z\Lambda_i\oplus \Z\bigl(\frac1{d_0}\delta\bigr)\) and \(P^\vee=\bigoplus_{i\in I}\Z h_i\oplus \Z d\). Note that \(d_0=1\) for all affine types except \(A_{2n}^{(2)}\), for which one has \(d_0=2\).

Let \(\cl:P\longrightarrow P/\Z\bigl(\frac1{d_0}\delta\bigr)\) be the natural projection, and write \(P_{\mathrm{cl}}:=\cl(P)=\bigoplus_{i\in I}\Z\,\cl(\Lambda_i)\).
When working in \(P_{\mathrm{cl}}\), we write \(\Lambda_i\) and \(\alpha_i\)
also for their images \(\cl(\Lambda_i)\) and \(\cl(\alpha_i)\), respectively,
unless confusion may arise.
 We also set
\[
P^+=\{\lambda\in P\mid \langle h_i,\lambda\rangle\ge 0 \text{ for all } i\in I\},
\qquad
P_{\mathrm{cl}}^+=\cl(P^+).
\]
For \(\lambda\in P\), its \textit{level} is defined by \(\langle c,\lambda\rangle\).  Since \(\langle c,\delta\rangle=0\), the pairing
\(\lambda\mapsto\langle c,\lambda\rangle\) descends to
\(P_{\mathrm{cl}}\). Thus the level of an element of
\(P_{\mathrm{cl}}\) is well defined.
For \(l\in\Z_{\geq0}\), set
\[
(P_{\mathrm{cl}}^+)_l
:=
\{\lambda\in P_{\mathrm{cl}}^+\mid
\langle c,\lambda\rangle=l\}.
\]

\begin{definition}
	A \textit{\(U_q(\mathfrak g)\)-crystal} (resp. a \textit{\(U_q'(\mathfrak g)\)-crystal}) is a set \(B\) equipped with maps
	\[
	\widetilde e_i,\widetilde f_i:B\to B\cup\{0\},
	\qquad
	\varepsilon_i,\varphi_i:B\to \Z\cup\{-\infty\},
	\]
	together with a weight map $\wt:B\to P$ $\bigl(\text{resp. } \wt:B\to P_{\mathrm{cl}}\bigr)$,
	such that the following conditions hold for all \(b\in B\) and \(i\in I\):
	\begin{enumerate}
		\item \(\varphi_i(b)=\varepsilon_i(b)+\langle h_i,\wt(b)\rangle\);
		\item if \(\widetilde e_i b\in B\), then \(\wt(\widetilde e_i b)=\wt(b)+\alpha_i\);
		\item if \(\widetilde f_i b\in B\), then \(\wt(\widetilde f_i b)=\wt(b)-\alpha_i\);
		\item if \(\widetilde e_i b\in B\), then \(\varepsilon_i(\widetilde e_i b)=\varepsilon_i(b)-1\) and \(\varphi_i(\widetilde e_i b)=\varphi_i(b)+1\);
		\item if \(\widetilde f_i b\in B\), then \(\varepsilon_i(\widetilde f_i b)=\varepsilon_i(b)+1\) and \(\varphi_i(\widetilde f_i b)=\varphi_i(b)-1\);
		\item for \(b,b'\in B\), one has \(\widetilde f_i b=b'\) if and only if \(b=\widetilde e_i b'\);
		\item if \(\varphi_i(b)=-\infty\), then \(\widetilde e_i b=\widetilde f_i b=0\).
	\end{enumerate}
\end{definition}

For a \(U_q'(\mathfrak g)\)-crystal \(B\) and \(b\in B\), \(\mu\in P_{\mathrm{cl}}\), set $B_\mu=\{b\in B\mid \wt(b)=\mu\}$. Define
\begin{equation}\label{eq:varepsilon-varphi}
\varepsilon(b)=\sum_{i\in I}\varepsilon_i(b)\,\cl(\Lambda_i),
\qquad
\varphi(b)=\sum_{i\in I}\varphi_i(b)\,\cl(\Lambda_i).
\end{equation}

When no confusion can arise, we simply write \(\Lambda_i\) for \(\cl(\Lambda_i)\), and similarly suppress the distinction between \(\wt\) and its classical version.

If \(B_1\) and \(B_2\) are crystals, then their tensor product \(B_1\otimes B_2=B_1\times B_2\) is endowed with the crystal structure
\begin{equation}\label{eq:tensor-rule}
	\begin{aligned}
		\wt(b_1\otimes b_2)
		&=\wt(b_1)+\wt(b_2),\\
		\varepsilon_i(b_1\otimes b_2)
		&=\max\bigl(\varepsilon_i(b_1),\,
		\varepsilon_i(b_2)-\langle h_i,\wt(b_1)\rangle\bigr),\\
		\varphi_i(b_1\otimes b_2)
		&=\max\bigl(\varphi_i(b_2),\,
		\varphi_i(b_1)+\langle h_i,\wt(b_2)\rangle\bigr),\\
		\widetilde e_i(b_1\otimes b_2)
		&=
		\begin{cases}
			\widetilde e_i b_1\otimes b_2,& \text{if }\varphi_i(b_1)\ge \varepsilon_i(b_2),\\
			b_1\otimes \widetilde e_i b_2,& \text{if }\varphi_i(b_1)< \varepsilon_i(b_2),
		\end{cases}\\
		\widetilde f_i(b_1\otimes b_2)
		&=
		\begin{cases}
			\widetilde f_i b_1\otimes b_2,& \text{if }\varphi_i(b_1)> \varepsilon_i(b_2),\\
			b_1\otimes \widetilde f_i b_2,& \text{if }\varphi_i(b_1)\le \varepsilon_i(b_2).
		\end{cases}
	\end{aligned}
\end{equation}

\begin{definition}\label{def:perfect-crystal}
	Let \(l\) be a positive integer. A finite \(U_q'(\mathfrak g)\)-crystal \(B\) is called a \textit{perfect crystal of level \(l\)} if the following conditions are satisfied:
	\begin{enumerate}
		\item \(B\) is the crystal of a finite-dimensional \(U_q'(\mathfrak g)\)-module;
		\item \(B\otimes B\) is connected;
		\item there exists \(\lambda\in P_{\mathrm{cl}}\) such that $\wt(B)\subset \lambda-\sum_{i\neq 0}\Z_{\ge 0}\alpha_i,
		\ \#B_\lambda=1;$
		\item for every \(b\in B\), one has \(\langle c,\varepsilon(b)\rangle\ge l\);
		\item for each \(\lambda\in(P_{\mathrm{cl}}^+)_l\), there exist unique elements \(b^\lambda,\ b_\lambda\in B\) such that \(\varepsilon(b^\lambda)=\lambda\) and \(\varphi(b_\lambda)=\lambda\).
	\end{enumerate}
\end{definition}

The elements \(b^\lambda\) and \(b_\lambda\) are called the \textit{minimal elements} of \(B\).

\vskip 2mm

\begin{definition}\label{def:energy}
	Let \(B\) be a perfect crystal. An \emph{energy function} on \(B\otimes B\) is a map
	\(H : B \otimes B \rightarrow \Z\) such that, for every \(i \in I\) and
	\(b_1, b_2 \in B\) with \(\widetilde{f}_i(b_1 \otimes b_2) \neq 0\), the following holds:
	\begin{equation}\label{eq:energy}
	\begin{aligned}
		H\bigl(\widetilde{f}_i(b_1 \otimes b_2)\bigr) =
		\begin{cases}
			H(b_1 \otimes b_2), & \text{if } i \neq 0, \\[2pt]
			H(b_1 \otimes b_2) - 1, & \text{if } i = 0 \text{ and } \varphi_0(b_1) > \varepsilon_0(b_2), \\[2pt]
			H(b_1 \otimes b_2) + 1, & \text{if } i = 0 \text{ and } \varphi_0(b_1) \leq \varepsilon_0(b_2).
		\end{cases}
	\end{aligned}
	\end{equation}
\end{definition}

\begin{lemma}
\label{lem:energy-uniqueness-general}
	Let \(B\) be a perfect crystal, and let
	\(H,H':B\otimes B\to \Z\) be two energy functions. Then there exists a constant
	\(C\in\Z\) such that $H'=H+C$
	on \(B\otimes B\).
\end{lemma}

\begin{proof}
	Set \(d:=H'-H\). We show that \(d\) is constant on \(B\otimes B\).
	
	Let \(x=b_1\otimes b_2\in B\otimes B\), and suppose that
	\(\widetilde f_i x\ne0\). If \(i\ne0\), then we have $H(\widetilde f_i x)=H(x)$ and $H'(\widetilde f_i x)=H'(x)$,
	and hence \(d(\widetilde f_i x)=d(x)\).
	
	If \(i=0\), then by Definition~\ref{def:energy}, both \(H\) and \(H'\)
	change by the same amount, namely by \(-1\) when
	\(\varphi_0(b_1)>\varepsilon_0(b_2)\), and by \(+1\) when
	\(\varphi_0(b_1)\le \varepsilon_0(b_2)\). Hence again $d(\widetilde f_0 x)=d(x)$.
	
 Since \(B\otimes B\) is connected, \(d\) is constant on
	\(B\otimes B\). Therefore \(d= C\) for some \(C\in\Z\), and so $H'=H+C$.
\end{proof}

\begin{corollary}\label{cor:energy-equality}
	Let \(B\) be a perfect crystal of level \(l\), let
	\(\lambda\in(P_{\mathrm{cl}}^+)_l\), and let \(H,H'\) be energy functions
	on \(B\otimes B\). If $H(b^\lambda\otimes b_\lambda)
	=
	H'(b^\lambda\otimes b_\lambda)$,
	then \(H=H'\) on \(B\otimes B\).
\end{corollary}

\begin{proof}
The assertion follows from Lemma~\ref{lem:energy-uniqueness-general}.
\end{proof}

\vskip 2mm

\section{The structure of perfect crystals in classical affine types}
\label{sec:perfect-crystals-classical-affine}

In this section, we recall the coordinate realizations of the level-\(l\)
perfect crystals in \cite[Sections~6--8]{KMOTU98} and \cite{KMN2}. We write $(a)_+:=\max\{a,0\}$. The rank parameter \(n\) is assumed to be in the admissible range for the
corresponding affine type:
\begin{align*}
&A_{n}^{(1)}:\ n\geq 1,\quad
A_{2n}^{(2)}:\ n\geq 1,\quad
D_{n+1}^{(2)}:\ n\geq 2,\quad
C_n^{(1)}:\ n\geq 2,\\
&A_{2n-1}^{(2)}:\ n\geq 3,\quad
D_n^{(1)}:\ n\geq 4, \quad
B_n^{(1)}:\ n\geq 3.
\end{align*}

Whenever a formula for a Kashiwara operator gives a vector outside the
specified underlying set, the result is understood to be \(0\). The raising operators are determined by $\et{i}b=b'
\ \Longleftrightarrow\ \ft{i}b'=b$.
In the formulas below, only the displayed coordinates are changed; all
undisplayed coordinates remain unchanged.

\subsection{Perfect crystal of type $A_n^{(1)}$}
\label{subsec:perfect-crystal-cyclic-type}
 Let \(I=\Z/(n+1)\Z\), and let \(l\geq1\).
The level-\(l\) perfect crystal of type \(A_n^{(1)}\) is
\begin{equation}\label{eq:An-perfect-set}
B=\{(x_i)_{i\in I}\in\Z_{\geq0}^{I}\mid \sum_{i\in I}x_i=l\}.
\end{equation}
All indices are understood modulo \(n+1\). For \(b=(x_i)_{i\in I}\), the
Kashiwara operators are
\begin{equation}\label{eq:An-perfect-f}
	\ft{i}b
	=
	\begin{cases}
		(x'_j)_{j\in I}, & \text{if }x_i>0,\\[1mm]
		0, & \text{if }x_i=0,
	\end{cases}
\end{equation}
where $x'_i=x_i-1$, $x'_{i+1}=x_{i+1}+1$ and $x'_j=x_j\quad (j\neq i,i+1)$.

The functions \(\vphi_i\) and \(\veps_i\) are
\begin{equation}\label{eq:An-perfect-eps-phi}
	\vphi_i(b)=x_i,
	\qquad
	\veps_i(b)=x_{i+1}
	\qquad (i\in I).
\end{equation}
The weight is
\begin{equation}\label{eq:An-perfect-weight}
	\wt(b)
	=
	\sum_{i\in I}(x_i-x_{i+1})\Lambda_i.
\end{equation}

As a \(U_q(A_n)\)-crystal, $B$ is the one-row crystal
\(B(l\overline{\Lambda}_1)\). In the coordinate order
\((x_0,x_1,\dots,x_n)\), the coordinate \(x_i\) records the multiplicity of
the letter \(i\) for \(1\leq i\leq n\), while \(x_0\) records the multiplicity
of the letter \(n+1\).

\subsection{Perfect crystal of types $A_{2n}^{(2)}$, $D_{n+1}^{(2)}$, $C_n^{(1)}$, $A_{2n-1}^{(2)}$, $D_n^{(1)}$, $B_n^{(1)}$}
\label{subsec:perfect-crystal-barred-coordinate-types}

We now describe the remaining six families uniformly. Their elements are
written in barred coordinates as
\[
b=(x_1,\dots,x_n\mid \ol x_n,\dots,\ol x_1),
\]
or, for the two families $D_{n+1}^{(2)}$ and $B_n^{(1)}$ with an additional zero-coordinate, as
\[
b=(x_1,\dots,x_n\mid x_0\mid \ol x_n,\dots,\ol x_1),
\qquad x_0\in\{0,1\}.
\]
Set $S(b):=\sum_{i=1}^{n}(x_i+\ol x_i)$.

For brevity, let \((x_i\mid \ol x_i)\) denote
\((x_1,\dots,x_n\mid \ol x_n,\dots,\ol x_1)\), and
\((x_i\mid x_0\mid \ol x_i)\) denote
\((x_1,\dots,x_n\mid x_0\mid \ol x_n,\dots,\ol x_1)\).

\subsubsection*{Underlying sets}

The underlying sets are as follows.

\begin{enumerate}[label=\textup{(\roman*)},leftmargin=8mm]
	\item Type \(A_{2n}^{(2)}\):
	\begin{equation}\label{eq:A2n-ad-set}
		B
		=
		\left\{
		b=(x_i\mid \ol x_i)\in\Z_{\geq0}^{2n}
		\ \middle|\
		S(b)\leq l
		\right\}.
	\end{equation}
	
	\item Type \(D_{n+1}^{(2)}\):
	\begin{equation}\label{eq:D2n-perfect-set}
		B
		=
		\left\{
		b=(x_i\mid x_0\mid \ol x_i)
		\ \middle|\
		x_0\in\{0,1\},\
		x_i,\ol x_i\in\Z_{\geq0},\
		x_0+S(b)\leq l
		\right\}.
	\end{equation}
	
	\item Type \(C_n^{(1)}\):
	\begin{equation}\label{eq:Cn-perfect-set}
		B
		=
		\left\{
		b=(x_i\mid \ol x_i)\in\Z_{\geq0}^{2n}
		\ \middle|\
		S(b)\leq2l,\quad S(b)\in2\Z
		\right\}.
	\end{equation}
	
	\item Type \(A_{2n-1}^{(2)}\):
	\begin{equation}\label{eq:A2n-1-perfect-set}
		B
		=
		\left\{
		b=(x_i\mid \ol x_i)\in\Z_{\geq0}^{2n}
		\ \middle|\
		S(b)=l
		\right\}.
	\end{equation}
	
	\item Type \(D_n^{(1)}\):
	\begin{equation}\label{eq:Dn-perfect-set}
		B
		=
		\left\{
		b=(x_i\mid \ol x_i)\in\Z_{\geq0}^{2n}
		\ \middle|\
		S(b)=l,\quad x_n=0\text{ or }\ol x_n=0
		\right\}.
	\end{equation}
	
	\item Type \(B_n^{(1)}\):
	\begin{equation}\label{eq:Bn-perfect-set}
		B
		=
		\left\{
		b=(x_i\mid x_0\mid \ol x_i)
		\ \middle|\
		x_0\in\{0,1\},\
		x_i,\ol x_i\in\Z_{\geq0},\
		x_0+S(b)=l
		\right\}.
	\end{equation}
\end{enumerate}

\subsubsection*{Common internal nodes}
For the five types $A_{2n}^{(2)},\
D_{n+1}^{(2)},\
C_n^{(1)},\
A_{2n-1}^{(2)},\
B_n^{(1)},$
the internal nodes are \(1\leq i\leq n-1\). For type \(D_n^{(1)}\), the
internal nodes are \(1\leq i\leq n-2\). In all these cases one has
\begin{equation}\label{eq:common-internal-f}
	\ft{i}b
	=
	\begin{cases}
		(x_i-1,\ x_{i+1}+1), & \text{if }x_{i+1}\geq \ol x_{i+1},\\[1mm]
		(\ol x_{i+1}-1,\ \ol x_i+1), & \text{if }x_{i+1}<\ol x_{i+1}.
	\end{cases}
\end{equation}
\begin{equation}\label{eq:common-internal-eps-phi}
	\vphi_i(b)=x_i+(\ol x_{i+1}-x_{i+1})_+,
	\qquad
	\veps_i(b)=\ol x_i+(x_{i+1}-\ol x_{i+1})_+.
\end{equation}

\subsubsection*{Type-dependent boundary data}
The following formulas give exactly the data which are not covered by the
common internal formulas in \eqref{eq:common-internal-f}--\eqref{eq:common-internal-eps-phi}.

\vskip 2mm

\paragraph{Type \texorpdfstring{\(A_{2n}^{(2)}\)}{A2n twisted}.}
\begin{subequations}\label{eq:A2n-ad-f-eps-phi}
	\begin{align}
		\ft{0}b
		&=
		\begin{cases}
			(x_1+1), & \text{if }x_1\geq \ol x_1,\\
			(\ol x_1-1), & \text{if }x_1<\ol x_1,
		\end{cases}
		\label{eq:A2n-ad-f0}\\
		\ft{n}b
		&=
		(x_n-1,\ \ol x_n+1)
		\label{eq:A2n-ad-fn}\\
		\vphi_0(b)
		&=
		l-S(b)+2(\ol x_1-x_1)_+,
		\quad
		\veps_0(b)=l-S(b)+2(x_1-\ol x_1)_+
		\label{eq:A2n-ad-eps-phi-0}\\
		\vphi_n(b)
		&=
		x_n,
		\quad
		\veps_n(b)=\ol x_n.
		\label{eq:A2n-ad-eps-phi-n}
	\end{align}
\end{subequations}

\vskip 1mm

\paragraph{Type \texorpdfstring{\(D_{n+1}^{(2)}\)}{D n plus 1 twisted}.}
\begin{subequations}\label{eq:D2n-perfect-f-eps-phi}
	\begin{align}
		\ft{0}b
		&=
		\begin{cases}
			(x_1+1), & \text{if }x_1\geq \ol x_1,\\
			(\ol x_1-1), & \text{if }x_1<\ol x_1,
		\end{cases}
		\label{eq:D2n-perfect-f0}\\
		\ft{n}b
		&=
		\begin{cases}
			(x_n-1,\ x_0+1), & \text{if }x_0=0,\\
			(x_0-1,\ \ol x_n+1), & \text{if }x_0=1,
		\end{cases}
		\label{eq:D2n-perfect-fn}\\
		\vphi_0(b)
		&=
		l-x_0-S(b)+2(\ol x_1-x_1)_+,
		\quad
		\veps_0(b)=l-x_0-S(b)+2(x_1-\ol x_1)_+
		\label{eq:D2n-perfect-eps-phi-0}\\
		\vphi_n(b)
		&=
		2x_n+x_0,
		\quad
		\veps_n(b)=2\ol x_n+x_0.
		\label{eq:D2n-perfect-eps-phi-n}
	\end{align}
\end{subequations}

\vskip 1mm

\paragraph{Type \texorpdfstring{\(C_n^{(1)}\)}{C n affine}.}
\begin{subequations}\label{eq:Cn-perfect-f-eps-phi}
	\begin{align}
		\ft{0}b
		&=
		\begin{cases}
			(x_1+2), & \text{if }x_1\geq \ol x_1,\\
			(x_1+1,\ \ol x_1-1), & \text{if }x_1=\ol x_1-1,\\
			(\ol x_1-2), & \text{if }x_1\leq \ol x_1-2,
		\end{cases}
		\label{eq:Cn-perfect-f0}\\
		\ft{n}b
		&=
		(x_n-1,\ \ol x_n+1)
		\label{eq:Cn-perfect-fn}\\
		\vphi_0(b)
		&=
		l-\frac12S(b)+(\ol x_1-x_1)_+,
		\quad
		\veps_0(b)=l-\frac12S(b)+(x_1-\ol x_1)_+
		\label{eq:Cn-perfect-eps-phi-0}\\
		\vphi_n(b)
		&=
		x_n,
		\quad
		\veps_n(b)=\ol x_n.
		\label{eq:Cn-perfect-eps-phi-n}
	\end{align}
\end{subequations}

\vskip 1mm

\paragraph{Type \texorpdfstring{\(A_{2n-1}^{(2)}\)}{A odd twisted}.}
\begin{subequations}\label{eq:A2n-1-perfect-f-eps-phi}
	\begin{align}
		\ft{0}b
		&=
		\begin{cases}
			(x_2+1,\ \ol x_1-1), & \text{if }x_2\geq \ol x_2,\\
			(x_1+1,\ \ol x_2-1), & \text{if }x_2<\ol x_2,
		\end{cases}
		\label{eq:A2n-1-perfect-f0}\\
		\ft{n}b
		&=
		(x_n-1,\ \ol x_n+1)
		\label{eq:A2n-1-perfect-fn}\\
		\vphi_0(b)
		&=
		\ol x_1+(\ol x_2-x_2)_+,
		\quad
		\veps_0(b)=x_1+(x_2-\ol x_2)_+
		\label{eq:A2n-1-perfect-eps-phi-0}\\
		\vphi_n(b)
		&=
		x_n,
		\quad
		\veps_n(b)=\ol x_n.
		\label{eq:A2n-1-perfect-eps-phi-n}
	\end{align}
\end{subequations}

\vskip 1mm

\paragraph{Type \texorpdfstring{\(D_n^{(1)}\)}{D n affine}.}
\begin{subequations}\label{eq:Dn-perfect-f-eps-phi}
	\begin{align}
		\ft{0}b
		&=
		\begin{cases}
			(x_2+1,\ \ol x_1-1), & \text{if }x_2\geq \ol x_2,\\
			(x_1+1,\ \ol x_2-1), & \text{if }x_2<\ol x_2,
		\end{cases}
		\label{eq:Dn-perfect-f0}\\
		\ft{n-1}b
		&=
		\begin{cases}
			(x_{n-1}-1,\ x_n+1), & \text{if }\ol x_n=0,\\
			(\ol x_n-1,\ \ol x_{n-1}+1), & \text{if }x_n=0,\ \ol x_n\geq1,
		\end{cases}
		\label{eq:Dn-perfect-fn-1}\\
		\ft{n}b
		&=
		\begin{cases}
			(x_n-1,\ \ol x_{n-1}+1), & \text{if }x_n\geq1,\ \ol x_n=0,\\
			(x_{n-1}-1,\ \ol x_n+1), & \text{if }x_n=0,
		\end{cases}
		\label{eq:Dn-perfect-fn}\\
		\vphi_0(b)
		&=
		\ol x_1+(\ol x_2-x_2)_+,
		\quad
		\veps_0(b)=x_1+(x_2-\ol x_2)_+
		\label{eq:Dn-perfect-eps-phi-0}\\
		\vphi_{n-1}(b)
		&=
		x_{n-1}+\ol x_n,
		\quad
		\veps_{n-1}(b)=\ol x_{n-1}+x_n
		\label{eq:Dn-perfect-eps-phi-n-1}\\
		\vphi_n(b)
		&=
		x_{n-1}+x_n,
		\quad
		\veps_n(b)=\ol x_{n-1}+\ol x_n.
		\label{eq:Dn-perfect-eps-phi-n}
	\end{align}
\end{subequations}

\vskip 1mm

\paragraph{Type \texorpdfstring{\(B_n^{(1)}\)}{B n affine}.}
\begin{subequations}\label{eq:Bn-perfect-f-eps-phi}
	\begin{align}
		\ft{0}b
		&=
		\begin{cases}
			(x_2+1,\ \ol x_1-1), & \text{if }x_2\geq \ol x_2,\\
			(x_1+1,\ \ol x_2-1), & \text{if }x_2<\ol x_2,
		\end{cases}
		\label{eq:Bn-perfect-f0}\\
		\ft{n}b
		&=
		\begin{cases}
			(x_n-1,\ x_0+1), & \text{if }x_0=0,\\
			(x_0-1,\ \ol x_n+1), & \text{if }x_0=1,
		\end{cases}
		\label{eq:Bn-perfect-fn}\\
		\vphi_0(b)
		&=
		\ol x_1+(\ol x_2-x_2)_+,
		\quad
		\veps_0(b)=x_1+(x_2-\ol x_2)_+
		\label{eq:Bn-perfect-eps-phi-0}\\
		\vphi_n(b)
		&=
		2x_n+x_0,
		\quad
		\veps_n(b)=2\ol x_n+x_0.
		\label{eq:Bn-perfect-eps-phi-n}
	\end{align}
\end{subequations}

\subsubsection*{Weights}

For each barred-coordinate type, the weight is given uniformly by
\begin{equation}\label{eq:barred-perfect-weight}
	\wt(b)
	=
	\sum_{i=0}^{n}\bigl(\vphi_i(b)-\veps_i(b)\bigr)\Lambda_i.
\end{equation}

\begin{lemma}\label{lem:minimal-elements-all-types}		
Let \(B\) be a level-$l$ perfect crystal and let $\lambda=\sum_{i\in I}m_i\Lambda_i\in(P_{\mathrm{cl}}^+)_l$. Then the minimal
	elements \(b_\lambda, b^\lambda\in B\) are given as follows.
	
	\begin{enumerate}
		\item For type \(A_n^{(1)}\), let \(\sum_{i\in I}m_i=l\). Then $b_\lambda=(m_i)_{i\in I},\ b^\lambda=(m_{i-1})_{i\in I},$
		where all indices are understood modulo \(n+1\).
		Equivalently, in the coordinate order \((x_0,x_1,\dots,x_n)\),
		\[
		b_\lambda=(m_0,m_1,\dots,m_n),
		\qquad
		b^\lambda=(m_n,m_0,m_1,\dots,m_{n-1}).
		\]
		
		\vskip 1mm
		
		\item For type \(A_{2n}^{(2)}\), let $m_0+2\sum_{i=1}^n m_i=l$.
		Then $b_\lambda=b^\lambda
		=
		(m_1,\dots,m_n\mid m_n,\dots,m_1)$.
		
		\vskip 1mm
		
		\item For type \(D_{n+1}^{(2)}\), let $m_0+2\sum_{i=1}^{n-1}m_i+m_n=l$.
		Let \(r\in\{0,1\}\) be determined by $r\equiv m_n\pmod 2$,
		and set $m=(m_n-r)/2$.
		Then 
		\[
		b_\lambda=b^\lambda
		=
		(m_1,\dots,m_{n-1},m\mid r
		\mid
		m,m_{n-1},\dots,m_1).
	    \]
		
		\vskip 1mm
		
		\item For type \(C_n^{(1)}\), let $m_0+\sum_{i=1}^n m_i=l$.
		Then $b_\lambda=b^\lambda
		=
		(m_1,\dots,m_n\mid m_n,\dots,m_1)$.
	
		\vskip 1mm
		
		\item For type \(A_{2n-1}^{(2)}\), let $m_0+m_1+2\sum_{i=2}^n m_i=l$.
		Then
		\[
		b_\lambda
		=
		(m_1,m_2,\dots,m_n\mid m_n,\dots,m_2,m_0),\ b^\lambda
		=
		(m_0,m_2,\dots,m_n\mid m_n,\dots,m_2,m_1).
		\]

\vskip 1mm

		\item For type \(D_n^{(1)}\), let $m_0+m_1+2\sum_{i=2}^{n-2}m_i+m_{n-1}+m_n=l$.
		Set $a=\min(m_{n-1},m_n)$, $u=(m_n-m_{n-1})_+$ and $v=(m_{n-1}-m_n)_+$.
		Then
		\[
		b_\lambda
		=
		(m_1,m_2,\dots,m_{n-2},a,u
		\mid
		v,a,m_{n-2},\dots,m_2,m_0),
		\]
		\[
		b^\lambda
		=
		(m_0,m_2,\dots,m_{n-2},a,v
		\mid
		u,a,m_{n-2},\dots,m_2,m_1).
		\]
		
		\vskip 1mm
		
		\item For type \(B_n^{(1)}\), let $m_0+m_1+2\sum_{i=2}^{n-1}m_i+m_n=l$.
		Let \(r\in\{0,1\}\) be determined by $r\equiv m_n\pmod 2$,
		and set $m=(m_n-r)/2$.
		Then
		\[
		b_\lambda
		=
		(m_1,m_2,\dots,m_{n-1},m\mid r
		\mid
		m,m_{n-1},\dots,m_2,m_0),
		\]
		\[
		b^\lambda
		=
		(m_0,m_2,\dots,m_{n-1},m \mid r
		\mid
		m,m_{n-1},\dots,m_2,m_1).
		\]
	\end{enumerate}
\end{lemma}

\begin{proof}
For the types $A_{2n-1}^{(2)},\ B_n^{(1)},\ D_n^{(1)},\
	D_{n+1}^{(2)},\ A_{2n}^{(2)}$,
	the minimal elements are proved in \cite[Sections~6.7--6.11]{KMN2}. It remains to verify the types
	\(A_n^{(1)}\) and \(C_n^{(1)}\).
	
	For type \(A_n^{(1)}\), let \(b=(x_i)_{i\in I}\in B\). By
	\eqref{eq:An-perfect-eps-phi}, we have $\varphi(b)=\sum_{i\in I}x_i\Lambda_i$ and $\varepsilon(b)=\sum_{i\in I}x_{i+1}\Lambda_i$.
Hence \(\varphi(b)=\lambda=\sum_{i\in I}m_i\Lambda_i\) if and only if
	\(x_i=m_i\) for all \(i\in I\), giving \(b_\lambda=(m_i)_{i\in I}\).
	Similarly, \(\varepsilon(b)=\lambda\) if and only if \(x_{i+1}=m_i\) for
	all \(i\in I\), equivalently \(x_i=m_{i-1}\) for all \(i\in I\), giving
	\(b^\lambda=(m_{i-1})_{i\in I}\). Both elements lie in \(B\) by
	\eqref{eq:An-perfect-set}, since \(\sum_{i\in I}m_i=l\), and uniqueness
	follows from the coordinate equations.
	
	For type \(C_n^{(1)}\), let
	\(b=(x_1,\dots,x_n\mid \ol x_n,\dots,\ol x_1)\in B\). By
	\eqref{eq:common-internal-eps-phi} and the type \(C_n^{(1)}\) boundary
	formulas in \eqref{eq:Cn-perfect-f-eps-phi}, together with
	\(S(b)=\sum_{i=1}^n(x_i+\ol x_i)\), one obtains
	\[
	\langle c,\varphi(b)\rangle
	=
	\langle c,\varepsilon(b)\rangle
	=
	l+\frac12\sum_{i=1}^n |x_i-\ol x_i|.
	\]
	Since \(\lambda\in(P_{\mathrm{cl}}^+)_l\), either
	\(\varphi(b)=\lambda\) or \(\varepsilon(b)=\lambda\) forces
	\(\sum_{i=1}^n |x_i-\ol x_i|=0\), hence \(x_i=\ol x_i\) for all \(i\).
	Thus \(b=(k_1,\dots,k_n\mid k_n,\dots,k_1)\). By \eqref{eq:varepsilon-varphi},  \eqref{eq:common-internal-eps-phi}, \eqref{eq:Cn-perfect-eps-phi-n}, we have
	\[
	\varphi(b)=\varepsilon(b)
	=
	\left(l-\sum_{i=1}^n k_i\right)\Lambda_0
	+
	\sum_{i=1}^n k_i\Lambda_i.
	\]
	Comparing with
	\(\lambda=m_0\Lambda_0+\sum_{i=1}^n m_i\Lambda_i\), where
	\(m_0+\sum_{i=1}^n m_i=l\), gives \(k_i=m_i\) for all \(1\le i\le n\).
	Therefore, we have $b_\lambda=b^\lambda=(m_1,\dots,m_n\mid m_n,\dots,m_1)$.
	This element belongs to \(B\) by \eqref{eq:Cn-perfect-set}, since
	\(S(b)=2\sum_{i=1}^n m_i=2(l-m_0)\le 2l\) and \(S(b)\in2\Z\).
	Uniqueness follows from the coordinate comparison.
\end{proof}

\vskip 2mm

\section{Closed formulas for the energy functions}
\label{sec:closed-form-energy-functions}

In this section, we give uniform closed formulas for the energy functions
on the level-\(l\) perfect crystals described in
Sections~\ref{subsec:perfect-crystal-cyclic-type}--\ref{subsec:perfect-crystal-barred-coordinate-types}.

\subsection{A unified max-linear formula for the energy functions}
Let \(X\) be one of the seven classical affine types, let \(B\) be the
corresponding level-\(l\) perfect crystal, and fix \(b_1,b_2\in B\).

If \(X=A_n^{(1)}\), write $b_1=(x_i)_{i\in I}$, $b_2=(y_i)_{i\in I}$, $I=\mathbb Z/(n+1)\mathbb Z$.
For the remaining six types, write $b_1=(x_1,\dots,x_n\mid\ol x_n,\dots,\ol x_1)$ and $b_2=(y_1,\dots,y_n\mid\ol y_n,\dots,\ol y_1)$,
with the additional \(0\)-coordinates \(x_0,y_0\) in types
\(D_{n+1}^{(2)}\) and \(B_n^{(1)}\).

Define the type-dependent index set
\begin{equation}\label{eq:uniform-index-set}
	J_X
	:=
	\begin{cases}
		\{0,1,\dots,n\},
		& X=A_n^{(1)},\\
		\{1,2,\dots,n\},
		& X=A_{2n}^{(2)},D_{n+1}^{(2)},C_n^{(1)},
		A_{2n-1}^{(2)},D_n^{(1)},B_n^{(1)}.
	\end{cases}
\end{equation}

Set \(A_0=B_0=0\), and for
\(1\leq r\leq n\), define
\begin{equation}\label{eq:uniform-AB}
	\begin{aligned}
	A_r
	&:=
	\sum_{j=1}^{r}(\ol x_j-\ol y_j)\ (\text{for the six barred-coordinate types}),\\
	B_r
	&:=
	\sum_{j=1}^{r}(y_j-x_j)\ (\text{for all seven types}).
	\end{aligned}
\end{equation}

For \(r\in J_X\), define \((\alpha_r^X,\beta_r^X)\) by
\begin{equation}\label{eq:uniform-alpha-beta}
	(\alpha_r^X,\beta_r^X)
	:=
	\begin{cases}
		\left(
		\displaystyle
		B_r+y_{r+1},\,
		\displaystyle
		B_r+y_{r+1}
		\right),\ 
		X=A_n^{(1)},0\leq r\leq n,
		\\[4mm]
		\bigl(
		A_{r-1}+(\ol x_r-x_r)_+,\,
		B_{r-1}+(y_r-\ol y_r)_+
		\bigr),\
		X=A_{2n}^{(2)},D_{n+1}^{(2)},C_n^{(1)},
		1\leq r\leq n,
		\\[4mm]
		(\ol x_1-x_1,\ y_1-\ol y_1),\
		X=A_{2n-1}^{(2)},D_n^{(1)},B_n^{(1)},
		r=1,
		\\[4mm]
		\bigl(
		A_{r-1}+(\ol x_r-x_r)_+,\,
		B_{r-1}+(y_r-\ol y_r)_+
		\bigr),\
		X=A_{2n-1}^{(2)},B_n^{(1)},
		2\leq r\leq n,
		\\[4mm]
		\bigl(
		A_{r-1}+(\ol x_r-x_r)_+,\,
		B_{r-1}+(y_r-\ol y_r)_+
		\bigr),\
		X=D_n^{(1)},
		2\leq r\leq n-1,
		\\[4mm]
		\bigl(
		A_{n-1}+x_n-\ol y_n,\,
		B_{n-1}+\ol y_n-x_n
		\bigr),\
		X=D_n^{(1)},
		r=n.
	\end{cases}
\end{equation}

In the \(A_n^{(1)}\) case, all the indices are understood modulo \(n+1\),
so that \(y_{n+1}=y_0\).

Now define
\begin{equation}\label{eq:uniform-UV-seven-types}
	\mathcal U_X
	:=
	\max_{r\in J_X}\alpha_r^X,
	\qquad
	\mathcal V_X
	:=
	\max_{r\in J_X}\beta_r^X.
\end{equation}

Define the type-dependent shift \(\Delta_X\) and coefficient
\(\Gamma_X\) by
\begin{equation}\label{eq:uniform-delta-c}
	(\Delta_X,\Gamma_X)
	:=
	\begin{cases}
		(0,1),
		& X=A_n^{(1)},A_{2n-1}^{(2)},D_n^{(1)},B_n^{(1)},\\
		(\Delta_A,2),
		& X=A_{2n}^{(2)},\\
		(\Delta_D,2),
		& X=D_{n+1}^{(2)},\\
		(\Delta_C,1),
		& X=C_n^{(1)}.
	\end{cases}
\end{equation}
Here
\begin{align*}
	\Delta_A
	&:=
	\sum_{i=1}^{n}(y_i+\ol y_i)
	-
	\sum_{i=1}^{n}(x_i+\ol x_i),
	\\
	\Delta_D
	&:=
	y_0+\sum_{i=1}^{n}(y_i+\ol y_i)
	-
	(
	x_0+\sum_{i=1}^{n}(x_i+\ol x_i)
	),
	\\
	\Delta_C
	&:=
	\frac12\sum_{i=1}^{n}(y_i+\ol y_i)
	-
	\frac12\sum_{i=1}^{n}(x_i+\ol x_i).
\end{align*}

We define the function $H_X:B\otimes B\longrightarrow \mathbb Z$ as follows.
\begin{equation}\label{eq:uniform-energy-seven-types}
	 H_X(b_1\otimes b_2)
	:=
	\max
	\left\{
	\Delta_X+\Gamma_X\mathcal U_X,\,
	-\Delta_X+\Gamma_X\mathcal V_X
	\right\}.
\end{equation}

The coordinate constraints imply that \(\Delta_X\in\mathbb Z\);
hence \(H_X\) is integer-valued.

\begin{lemma}[Internal-color invariance of $H_X$]
	\label{lem:common-local-internal-colors}
	Let \(X\) be one of the seven classical affine types, and let \(B\) be
	the corresponding level-\(l\) perfect crystal. Suppose that \(i\)
	satisfies
	\[
	\begin{cases}
		1\leq i\leq n,
		& X=A_n^{(1)},\\
		1\leq i\leq n-1,
		& X=A_{2n}^{(2)},D_{n+1}^{(2)},C_n^{(1)},
		A_{2n-1}^{(2)},B_n^{(1)},\\
		1\leq i\leq n-2,
		& X=D_n^{(1)}.
	\end{cases}
	\]
	If \(b_1,b_2\in B\) and
	\(\widetilde f_i(b_1\otimes b_2)\neq0\), then
	\[
	H_X\bigl(\widetilde f_i(b_1\otimes b_2)\bigr)
	=
	H_X(b_1\otimes b_2).
	\]
\end{lemma}

\begin{proof}
	First suppose that \(X=A_n^{(1)}\). It follows from
	\eqref{eq:uniform-alpha-beta}--\eqref{eq:uniform-energy-seven-types} that $H_X(b_1\otimes b_2)
	=
	\max_{0\leq r\leq n}\{B_r+y_{r+1}\}$.
By \eqref{eq:uniform-AB}, we have
	\begin{equation}\label{eq:Biminus1}
		(B_{i-1}+y_i)-(B_i+y_{i+1})
		=
		x_i-y_{i+1}.
	\end{equation}
	Here and throughout the \(A_n^{(1)}\) case, the indices are understood
	modulo \(n+1\).

	Equations \eqref{eq:tensor-rule} and
	\eqref{eq:An-perfect-eps-phi} give
	\[
	\widetilde f_i(b_1\otimes b_2)
	=
	\begin{cases}
		\widetilde f_i b_1\otimes b_2,
		& x_i>y_{i+1},\\
		b_1\otimes\widetilde f_i b_2,
		& x_i\leq y_{i+1}.
	\end{cases}
	\]
	
	If \(x_i>y_{i+1}\), then
	\eqref{eq:An-perfect-f} and \eqref{eq:uniform-AB} show that only the
	candidate \(B_i+y_{i+1}\) changes, and it changes to
	\(B_i+y_{i+1}+1\). Since the coordinates are integral,
	\eqref{eq:Biminus1} gives
	\(B_i+y_{i+1}+1\leq B_{i-1}+y_i\). Hence the maximum is unchanged.
	
	If \(x_i\leq y_{i+1}\), then only the candidate
	\(B_{i-1}+y_i\) changes, and it changes to
	\(B_{i-1}+y_i-1\). Equation \eqref{eq:Biminus1} gives
	\(B_{i-1}+y_i\leq B_i+y_{i+1}\), so the maximum is again unchanged.
	
	Now suppose that \(X\) is one of the six barred-coordinate types.
	By \eqref{eq:tensor-rule}, \(\widetilde f_i\) acts on exactly one
	tensor factor. By \eqref{eq:common-internal-f}, it preserves the total
	coordinate sum of that factor and leaves its \(0\)-coordinate unchanged.
	Thus \(\Delta_A,\Delta_D,\Delta_C\) are invariant. For the remaining
	three barred-coordinate types, \(\Delta_X=0\) by
	\eqref{eq:uniform-delta-c}. Hence \(\Delta_X\) is invariant.
	
For
\[
\begin{cases}
	1\leq r\leq n,
	& X=A_{2n}^{(2)},D_{n+1}^{(2)},C_n^{(1)},\\
	2\leq r\leq n,
	& X=A_{2n-1}^{(2)},B_n^{(1)},\\
	2\leq r\leq n-1,
	& X=D_n^{(1)},
\end{cases}
\]
equation \eqref{eq:uniform-alpha-beta} gives
\[
\alpha_r^X
=
\max\{A_{r-1},A_{r-1}+\bar x_r-x_r\},
\qquad
\beta_r^X
=
\max\{B_{r-1},B_{r-1}+y_r-\bar y_r\}.
\]	

	Set
	\begin{equation}\label{eq:eta-xi-zeta}
		\begin{aligned}
			&\eta:=A_{i-1}+\bar x_i-x_i,
			\qquad
			\xi:=A_i,
			\qquad
			\zeta:=A_i+\bar x_{i+1}-x_{i+1},\\
			&\bar\eta:=B_{i-1}+y_i-\bar y_i,
			\qquad
			\bar\xi:=B_i,
			\qquad
			\bar\zeta:=B_i+y_{i+1}-\bar y_{i+1}.
		\end{aligned}
	\end{equation}
	
	Equations \eqref{eq:common-internal-f} and
	\eqref{eq:uniform-AB} show that all linear branches outside the
	indices \(r=i,i+1\) are fixed. Thus the only linear branches that may change are the six quantities
	in \eqref{eq:eta-xi-zeta}.
	
	For \(X=A_{2n-1}^{(2)},D_n^{(1)},B_n^{(1)}\), if \(i=1\), the
	exceptional \(r=1\) candidates in
	\eqref{eq:uniform-alpha-beta} are \(\eta\) and \(\bar\eta\);
	if \(i>1\), those candidates are fixed. In type \(D_n^{(1)}\), the
	terminal candidates
	\(A_{n-1}+x_n-\bar y_n\) and
	\(B_{n-1}+\bar y_n-x_n\) are fixed, since \(i\leq n-2\), the changes
	at \(i,i+1\) cancel in \(A_{n-1},B_{n-1}\), and
	\(x_n,\bar y_n\) do not change. By \eqref{eq:uniform-UV-seven-types}, it is therefore enough to
	compare the maxima of the varying branches in
	\eqref{eq:eta-xi-zeta}.
	
	By \eqref{eq:uniform-AB},
	\begin{equation}\label{eq:local-branch-differences}
		\begin{aligned}
			\xi-\eta
			&=x_i-\bar y_i,
			&
			\zeta-\xi
			&=\bar x_{i+1}-x_{i+1},\\
			\bar\eta-\bar\xi
			&=x_i-\bar y_i,
			&
			\bar\xi-\bar\zeta
			&=\bar y_{i+1}-y_{i+1},\\
			\zeta-\eta
			&=x_i-\bar y_i+\bar x_{i+1}-x_{i+1},
			&
			\bar\eta-\bar\zeta
			&=x_i-\bar y_i+\bar y_{i+1}-y_{i+1}.
		\end{aligned}
	\end{equation}
	In Cases~\textup{(i)}--\textup{(iv)}, the quantities in
	\eqref{eq:eta-xi-zeta} denote their values before applying
	\(\widetilde f_i\).
	
	\medskip
	\noindent
	\emph{(i) \(\widetilde f_i\) acts on \(b_1\) and
		\(x_{i+1}\geq\bar x_{i+1}\).}
	
	By \eqref{eq:common-internal-eps-phi} and
	\eqref{eq:tensor-rule},
	\(x_i>\bar y_i+(y_{i+1}-\bar y_{i+1})_+\).
	Equations \eqref{eq:common-internal-f} and
	\eqref{eq:uniform-AB} give
	\[
	\begin{aligned}
		(\eta,\xi,\zeta)
		\longmapsto(\eta+1,\xi,\zeta-1),\quad
		(\bar\eta,\bar\xi,\bar\zeta)
		\longmapsto(\bar\eta,\bar\xi+1,\bar\zeta+1).
	\end{aligned}
	\]
	The tensor inequality and integrality give
	\(x_i-\bar y_i\geq1\) and
	\(x_i-\bar y_i\geq y_{i+1}-\bar y_{i+1}+1\).
	Together with \(x_{i+1}\geq\bar x_{i+1}\) and
	\eqref{eq:local-branch-differences}, this yields
	\(\eta+1\leq\xi\), \(\zeta\leq\xi\),
	\(\bar\xi+1\leq\bar\eta\), and
	\(\bar\zeta+1\leq\bar\eta\).
	Hence
	\[
	\max\{\eta,\xi,\zeta\}
	=
	\max\{\eta+1,\xi,\zeta-1\}
	=
	\xi,\quad
	\max\{\bar\eta,\bar\xi,\bar\zeta\}
	=
	\max\{\bar\eta,\bar\xi+1,\bar\zeta+1\}
	=
	\bar\eta.
	\]
	Thus \(\mathcal U_X\) and \(\mathcal V_X\) are unchanged.
	
	\medskip
	\noindent
	\emph{(ii) \(\widetilde f_i\) acts on \(b_1\) and
		\(x_{i+1}<\bar x_{i+1}\).}
	
	By \eqref{eq:common-internal-eps-phi} and
	\eqref{eq:tensor-rule},
	\(x_i+\bar x_{i+1}-x_{i+1}>\bar y_i\). Moreover, $(\eta,\xi,\zeta)
	\longmapsto
	(\eta+1,\xi+1,\zeta)$,
	while all \(\mathcal V_X\)-branches are fixed.
	By \eqref{eq:local-branch-differences} and integrality,
	\(\eta+1\leq\zeta\) and \(\xi+1\leq\zeta\). Therefore
	\[
	\max\{\eta,\xi,\zeta\}
	=
	\max\{\eta+1,\xi+1,\zeta\}
	=
	\zeta.
	\]
	Hence \(\mathcal U_X\) and \(\mathcal V_X\) are unchanged.
	
	\medskip
	\noindent
	\emph{(iii) \(\widetilde f_i\) acts on \(b_2\) and
		\(y_{i+1}\geq\bar y_{i+1}\).}
	
	Equations \eqref{eq:common-internal-eps-phi} and
	\eqref{eq:tensor-rule} give
	\(x_i\leq\varphi_i(b_1)\leq
	\bar y_i+y_{i+1}-\bar y_{i+1}\). Moreover, $(\bar\eta,\bar\xi,\bar\zeta)
	\longmapsto
	(\bar\eta-1,\bar\xi-1,\bar\zeta)$,
	while all \(\mathcal U_X\)-branches are fixed.
	By \eqref{eq:local-branch-differences},
	\(\bar\eta\leq\bar\zeta\) and
	\(\bar\xi\leq\bar\zeta\). Therefore 
	$$
	\max\{\bar\eta,\bar\xi,\bar\zeta\}
	=
	\max\{\bar\eta-1,\bar\xi-1,\bar\zeta\}
	=
	\bar\zeta.
	$$
	Hence \(\mathcal U_X\) and \(\mathcal V_X\) are unchanged.
	
	\medskip
	\noindent
	\emph{(iv) \(\widetilde f_i\) acts on \(b_2\) and
		\(y_{i+1}<\bar y_{i+1}\).}

Equations \eqref{eq:common-internal-eps-phi} and
\eqref{eq:tensor-rule} give $x_i+(\bar x_{i+1}-x_{i+1})_+
=
\varphi_i(b_1)
\leq\bar y_i$.

Equations \eqref{eq:common-internal-f} and
\eqref{eq:uniform-AB} give
\[
\begin{aligned}
	(\eta,\xi,\zeta)
	\longmapsto(\eta,\xi-1,\zeta-1),\quad
	(\bar\eta,\bar\xi,\bar\zeta)
	\longmapsto(\bar\eta-1,\bar\xi,\bar\zeta+1).
\end{aligned}
\]

In particular, \(x_i\leq\bar y_i\). By
\eqref{eq:local-branch-differences},
\[
\xi-\eta=x_i-\bar y_i\leq0,
\qquad
\bar\eta-\bar\xi=x_i-\bar y_i\leq0,
\]
and
\[
\begin{aligned}
	\zeta-\eta
	=
	x_i-\bar y_i+\bar x_{i+1}-x_{i+1}
	\leq
	x_i-\bar y_i+(\bar x_{i+1}-x_{i+1})_+
	=
	\varphi_i(b_1)-\bar y_i
	\leq0.
\end{aligned}
\]
Hence
\(\xi\leq\eta\), \(\zeta\leq\eta\), and
\(\bar\eta\leq\bar\xi\).	
	
 Since
	\(y_{i+1}<\bar y_{i+1}\) and the coordinates are integral,
	\(\bar\zeta+1\leq\bar\xi\). Therefore
	\[
	\max\{\eta,\xi,\zeta\}
	=
	\max\{\eta,\xi-1,\zeta-1\}
	=
	\eta,\quad
	\max\{\bar\eta,\bar\xi,\bar\zeta\}
	=
	\max\{\bar\eta-1,\bar\xi,\bar\zeta+1\}
	=
	\bar\xi.
	\]
	Hence \(\mathcal U_X\) and \(\mathcal V_X\) are unchanged.
	
	Thus \(\Delta_X,\mathcal U_X,\mathcal V_X\) are invariant under
	\(\widetilde f_i\). Since \(\Gamma_X\) depends only on \(X\),
	\eqref{eq:uniform-energy-seven-types} gives $H_X\bigl(\widetilde f_i(b_1\otimes b_2)\bigr)
	=
	H_X(b_1\otimes b_2)$.
\end{proof}

\vskip 2mm

\begin{lemma}[Boundary-color recursion of \(H_X\)]
	\label{lem:boundary-color-recursion-all-types}
	Let \(X\) be one of the seven classical affine types, and let \(B\) be
	the corresponding level-\(l\) perfect crystal. Suppose that
	\[
	i\in
	\begin{cases}
		\{0\},
		& X=A_n^{(1)},\\
		\{0,n\},
		& X=A_{2n}^{(2)},D_{n+1}^{(2)},C_n^{(1)},
		A_{2n-1}^{(2)},B_n^{(1)},\\
		\{0,n-1,n\},
		& X=D_n^{(1)}.
	\end{cases}
	\]
	If \(b_1,b_2\in B\) and $\widetilde f_i(b_1\otimes b_2)\neq0$,
	then the function \(H_X\) defined in
	\eqref{eq:uniform-energy-seven-types} 
	satisfies \eqref{eq:energy} at color $i$.
\end{lemma}

\begin{proof}
	If \(X=A_n^{(1)}\), write
	\(b_1=(x_i)_{i\in I}\) and \(b_2=(y_i)_{i\in I}\), where
	\(I=\mathbb Z/(n+1)\mathbb Z\). For each of the remaining six types,
	write
	\[
	b_1=(x_1,\ldots,x_n\mid\ol x_n,\ldots,\ol x_1),
	\qquad
	b_2=(y_1,\ldots,y_n\mid\ol y_n,\ldots,\ol y_1),
	\]
	including the additional \(0\)-coordinates \(x_0,y_0\) in types
	\(D_{n+1}^{(2)}\) and \(B_n^{(1)}\).
	
	For \(r\in J_X\), set
	\[
	P_r^X
	:=
	\Delta_X+\Gamma_X\alpha_r^X,
	\qquad
	Q_r^X
	:=
	-\Delta_X+\Gamma_X\beta_r^X.
	\]
	Primes denote the values obtained after applying the relevant Kashiwara
	operator. Thus
	\[
	(P_r^X)'
	=
	\Delta_X'+\Gamma_X(\alpha_r^X)',
	\qquad
	(Q_r^X)'
	=
	-\Delta_X'+\Gamma_X(\beta_r^X)'.
	\]
	By \eqref{eq:uniform-index-set},
	\eqref{eq:uniform-UV-seven-types}, and
	\eqref{eq:uniform-energy-seven-types}, we have $	H_X(b_1\otimes b_2)
	=
	\max_{r\in J_X}\{P_r^X,Q_r^X\}$.

\medskip
\noindent
\emph{The color \(0\) in type \(A_n^{(1)}\).}

Let \(X=A_n^{(1)}\). By \eqref{eq:uniform-alpha-beta} and
\eqref{eq:uniform-delta-c},
\(P_r^X=Q_r^X=B_r+y_{r+1}\) for \(0\leq r\leq n\).
By \eqref{eq:An-perfect-set} and \eqref{eq:uniform-AB},
\(B_n=x_0-y_0\); hence, since \(y_{n+1}=y_0\),
\(P_0^X=Q_0^X=y_1\) and \(P_n^X=Q_n^X=x_0\).

By \eqref{eq:An-perfect-eps-phi} and \eqref{eq:tensor-rule},
\[
\widetilde f_0(b_1\otimes b_2)
=
\begin{cases}
	\widetilde f_0b_1\otimes b_2,
	& x_0>y_1,\\
	b_1\otimes\widetilde f_0b_2,
	& x_0\leq y_1.
\end{cases}
\]

Suppose first that \(x_0>y_1\). By \eqref{eq:An-perfect-f} and
\eqref{eq:uniform-AB},
\[
(P_0^X)'=P_0^X,
\qquad
(P_r^X)'=P_r^X-1
\quad(1\leq r\leq n).
\]
The same relations hold for \(Q_r^X\). Since \(x_0,y_1\in\mathbb Z\),
\((P_0^X)'=y_1\leq x_0-1=(P_n^X)'\), while
\(P_0^X=y_1<P_n^X=x_0\). Therefore, by
\eqref{eq:uniform-energy-seven-types},
\[
\begin{aligned}
	H_X\bigl(\widetilde f_0(b_1\otimes b_2)\bigr)
	=
	\max_{1\leq r\leq n}\{P_r^X-1\}
	=
	\max_{1\leq r\leq n}\{P_r^X\}-1
	=
	H_X(b_1\otimes b_2)-1.
\end{aligned}
\]

Suppose next that \(x_0\leq y_1\). By
\eqref{eq:An-perfect-f} and \eqref{eq:uniform-AB},
\[
(P_r^X)'=P_r^X+1
\quad(0\leq r\leq n-1),
\qquad
(P_n^X)'=P_n^X.
\]
The same relations hold for \(Q_r^X\). Moreover,
\((P_n^X)'=x_0\leq y_1<P_0^X+1=(P_0^X)'\), while
\(P_n^X=x_0\leq P_0^X=y_1\). Hence, by
\eqref{eq:uniform-energy-seven-types},
\[
\begin{aligned}
	H_X\bigl(\widetilde f_0(b_1\otimes b_2)\bigr)
	=
	\max_{0\leq r\leq n-1}\{P_r^X+1\}
	=
	\max_{0\leq r\leq n-1}\{P_r^X\}+1
	=
	H_X(b_1\otimes b_2)+1.
\end{aligned}
\]
Thus \(H_X\) satisfies the color-\(0\) recursion
\eqref{eq:energy} in type \(A_n^{(1)}\).

\medskip
\noindent
\emph{The color \(0\) in types \(A_{2n}^{(2)}\) and
	\(D_{n+1}^{(2)}\).}

First let \(X=A_{2n}^{(2)}\). By
\eqref{eq:uniform-delta-c}, \(\Delta_X=\Delta_A\) and
\(\Gamma_X=2\). By \eqref{eq:A2n-ad-eps-phi-0},
\eqref{eq:tensor-rule}, and \eqref{eq:uniform-alpha-beta},
\(\widetilde f_0\) acts on the first tensor factor if and only if
\(\Delta_X+2\alpha_1^X>2\beta_1^X\), and on the second tensor factor
otherwise. By \eqref{eq:A2n-ad-f0},
\eqref{eq:uniform-alpha-beta}, and
\eqref{eq:uniform-delta-c}, the candidate changes are
\[
\resizebox{\textwidth}{!}{%
	\(
	\begin{array}{c|c|l|l}
		\hline
		\text{factor} & \text{condition} & \text{candidate changes}
		& \text{exceptional domination}
		\\ \hline
		1
		& x_1<\ol x_1
		& (P_r^X)'=P_r^X-1,\ (Q_r^X)'=Q_r^X-1\quad(1\leq r\leq n)
		& -
		\\ \hline
		1
		& x_1\geq\ol x_1
		& (P_r^X)'=P_r^X-1\ (1\leq r\leq n),\
		(Q_r^X)'=Q_r^X-1\ (2\leq r\leq n),\
		(Q_1^X)'=Q_1^X+1
		& Q_1^X\leq P_1^X,\quad (Q_1^X)'\leq(P_1^X)'
		\\ \hline
		2
		& y_1\geq\ol y_1
		& (P_r^X)'=P_r^X+1,\ (Q_r^X)'=Q_r^X+1\quad(1\leq r\leq n)
		& -
		\\ \hline
		2
		& y_1<\ol y_1
		& (P_r^X)'=P_r^X+1\ (2\leq r\leq n),\
		(Q_r^X)'=Q_r^X+1\ (1\leq r\leq n),\
		(P_1^X)'=P_1^X-1
		& P_1^X\leq Q_1^X,\quad (P_1^X)'<(Q_1^X)'
		\\ \hline
	\end{array}
	\)
}
\]

For the exceptional first-factor branch,
\(x_1\geq\ol x_1\) gives \(\alpha_1^X=0\), while the strict
tensor-factor condition and integrality give
\(\Delta_X\geq2\beta_1^X+1\). Hence
\(Q_1^X=-\Delta_X+2\beta_1^X\leq-1<1\leq P_1^X=\Delta_X\) and
\((Q_1^X)'=-\Delta_X+2\beta_1^X+1\leq0\leq
(P_1^X)'=\Delta_X-1\).

For the exceptional second-factor branch,
\(y_1<\ol y_1\) gives \(\beta_1^X=0\), while the weak
tensor-factor condition gives \(\Delta_X+2\alpha_1^X\leq0\).
Thus \(\Delta_X\leq0\),
\(P_1^X=\Delta_X+2\alpha_1^X\leq0\leq Q_1^X=-\Delta_X\), and
\((P_1^X)'=P_1^X-1<(Q_1^X)'=Q_1^X+1\). Therefore, \(H_X\) satisfies the color-\(0\) recursion.

Now let \(X=D_{n+1}^{(2)}\). By
\eqref{eq:uniform-delta-c}, \(\Delta_X=\Delta_D\) and
\(\Gamma_X=2\). By \eqref{eq:D2n-perfect-eps-phi-0},
\eqref{eq:tensor-rule}, \eqref{eq:D2n-perfect-f0}, and
\eqref{eq:uniform-alpha-beta}, the tensor-factor conditions and
coordinate branches are the same as for \(A_{2n}^{(2)}\).
Moreover, \(x_0,y_0\) are fixed by \(\widetilde f_0\), so the change
of \(\Delta_D\) equals the corresponding change of \(\Delta_A\).
Hence the same table and inequalities apply, and \(H_X\) satisfies the
color-\(0\) recursion in type \(D_{n+1}^{(2)}\).

\medskip
\noindent
\emph{The color \(0\) in type \(C_n^{(1)}\).}

Let \(X=C_n^{(1)}\). By \eqref{eq:uniform-delta-c},
\(\Delta_X=\Delta_C\) and \(\Gamma_X=1\). By
\eqref{eq:Cn-perfect-eps-phi-0}, \eqref{eq:tensor-rule}, and
\eqref{eq:uniform-alpha-beta}, \(\widetilde f_0\) acts on the first
tensor factor if and only if
\(\Delta_X+\alpha_1^X>\beta_1^X\), and on the second tensor factor
otherwise. By \eqref{eq:Cn-perfect-f0},
\eqref{eq:uniform-alpha-beta}, and
\eqref{eq:uniform-delta-c}, the candidate changes are
\[
\resizebox{\textwidth}{!}{%
	\(
	\begin{array}{c|c|l|l}
		\hline
		\text{factor} & \text{condition} & \text{candidate changes}
		& \text{exceptional domination}
		\\ \hline
		1
		& x_1\geq\ol x_1
		& (P_r^X)'=P_r^X-1\ (1\leq r\leq n),\
		(Q_r^X)'=Q_r^X-1\ (2\leq r\leq n),\
		(Q_1^X)'=Q_1^X+1
		& Q_1^X\leq P_1^X,\quad (Q_1^X)'\leq(P_1^X)'
		\\ \hline
		1
		& x_1=\ol x_1-1
		& (P_r^X)'=P_r^X-1\ (1\leq r\leq n),\
		(Q_r^X)'=Q_r^X-1\ (2\leq r\leq n),\
		(Q_1^X)'=Q_1^X
		& Q_1^X\leq P_1^X,\quad (Q_1^X)'\leq(P_1^X)'
		\\ \hline
		1
		& x_1\leq\ol x_1-2
		& (P_r^X)'=P_r^X-1,\ (Q_r^X)'=Q_r^X-1\quad(1\leq r\leq n)
		& -
		\\ \hline
		2
		& y_1\geq\ol y_1
		& (P_r^X)'=P_r^X+1,\ (Q_r^X)'=Q_r^X+1\quad(1\leq r\leq n)
		& -
		\\ \hline
		2
		& y_1=\ol y_1-1
		& (P_r^X)'=P_r^X+1\ (2\leq r\leq n),\
		(P_1^X)'=P_1^X,\
		(Q_r^X)'=Q_r^X+1\ (1\leq r\leq n)
		& P_1^X\leq Q_1^X,\quad (P_1^X)'<(Q_1^X)'
		\\ \hline
		2
		& y_1\leq\ol y_1-2
		& (P_r^X)'=P_r^X+1\ (2\leq r\leq n),\
		(P_1^X)'=P_1^X-1,\
		(Q_r^X)'=Q_r^X+1\ (1\leq r\leq n)
		& P_1^X\leq Q_1^X,\quad (P_1^X)'<(Q_1^X)'
		\\ \hline
	\end{array}
	\)
}
\]

By \eqref{eq:Cn-perfect-set},
\(S(b_1),S(b_2)\in2\mathbb Z\), and hence
\(\Delta_X=\frac12(S(b_2)-S(b_1))\in\mathbb Z\).
For \(x_1\geq\ol x_1\), one has \(\alpha_1^X=0\), and the strict
tensor-factor condition gives
\(\Delta_X\geq\beta_1^X+1\). Hence
\(Q_1^X=-\Delta_X+\beta_1^X\leq-1<1\leq
P_1^X=\Delta_X\) and
\((Q_1^X)'=-\Delta_X+\beta_1^X+1\leq0\leq
(P_1^X)'=\Delta_X-1\).

For \(x_1=\ol x_1-1\), one has \(\alpha_1^X=1\), and the strict
tensor-factor condition gives \(\Delta_X\geq\beta_1^X\). Thus
\(Q_1^X=-\Delta_X+\beta_1^X\leq0\leq
P_1^X=\Delta_X+1\) and
\((Q_1^X)'=Q_1^X\leq0\leq(P_1^X)'=\Delta_X\).

For \(y_1\leq\ol y_1-1\), one has \(\beta_1^X=0\), and the weak
tensor-factor condition gives
\(\Delta_X+\alpha_1^X\leq0\). Hence
\(\Delta_X\leq0\) and
\(P_1^X=\Delta_X+\alpha_1^X\leq0\leq
Q_1^X=-\Delta_X\). If \(y_1=\ol y_1-1\), then
\((P_1^X)'=P_1^X<(Q_1^X)'=Q_1^X+1\); if
\(y_1\leq\ol y_1-2\), then
\((P_1^X)'=P_1^X-1<(Q_1^X)'=Q_1^X+1\).

Therefore \(H_X\) satisfies the color-\(0\) recursion in type
\(C_n^{(1)}\).

\medskip
\noindent
\emph{The color \(0\) in types \(A_{2n-1}^{(2)}\),
	\(D_n^{(1)}\), and \(B_n^{(1)}\).}

Let \(X\in\{A_{2n-1}^{(2)},D_n^{(1)},B_n^{(1)}\}\). By
\eqref{eq:uniform-delta-c}, \(\Delta_X=0\) and \(\Gamma_X=1\), so
\(P_r^X=\alpha_r^X\) and \(Q_r^X=\beta_r^X\). By
\eqref{eq:A2n-1-perfect-eps-phi-0},
\eqref{eq:Dn-perfect-eps-phi-0},
\eqref{eq:Bn-perfect-eps-phi-0}, and
\eqref{eq:tensor-rule}, \(\widetilde f_0\) acts on the first tensor
factor if and only if
\[
\ol x_1+(\ol x_2-x_2)_+
>
y_1+(y_2-\ol y_2)_+,
\]
and on the second tensor factor otherwise. By
\eqref{eq:A2n-1-perfect-f0},
\eqref{eq:Dn-perfect-f0},
\eqref{eq:Bn-perfect-f0},
\eqref{eq:uniform-AB}, and
\eqref{eq:uniform-alpha-beta}, the candidate changes for
\(1\leq r\leq n-1\) are
\[
\resizebox{\textwidth}{!}{%
	\(
	\begin{array}{c|c|l|l}
		\hline
		\text{factor} & \text{condition} & \text{candidate changes}
		& \text{exceptional domination}
		\\ \hline
		1
		& x_2\geq\ol x_2
		& (P_r^X)'=P_r^X-1\ (1\leq r\leq n-1),\
		(Q_1^X)'=Q_1^X,\
		(Q_2^X)'=Q_2^X,\
		(Q_r^X)'=Q_r^X-1\ (3\leq r\leq n-1)
		& \begin{gathered}
			Q_2^X<P_1^X,\quad Q_1^X<P_2^X,\\
			(Q_2^X)'\leq(P_1^X)',\quad
			(Q_1^X)'\leq(P_2^X)'
		\end{gathered}
		\\ \hline
		1
		& x_2<\ol x_2
		& (P_r^X)'=P_r^X-1\ (1\leq r\leq n-1),\
		(Q_1^X)'=Q_1^X,\
		(Q_r^X)'=Q_r^X-1\ (2\leq r\leq n-1)
		& Q_1^X<P_2^X,\quad (Q_1^X)'\leq(P_2^X)'
		\\ \hline
		2
		& y_2\geq\ol y_2
		& (P_1^X)'=P_1^X,\
		(P_r^X)'=P_r^X+1\ (2\leq r\leq n-1),\
		(Q_r^X)'=Q_r^X+1\ (1\leq r\leq n-1)
		& P_1^X\leq Q_2^X,\quad (P_1^X)'<(Q_2^X)'
		\\ \hline
		2
		& y_2<\ol y_2
		& (P_1^X)'=P_1^X,\
		(P_2^X)'=P_2^X,\
		(P_r^X)'=P_r^X+1\ (3\leq r\leq n-1),\
		(Q_r^X)'=Q_r^X+1\ (1\leq r\leq n-1)
		& \begin{gathered}
			P_1^X\leq Q_2^X,\quad P_2^X\leq Q_1^X,\\
			(P_1^X)'<(Q_2^X)',\quad
			(P_2^X)'<(Q_1^X)'
		\end{gathered}
		\\ \hline
	\end{array}
	\)
}
\]

For \(X=A_{2n-1}^{(2)}\) and \(B_n^{(1)}\),
\eqref{eq:A2n-1-perfect-f0}, \eqref{eq:Bn-perfect-f0},
\eqref{eq:uniform-AB}, and \eqref{eq:uniform-alpha-beta}
show that the row-wise changes in the table also hold for \(r=n\). For \(X=D_n^{(1)}\),
\eqref{eq:Dn-perfect-f0}, \eqref{eq:uniform-AB}, and
\eqref{eq:uniform-alpha-beta} give
\[
\begin{aligned}
	(P_n^X)'&=P_n^X-1,
	&
	(Q_n^X)'&=Q_n^X-1
	&&\text{in the first-factor cases},\\
	(P_n^X)'&=P_n^X+1,
	&
	(Q_n^X)'&=Q_n^X+1
	&&\text{in the second-factor cases}.
\end{aligned}
\]

In the first-factor branch, if \(x_2\geq\ol x_2\), then the strict
tensor-factor condition gives
\(P_1^X-Q_2^X
=\ol x_1-y_1-(y_2-\ol y_2)_+>0\) and
\(P_2^X-Q_1^X=\ol x_1-y_1>0\).
If \(x_2<\ol x_2\), it gives
\(P_2^X-Q_1^X
=\ol x_1+\ol x_2-x_2-y_1>0\).
Since the candidates are integral, these inequalities give
\[
Q_2^X\leq P_1^X-1=(P_1^X)',
\qquad
Q_1^X\leq P_2^X-1=(P_2^X)'
\]
in the first row, and
\(Q_1^X\leq P_2^X-1=(P_2^X)'\) in the second row.

In the second-factor branch, if \(y_2\geq\ol y_2\), then the weak
tensor-factor condition gives \(P_1^X\leq Q_2^X\). If
\(y_2<\ol y_2\), it gives
\(P_1^X\leq Q_2^X\) and \(P_2^X\leq Q_1^X\).
The transformed inequalities in the last two rows follow because the
corresponding \(Q\)-candidates increase by \(1\), while the exceptional
\(P\)-candidates are unchanged.
Consequently, by \eqref{eq:uniform-energy-seven-types}, \(H_X\) satisfies the color-\(0\) recursion
\eqref{eq:energy} in types \(A_{2n-1}^{(2)}\),
\(D_n^{(1)}\), and \(B_n^{(1)}\).

\medskip
\noindent
\emph{The terminal color \(n\) in types \(A_{2n}^{(2)}\),
	\(C_n^{(1)}\), and \(A_{2n-1}^{(2)}\).}

Let \(X\in\{A_{2n}^{(2)},C_n^{(1)},A_{2n-1}^{(2)}\}\).
By \eqref{eq:A2n-ad-eps-phi-n},
\eqref{eq:Cn-perfect-eps-phi-n},
\eqref{eq:A2n-1-perfect-eps-phi-n}, and
\eqref{eq:tensor-rule}, \(\widetilde f_n\) acts on the first tensor
factor if and only if \(x_n>\ol y_n\), and on the second tensor factor
otherwise. By \eqref{eq:A2n-ad-fn},
\eqref{eq:Cn-perfect-fn},
\eqref{eq:A2n-1-perfect-fn},
\eqref{eq:uniform-AB},
\eqref{eq:uniform-alpha-beta}, and
\eqref{eq:uniform-delta-c}, all candidates with \(r<n\) are unchanged,
while the terminal candidates satisfy
\[
\resizebox{\textwidth}{!}{%
	\(
	\begin{array}{c|c|l|l}
		\hline
		\text{factor} & \text{condition} & \text{terminal candidate changes}
		& \text{exceptional domination}
		\\ \hline
		1
		& x_n>\ol y_n
		& \begin{gathered}
			(P_n^X)'=\Delta_X+\Gamma_X
			\bigl(A_{n-1}+(\ol x_n-x_n+2)_+\bigr),\\
			(Q_n^X)'=Q_n^X
		\end{gathered}
		& (P_n^X)'\neq P_n^X
		\ \Longrightarrow\
		P_n^X\leq(P_n^X)'\leq Q_n^X=(Q_n^X)'
		\\ \hline
		2
		& x_n\leq\ol y_n
		& \begin{gathered}
			(P_n^X)'=P_n^X,\\
			(Q_n^X)'=-\Delta_X+\Gamma_X
			\bigl(B_{n-1}+(y_n-\ol y_n-2)_+\bigr)
		\end{gathered}
		& (Q_n^X)'\neq Q_n^X
		\ \Longrightarrow\
		(Q_n^X)'\leq Q_n^X\leq P_n^X=(P_n^X)'
		\\ \hline
	\end{array}
	\)
}
\]

In the first-factor case, \((a+2)_+\geq a_+\) gives
\((P_n^X)'\geq P_n^X\). If \((P_n^X)'=P_n^X\), then all candidates are
unchanged. Otherwise, \(\ol x_n-x_n\geq-1\), and
\[
\begin{aligned}
	Q_n^X-(P_n^X)'
	&=
	\Bigl[-\Delta_X+\Gamma_X
	\bigl(B_{n-1}+(y_n-\ol y_n)_+\bigr)\Bigr]-
	\Bigl[\Delta_X+\Gamma_X
	\bigl(A_{n-1}+\ol x_n-x_n+2\bigr)\Bigr]\\
	&=
	\Gamma_X\Bigl(
	2(x_n-\ol y_n-1)+(\ol y_n-y_n)_+
	\Bigr)
	\geq0.
\end{aligned}
\]
The identity follows from
\eqref{eq:uniform-AB}, \eqref{eq:uniform-delta-c},
\eqref{eq:A2n-ad-set}, \eqref{eq:Cn-perfect-set}, and
\eqref{eq:A2n-1-perfect-set}; the inequality follows from
\(x_n>\ol y_n\) and integrality. Hence
\(P_n^X\leq(P_n^X)'\leq Q_n^X=(Q_n^X)'\).

In the second-factor case, \((a-2)_+\leq a_+\) gives
\((Q_n^X)'\leq Q_n^X\). If \((Q_n^X)'=Q_n^X\), then all candidates are
unchanged. Otherwise, \(y_n>\ol y_n\), and
\[
\begin{aligned}
	P_n^X-Q_n^X
	&=
	\Bigl[\Delta_X+\Gamma_X
	\bigl(A_{n-1}+(\ol x_n-x_n)_+\bigr)\Bigr]-
	\Bigl[-\Delta_X+\Gamma_X
	\bigl(B_{n-1}+y_n-\ol y_n\bigr)\Bigr]\\
	&=
	\Gamma_X\Bigl(
	2(\ol y_n-x_n)+(x_n-\ol x_n)_+
	\Bigr)
	\geq0.
\end{aligned}
\]
The identity follows from
\eqref{eq:uniform-AB}, \eqref{eq:uniform-delta-c},
\eqref{eq:A2n-ad-set}, \eqref{eq:Cn-perfect-set}, and
\eqref{eq:A2n-1-perfect-set}; the inequality follows from
\(x_n\leq\ol y_n\). Hence
\((Q_n^X)'\leq Q_n^X\leq P_n^X=(P_n^X)'\).

Since all candidates with \(r<n\) are unchanged,
\eqref{eq:uniform-energy-seven-types} gives $H_X\bigl(\widetilde f_n(b_1\otimes b_2)\bigr)
=
H_X(b_1\otimes b_2)$.

\medskip
\noindent
\emph{The terminal color \(n\) in types \(D_{n+1}^{(2)}\) and
	\(B_n^{(1)}\).}

Let \(X\in\{D_{n+1}^{(2)},B_n^{(1)}\}\). By
\eqref{eq:D2n-perfect-eps-phi-n},
\eqref{eq:Bn-perfect-eps-phi-n}, and
\eqref{eq:tensor-rule}, \(\widetilde f_n\) acts on the first tensor
factor if and only if
\(2x_n+x_0>2\ol y_n+y_0\), and on the second tensor factor otherwise.
By \eqref{eq:D2n-perfect-fn}, \eqref{eq:Bn-perfect-fn}, and
\eqref{eq:uniform-delta-c}, one has \(\Delta_X'=\Delta_X\): the quantity \(x_0+S(b_1)\), respectively \(y_0+S(b_2)\),
is preserved in
type \(D_{n+1}^{(2)}\), whereas \(\Delta_X=0\) in type
\(B_n^{(1)}\). Hence, by \eqref{eq:uniform-AB} and
\eqref{eq:uniform-alpha-beta}, all candidates with \(r<n\) are
unchanged.

Suppose first that \(2x_n+x_0>2\ol y_n+y_0\). Then
\[
(P_n^X)'
=
\Delta_X+\Gamma_X
\bigl(A_{n-1}+(\ol x_n-x_n+1)_+\bigr),
\qquad
(Q_n^X)'=Q_n^X,
\]
and \((P_n^X)'\geq P_n^X\). If \((P_n^X)'=P_n^X\), then all candidates
are unchanged. Otherwise, \(\ol x_n-x_n\geq0\), and
\[
\begin{aligned}
	Q_n^X-(P_n^X)'
	&=
	\Gamma_X\Bigl(
	2(x_n-\ol y_n)+x_0-y_0-1
	+(\ol y_n-y_n)_+
	\Bigr)
	\geq0.
\end{aligned}
\]
The identity follows from
\eqref{eq:uniform-AB}, \eqref{eq:uniform-delta-c},
\eqref{eq:D2n-perfect-set}, and \eqref{eq:Bn-perfect-set}.
The inequality follows from
\(2x_n+x_0>2\ol y_n+y_0\) and integrality. Thus
\(P_n^X\leq(P_n^X)'\leq Q_n^X=(Q_n^X)'\).

Suppose next that \(2x_n+x_0\leq2\ol y_n+y_0\). Then
\[
(P_n^X)'=P_n^X,
\qquad
(Q_n^X)'
=
-\Delta_X+\Gamma_X
\bigl(B_{n-1}+(y_n-\ol y_n-1)_+\bigr),
\]
and \((Q_n^X)'\leq Q_n^X\). If \((Q_n^X)'=Q_n^X\), then all candidates
are unchanged. Otherwise, \(y_n>\ol y_n\), and
\[
\begin{aligned}
	P_n^X-Q_n^X
	&=
	\Gamma_X\Bigl(
	2(\ol y_n-x_n)+y_0-x_0
	+(x_n-\ol x_n)_+
	\Bigr)
	\geq0.
\end{aligned}
\]
The identity follows from
\eqref{eq:uniform-AB}, \eqref{eq:uniform-delta-c},
\eqref{eq:D2n-perfect-set}, and \eqref{eq:Bn-perfect-set}.
The inequality follows from
\(2x_n+x_0\leq2\ol y_n+y_0\). Thus
\((Q_n^X)'\leq Q_n^X\leq P_n^X=(P_n^X)'\).

Since all candidates with \(r<n\) are unchanged,
\eqref{eq:uniform-energy-seven-types} gives $H_X\bigl(\widetilde f_n(b_1\otimes b_2)\bigr)
=
H_X(b_1\otimes b_2)$.

\medskip
\noindent
\emph{The terminal colors \(n-1\) and \(n\) in type \(D_n^{(1)}\).}

Let \(X=D_n^{(1)}\). By \eqref{eq:uniform-delta-c},
\(\Delta_X=0\) and \(\Gamma_X=1\), so
\(P_r^X=\alpha_r^X\) and \(Q_r^X=\beta_r^X\). By
\eqref{eq:Dn-perfect-set} and \eqref{eq:uniform-AB},
\[
B_{n-2}-A_{n-2}
=
x_{n-1}+\ol x_{n-1}+x_n+\ol x_n
-y_{n-1}-\ol y_{n-1}-y_n-\ol y_n.
\]
By \eqref{eq:Dn-perfect-eps-phi-n-1},
\eqref{eq:Dn-perfect-eps-phi-n}, and \eqref{eq:tensor-rule},
\(\widetilde f_{n-1}\) acts on the first tensor factor if and only if
\(x_{n-1}+\ol x_n>\ol y_{n-1}+y_n\), whereas
\(\widetilde f_n\) acts on the first tensor factor if and only if
\(x_{n-1}+x_n>\ol y_{n-1}+\ol y_n\); otherwise the corresponding
operator acts on the second tensor factor.

By \eqref{eq:Dn-perfect-fn-1}, \eqref{eq:Dn-perfect-fn},
\eqref{eq:uniform-AB}, and \eqref{eq:uniform-alpha-beta}, all
candidates with \(r\leq n-2\) are unchanged, while the terminal
candidates satisfy
\[
\resizebox{\textwidth}{!}{%
	\(
	\begin{array}{c|c|l|l}
		\hline
		\text{color and factor} & \text{condition}
		& \text{terminal candidate changes}
		& \text{exceptional domination}
		\\ \hline
		i=n-1,\ 1
		& x_{n-1}+\ol x_n>\ol y_{n-1}+y_n
		& \begin{gathered}
			(P_{n-1}^X)'=A_{n-2}
			+(\ol x_{n-1}-x_{n-1}+1)_+,\\
			(P_n^X)'=P_n^X+1,\quad
			(Q_{n-1}^X)'=Q_{n-1}^X,\quad
			(Q_n^X)'=Q_n^X
		\end{gathered}
		& \begin{gathered}
			(P_{n-1}^X)'\neq P_{n-1}^X
			\Longrightarrow
			P_{n-1}^X\leq(P_{n-1}^X)'\leq Q_n^X=(Q_n^X)',\\
			P_n^X\leq(P_n^X)'\leq Q_{n-1}^X=(Q_{n-1}^X)'
		\end{gathered}
		\\ \hline
		i=n-1,\ 2
		& x_{n-1}+\ol x_n\leq\ol y_{n-1}+y_n
		& \begin{gathered}
			(P_{n-1}^X)'=P_{n-1}^X,\quad
			(P_n^X)'=P_n^X,\\
			(Q_{n-1}^X)'=B_{n-2}
			+(y_{n-1}-\ol y_{n-1}-1)_+,\\
			(Q_n^X)'=Q_n^X-1
		\end{gathered}
		& \begin{gathered}
			(Q_{n-1}^X)'\neq Q_{n-1}^X
			\Longrightarrow
			(Q_{n-1}^X)'\leq Q_{n-1}^X\leq P_n^X=(P_n^X)',\\
			(Q_n^X)'\leq Q_n^X\leq P_{n-1}^X=(P_{n-1}^X)'
		\end{gathered}
		\\ \hline
		i=n,\ 1
		& x_{n-1}+x_n>\ol y_{n-1}+\ol y_n
		& \begin{gathered}
			(P_{n-1}^X)'=A_{n-2}
			+(\ol x_{n-1}-x_{n-1}+1)_+,\\
			(P_n^X)'=P_n^X,\quad
			(Q_{n-1}^X)'=Q_{n-1}^X,\quad
			(Q_n^X)'=Q_n^X+1
		\end{gathered}
		& \begin{gathered}
			(P_{n-1}^X)'\neq P_{n-1}^X
			\Longrightarrow
			P_{n-1}^X\leq(P_{n-1}^X)'\leq P_n^X=(P_n^X)',\\
			Q_n^X\leq(Q_n^X)'\leq Q_{n-1}^X=(Q_{n-1}^X)'
		\end{gathered}
		\\ \hline
		i=n,\ 2
		& x_{n-1}+x_n\leq\ol y_{n-1}+\ol y_n
		& \begin{gathered}
			(P_{n-1}^X)'=P_{n-1}^X,\quad
			(P_n^X)'=P_n^X-1,\\
			(Q_{n-1}^X)'=B_{n-2}
			+(y_{n-1}-\ol y_{n-1}-1)_+,\\
			(Q_n^X)'=Q_n^X
		\end{gathered}
		& \begin{gathered}
			(P_n^X)'\leq P_n^X\leq P_{n-1}^X=(P_{n-1}^X)',\\
			(Q_{n-1}^X)'\neq Q_{n-1}^X
			\Longrightarrow
			(Q_{n-1}^X)'\leq Q_{n-1}^X\leq Q_n^X=(Q_n^X)'
		\end{gathered}
		\\ \hline
	\end{array}
	\)
}
\]

For \(i=n-1\) in the first-factor case, if
\((P_{n-1}^X)'\neq P_{n-1}^X\), then
\(\ol x_{n-1}-x_{n-1}\geq0\), and
\[
\begin{aligned}
	Q_n^X-(P_{n-1}^X)'
	&=
	x_{n-1}+\ol x_n-\ol y_{n-1}-y_n-1
	\geq0,\\
	Q_{n-1}^X-(P_n^X)'
	&\geq
	x_{n-1}+\ol x_n-\ol y_{n-1}-y_n-1
	\geq0.
\end{aligned}
\]
In the second-factor case, if
\((Q_{n-1}^X)'\neq Q_{n-1}^X\), then
\(y_{n-1}>\ol y_{n-1}\), and
\[
\begin{aligned}
	P_n^X-Q_{n-1}^X
	&=
	\ol y_{n-1}+y_n-x_{n-1}-\ol x_n
	\geq0,\\
	P_{n-1}^X-Q_n^X
	&\geq
	\ol y_{n-1}+y_n-x_{n-1}-\ol x_n
	\geq0.
\end{aligned}
\]

For \(i=n\) in the first-factor case, if
\((P_{n-1}^X)'\neq P_{n-1}^X\), then
\(\ol x_{n-1}-x_{n-1}\geq0\), and
\[
\begin{aligned}
	P_n^X-(P_{n-1}^X)'
	&=
	x_{n-1}+x_n-\ol y_{n-1}-\ol y_n-1
	\geq0,\\
	Q_{n-1}^X-(Q_n^X)'
	&\geq
	x_{n-1}+x_n-\ol y_{n-1}-\ol y_n-1
	\geq0.
\end{aligned}
\]
In the second-factor case, if
\((Q_{n-1}^X)'\neq Q_{n-1}^X\), then
\(y_{n-1}>\ol y_{n-1}\), and
\[
\begin{aligned}
	P_{n-1}^X-P_n^X
	&\geq
	\ol y_{n-1}+\ol y_n-x_{n-1}-x_n
	\geq0,\\
	Q_n^X-Q_{n-1}^X
	&=
	\ol y_{n-1}+\ol y_n-x_{n-1}-x_n
	\geq0.
\end{aligned}
\]

Thus every possible new terminal candidate is dominated in the
first-factor cases, and every possibly lost terminal candidate is
already dominated in the second-factor cases. Since all candidates
with \(r\leq n-2\) are unchanged,
\eqref{eq:uniform-energy-seven-types} gives $H_X\bigl(\widetilde f_i(b_1\otimes b_2)\bigr)
=
H_X(b_1\otimes b_2)$ for $i=n-1,n$.

The boundary-color recursion has therefore been verified for all seven classical affine types. 
\end{proof}

\vskip 2mm

\begin{theorem}\label{thm:closed-form-energy-functions}
	Let \(X\) be one of the seven classical affine types, and let \(B\)
	be the corresponding level-\(l\) perfect crystal. Then the function
	\(H_X:B\otimes B\to\mathbb Z\) defined in
	\eqref{eq:uniform-energy-seven-types} is an energy function on
	\(B\otimes B\).
\end{theorem}

\begin{proof}
	By Lemma~\ref{lem:common-local-internal-colors},
	\(H_X\) satisfies \eqref{eq:energy} at every internal color, and by
	Lemma~\ref{lem:boundary-color-recursion-all-types}, it satisfies
	\eqref{eq:energy} at every boundary color. These colors exhaust the
	index set \(I\) for each of the seven types. Hence \(H_X\) is an
	energy function by Definition~\ref{def:energy}.
\end{proof}

\subsection{Recursive formulas for the energy functions}

We keep the rank assumptions and the coordinate conventions of
Section~\ref{sec:perfect-crystals-classical-affine} and
Section~\ref{sec:closed-form-energy-functions}. Thus
\(A_r,B_r\), \(\alpha_r^X,\beta_r^X\), \(\Delta_X,\Gamma_X\),
\(\mathcal U_X,\mathcal V_X\), and \(H_X\) are as defined in
\eqref{eq:uniform-AB}--\eqref{eq:uniform-energy-seven-types}.

By Theorem~\ref{thm:closed-form-energy-functions}, \(H_X\) is the
corresponding energy function. We now rewrite the same closed formulas
without using the maximum symbol.

For a nonempty finite list \(z_1,\dots,z_N\), define
\begin{equation}\label{eq:M_r}
	M_1:=z_1,
	\qquad
	M_r:=M_{r-1}+(z_r-M_{r-1})_+
	\qquad (2\leq r\leq N),
\end{equation}
and set
\begin{equation}\label{eq:mathfrak_M}
	\mathfrak M(z_1,\dots,z_N):=M_N.
\end{equation}

For type \(A_n^{(1)}\), set
\begin{equation}\label{eq:T_j}
	T_0:=y_1,
	\qquad
	T_j:=T_{j-1}+y_{j+1}-x_j
	\qquad (1\leq j\leq n),
	\qquad
	y_{n+1}=y_0.
\end{equation}

For the six barred-coordinate types, define
\begin{equation}\label{eq:U_X_V_X_rec}
	\mathcal U_X^{\mathrm{rec}}
	:=
	\mathfrak M(\alpha_1^X,\dots,\alpha_n^X),
	\qquad
	\mathcal V_X^{\mathrm{rec}}
	:=
	\mathfrak M(\beta_1^X,\dots,\beta_n^X),
\end{equation}
where \(\alpha_r^X,\beta_r^X\) are those of
\eqref{eq:uniform-alpha-beta}.

\begin{theorem}[Recursive closed formulas]
	\label{thm:unified-max-free-recursive-closed-formulas}
Let \(X\) be one of the seven classical affine types, let \(B\) be
the corresponding level-\(l\) perfect crystal, and let
\(b_1,b_2\in B\). Then the following formulas hold.
	
	\begin{enumerate}[label=\textup{(\roman*)},leftmargin=7mm]
		\item In type \(A_n^{(1)}\),
		\[
		H_X(b_1\otimes b_2)
		=
		\mathfrak M(T_0,T_1,\dots,T_n).
		\]
		
		\item In types
		\(A_{2n}^{(2)},D_{n+1}^{(2)},C_n^{(1)}\),
		\[
		H_X(b_1\otimes b_2)
		=
		\frac{
			\Gamma_X
			\bigl(
			\mathcal U_X^{\mathrm{rec}}
			+
			\mathcal V_X^{\mathrm{rec}}
			\bigr)
			+
			\left|
			2\Delta_X
			+
			\Gamma_X
			\bigl(
			\mathcal U_X^{\mathrm{rec}}
			-
			\mathcal V_X^{\mathrm{rec}}
			\bigr)
			\right|
		}{2}.
		\]
		
		\item In types
		\(A_{2n-1}^{(2)},D_n^{(1)},B_n^{(1)}\),
		\[
		H_X(b_1\otimes b_2)
		=
		\frac{
			\mathcal U_X^{\mathrm{rec}}
			+
			\mathcal V_X^{\mathrm{rec}}
			+
			\left|
			\mathcal U_X^{\mathrm{rec}}
			-
			\mathcal V_X^{\mathrm{rec}}
			\right|
		}{2}.
		\]
	\end{enumerate}
\end{theorem}

\begin{proof}
	For real numbers \(p,q\), one has
	\(p+(q-p)_+=\max\{p,q\}\). Hence
	\(M_r=\max\{M_{r-1},z_r\}\) for $2\leq r\leq N$, and induction on \(r\) gives $M_r=\max\{z_1,\dots,z_r\}$ $(1\leq r\leq N)$.
	Therefore, we have $\mathfrak M(z_1,\dots,z_N)=\max\{z_1,\dots,z_N\}$.

	For type \(A_n^{(1)}\), \eqref{eq:T_j},
	\eqref{eq:uniform-AB}, and \eqref{eq:uniform-alpha-beta} give
	\(T_j=B_j+y_{j+1}=\alpha_j^X=\beta_j^X\) for
	\(0\leq j\leq n\). Thus, by
	\eqref{eq:uniform-UV-seven-types},
	\eqref{eq:uniform-delta-c}, and
	\eqref{eq:uniform-energy-seven-types},
	\[
	H_X(b_1\otimes b_2)
	=
	\mathfrak M(T_0,T_1,\dots,T_n),
	\]
	which proves \textup{(i)}.
	
	For the six barred-coordinate types,
	\eqref{eq:uniform-index-set}, \eqref{eq:M_r},
	\eqref{eq:mathfrak_M}, \eqref{eq:U_X_V_X_rec}, and
	\eqref{eq:uniform-UV-seven-types} give $\mathcal U_X^{\mathrm{rec}}=\mathcal U_X$ and $\mathcal V_X^{\mathrm{rec}}=\mathcal V_X$. Using \eqref{eq:uniform-delta-c} and applying the identity
\(\max\{u,v\}=(u+v+|u-v|)/2\) to
\eqref{eq:uniform-energy-seven-types}, we obtain
\textup{(ii)} and \textup{(iii)}.
\end{proof}

\subsection{Low-rank examples of energy functions}
\begin{example}[Energy function in type \(A_1^{(1)}\)]
	\label{ex:A1affine-energy-matrix}
	Let \(X=A_1^{(1)}\), corresponding to \(n=1\) in type
	\(A_n^{(1)}\). By \eqref{eq:An-perfect-set}, the level-\(l\)
	perfect crystal is
	\[
	B
	=
	\left\{
	(x_0,x_1)\in\mathbb Z_{\geq0}^2
	\ \middle|\
	x_0+x_1=l
	\right\},
	\]
	so \(|B|=l+1\). Let
	\(b_1=(x_0,x_1), b_2=(y_0,y_1)\in B\).
	By \eqref{eq:T_j},
	\(T_0=y_1\) and
	\[
	T_1
	=
	T_0+y_2-x_1
	=
	y_1+y_0-x_1
	=
	x_0,
	\]
	where \(y_2=y_0\) and
	\(x_0+x_1=y_0+y_1=l\). Hence, by \eqref{eq:M_r}-- \eqref{eq:mathfrak_M} and Theorem \ref{thm:unified-max-free-recursive-closed-formulas}, we have
	\[
	H_X(b_1\otimes b_2)
	=\mathfrak M(T_0,T_1)
	=
	y_1+(x_0-y_1)_+.
	\]
\end{example}

\begin{remark}
	Our formula agrees with the energy matrix in \cite[Section~3]{DHK}, up to
	the normalization convention. In \cite[Section~3]{DHK}, the level is denoted
	by \(n\), and the crystal elements \(b_i\), \(0\le i\le n\), satisfy
	\[
	\wt(b_i)=(2i-n)\Lambda_0+(n-2i)\Lambda_1.
	\]
	In our notation, with \(l=n\), the element \(b(a)=(l-a,a)\) has weight
	\(\wt(b(a))=(l-2a)\Lambda_0+(2a-l)\Lambda_1\). Hence \(b_i\) in
	\cite[Section~3]{DHK} corresponds to \(b(n-i)\) in our notation. Therefore
	our  energy gives
	\[
	 H_X\bigl(b(n-i)\otimes b(n-j)\bigr)
	=
	(n-j)+\bigl(i-(n-j)\bigr)_+
	=
	\max\{i,n-j\},
	\]
	which is exactly the formula \(H_n(b_i\otimes b_j)=\max\{i,n-j\}\) in
	\cite[Section~3]{DHK}. 
\end{remark}

\begin{example}[Energy function in type \(A_2^{(2)}\)]
	\label{ex:A2twisted-energy-matrix}
	Let \(X=A_2^{(2)}\), corresponding to \(n=1\) in type
	\(A_{2n}^{(2)}\). By \eqref{eq:A2n-ad-set}, the level-\(l\)
	perfect crystal is $B
	=
	\left\{
	(x\mid\ol x)
	\mid
	x,\ol x\in\mathbb Z_{\geq0},\
	x+\ol x\leq l
	\right\}$, 
	so \(|B|=(l+1)(l+2)/2\). Let
	\(b_1=(x\mid\ol x), b_2=(y\mid\ol y)\in B\). By \eqref{eq:uniform-delta-c},
	\eqref{eq:uniform-alpha-beta}, and
	\eqref{eq:U_X_V_X_rec},
	\[
	\Delta_X
	=
	y+\ol y-x-\ol x,
	\qquad
	\Gamma_X=2,
	\]
	and $\mathcal U_X^{\mathrm{rec}}
	=
	(\ol x-x)_+,
	\
	\mathcal V_X^{\mathrm{rec}}
	=
	(y-\ol y)_+$.
	Therefore, by
	Theorem~\ref{thm:unified-max-free-recursive-closed-formulas},
	the energy matrix indexed by \(B\) has entries
	\begin{equation}
		\label{eq:A2twisted-energy-coordinate}
		\begin{aligned}
			H_X(b_1\otimes b_2)
			=
			(\ol x-x)_+
			+
			(y-\ol y)_++
			\left|
			y+\ol y-x-\ol x
			+
			(\ol x-x)_+
			-
			(y-\ol y)_+
			\right|.
		\end{aligned}
	\end{equation}
\end{example}

\begin{example}[Energy function in type \(D_3^{(2)}\)]
	\label{ex:D3twisted-energy-matrix}
	Let \(X=D_3^{(2)}\), corresponding to \(n=2\) in type
	\(D_{n+1}^{(2)}\). By \eqref{eq:D2n-perfect-set}, the level-\(l\)
	perfect crystal is
	\[
	\begin{aligned}
		B
		=
		\bigl\{
		(x_1,x_2\mid x_0\mid\ol x_2,\ol x_1)
		\ \big|\
		x_0\in\{0,1\},\
		x_1,x_2,\ol x_1,\ol x_2\in\mathbb Z_{\geq0},\
		x_0+x_1+x_2+\ol x_1+\ol x_2\leq l
		\bigr\}.
	\end{aligned}
	\]
	Hence $|B|
	=
	\binom{l+4}{4}
	+
	\binom{l+3}{4}$.
	Let $b_1=(x_1,x_2\mid x_0\mid\ol x_2,\ol x_1), b_2=(y_1,y_2\mid y_0\mid\ol y_2,\ol y_1)\in B$. By \eqref{eq:uniform-delta-c},
	\eqref{eq:uniform-AB}, and \eqref{eq:uniform-alpha-beta},
	\[
	\Delta_X
	=
	y_0+y_1+y_2+\ol y_1+\ol y_2
	-
	x_0-x_1-x_2-\ol x_1-\ol x_2,
	\qquad
	\Gamma_X=2,
	\]
	\[
	\begin{aligned}
		\alpha_1^X
		=
		(\ol x_1-x_1)_+,\
		\beta_1^X
		=
		(y_1-\ol y_1)_+,\
		\alpha_2^X
		=
		\ol x_1-\ol y_1+(\ol x_2-x_2)_+,\
		\beta_2^X
		=
		y_1-x_1+(y_2-\ol y_2)_+.
	\end{aligned}
	\]
	By \eqref{eq:U_X_V_X_rec}, we have $\mathcal U_X^{\mathrm{rec}}
	=
	\alpha_1^X+(\alpha_2^X-\alpha_1^X)_+$, $\mathcal V_X^{\mathrm{rec}}
	=
	\beta_1^X+(\beta_2^X-\beta_1^X)_+$.
	Since \(\Gamma_X=2\), Theorem~
	\ref{thm:unified-max-free-recursive-closed-formulas} gives
	\begin{align}
		H_X(b_1\otimes b_2)
		&=
		(\ol x_1-x_1)_+
		+
		\left(
		\ol x_1-\ol y_1+(\ol x_2-x_2)_+
		-
		(\ol x_1-x_1)_+
		\right)_+
		\notag\\
		&\quad
		+
		(y_1-\ol y_1)_+
		+
		\left(
		y_1-x_1+(y_2-\ol y_2)_+
		-
		(y_1-\ol y_1)_+
		\right)_+
		\notag\\
		&\quad
		+
		|
		y_0+y_1+y_2+\ol y_1+\ol y_2
		-
		x_0-x_1-x_2-\ol x_1-\ol x_2
		\notag\\
		&\qquad
		+
		(\ol x_1-x_1)_+
		+
		\left(
		\ol x_1-\ol y_1+(\ol x_2-x_2)_+
		-
		(\ol x_1-x_1)_+
		\right)_+
		\notag\\
		&\qquad
		-
		(y_1-\ol y_1)_+
		-
		\left(
		y_1-x_1+(y_2-\ol y_2)_+
		-
		(y_1-\ol y_1)_+
		\right)_+
		|.
		\label{eq:D3twisted-energy-expanded}
	\end{align}
\end{example}

\begin{example}[Energy function in type \(C_2^{(1)}\)]
	\label{ex:C2affine-energy-matrix}
	Let \(X=C_2^{(1)}\), corresponding to \(n=2\) in type
	\(C_n^{(1)}\). By \eqref{eq:Cn-perfect-set}, the level-\(l\)
	perfect crystal is
	\[
	\begin{aligned}
		B
		=
		\bigl\{
		(x_1,x_2\mid\ol x_2,\ol x_1)
		\ \big|\
		x_1,x_2,\ol x_1,\ol x_2\in\mathbb Z_{\geq0},\
		x_1+x_2+\ol x_1+\ol x_2\leq2l,\
		x_1+x_2+\ol x_1+\ol x_2\in2\mathbb Z
		\bigr\}.
	\end{aligned}
	\]
	Hence $|B|
	=
	\sum_{r=0}^{l}\binom{2r+3}{3}$.
	Let $b_1=(x_1,x_2\mid\ol x_2,\ol x_1), b_2=(y_1,y_2\mid\ol y_2,\ol y_1)\in B$. By \eqref{eq:uniform-delta-c},
	\eqref{eq:uniform-AB}, and \eqref{eq:uniform-alpha-beta},
	\[
	\Delta_X
	=
	\frac12
	\bigl(
	y_1+y_2+\ol y_1+\ol y_2
	-
	x_1-x_2-\ol x_1-\ol x_2
	\bigr),
	\qquad
	\Gamma_X=1,
	\]
	\[
	\begin{aligned}
		\alpha_1^X
		=
		(\ol x_1-x_1)_+,\
		\beta_1^X
		=
		(y_1-\ol y_1)_+,\
		\alpha_2^X
		=
		\ol x_1-\ol y_1+(\ol x_2-x_2)_+,\
		\beta_2^X
		=
		y_1-x_1+(y_2-\ol y_2)_+.
	\end{aligned}
	\]
	By \eqref{eq:U_X_V_X_rec}, we have $\mathcal U_X^{\mathrm{rec}}
	=
	\alpha_1^X+(\alpha_2^X-\alpha_1^X)_+$ and $\mathcal V_X^{\mathrm{rec}}
	=
	\beta_1^X+(\beta_2^X-\beta_1^X)_+$.
	Since \(\Gamma_X=1\), Theorem~
	\ref{thm:unified-max-free-recursive-closed-formulas} gives
	\begin{align}
		H_X(b_1\otimes b_2)
		&=
		\frac12
		\bigl[
		(\ol x_1-x_1)_+
		+
		\left(
		\ol x_1-\ol y_1+(\ol x_2-x_2)_+
		-
		(\ol x_1-x_1)_+
		\right)_+
		\notag\\
		&\quad
		+
		(y_1-\ol y_1)_+
		+
		\left(
		y_1-x_1+(y_2-\ol y_2)_+
		-
		(y_1-\ol y_1)_+
		\right)_+
		\notag\\
		&\quad
		+
		|
		y_1+y_2+\ol y_1+\ol y_2
		-
		x_1-x_2-\ol x_1-\ol x_2
		\notag\\
		&\qquad
		+
		(\ol x_1-x_1)_+
		+
		\left(
		\ol x_1-\ol y_1+(\ol x_2-x_2)_+
		-
		(\ol x_1-x_1)_+
		\right)_+
		\notag\\
		&\qquad
		-
		(y_1-\ol y_1)_+
		-
		\left(
		y_1-x_1+(y_2-\ol y_2)_+
		-
		(y_1-\ol y_1)_+
		\right)_+
		|
		\bigr].
		\label{eq:C2affine-energy-expanded}
	\end{align}
\end{example}

\begin{example}[Energy function in type \(A_5^{(2)}\)]
	\label{ex:A5twisted-energy-matrix}
	Let \(X=A_5^{(2)}\), corresponding to \(n=3\) in type
	\(A_{2n-1}^{(2)}\). By \eqref{eq:A2n-1-perfect-set}, the
	level-\(l\) perfect crystal is
	\[
	B
	=
	\{
	(x_1,x_2,x_3\mid\ol x_3,\ol x_2,\ol x_1)
	\mid
	x_i,\ol x_i\in\mathbb Z_{\geq0}\ (1\leq i\leq3),\
	\sum_{i=1}^{3}(x_i+\ol x_i)=l
	\},
	\]
	so \(|B|=\binom{l+5}{5}\). Let $b_1=(x_1,x_2,x_3\mid\ol x_3,\ol x_2,\ol x_1), b_2=(y_1,y_2,y_3\mid\ol y_3,\ol y_2,\ol y_1)\in B$. By \eqref{eq:uniform-delta-c},
	\eqref{eq:uniform-AB}, and \eqref{eq:uniform-alpha-beta}, we have $\Delta_X=0$, $\Gamma_X=1$,
	and
	\[
	\begin{aligned}
		\alpha_1^X
		&=
		\ol x_1-x_1,
		&
		\beta_1^X
		&=
		y_1-\ol y_1,
		\\
		\alpha_2^X
		&=
		\ol x_1-\ol y_1+(\ol x_2-x_2)_+,
		&
		\beta_2^X
		&=
		y_1-x_1+(y_2-\ol y_2)_+,
		\\
		\alpha_3^X
		&=
		\ol x_1-\ol y_1+\ol x_2-\ol y_2
		+(\ol x_3-x_3)_+,
		&
		\beta_3^X
		&=
		y_1-x_1+y_2-x_2+(y_3-\ol y_3)_+.
	\end{aligned}
	\]

	Hence, by \eqref{eq:U_X_V_X_rec},
	\begin{align*}
		\mathcal U_X^{\mathrm{rec}}
		&=
		\alpha_1^X+(\alpha_2^X-\alpha_1^X)_+
		+
		(
		\alpha_3^X-\alpha_1^X-(\alpha_2^X-\alpha_1^X)_+
		)_+,\\
		\mathcal V_X^{\mathrm{rec}}
		&=
		\beta_1^X+(\beta_2^X-\beta_1^X)_+
		+
		(
		\beta_3^X-\beta_1^X-(\beta_2^X-\beta_1^X)_+
		)_+.
	\end{align*}
	Therefore Theorem~
	\ref{thm:unified-max-free-recursive-closed-formulas} gives
\begin{align}
	H_X(b_1\otimes b_2)
	&=
	\frac12
	\bigl[
	\ol x_1-x_1
	+
	(
	x_1-\ol y_1+(\ol x_2-x_2)_+
	)_+
	\notag\\
	&\quad
	+
	(
	x_1-\ol y_1+\ol x_2-\ol y_2
	+
	(\ol x_3-x_3)_+
	-
	(
	x_1-\ol y_1+(\ol x_2-x_2)_+
	)_+
	)_+
	\notag\\
	&\quad
	+
	y_1-\ol y_1
	+
	(
	\ol y_1-x_1+(y_2-\ol y_2)_+
	)_+
	\notag\\
	&\quad
	+
	(
	\ol y_1-x_1+y_2-x_2
	+
	(y_3-\ol y_3)_+
	-
	(
	\ol y_1-x_1+(y_2-\ol y_2)_+
	)_+
	)_+
	\notag\\
	&\quad
	+
	|
	\ol x_1-x_1
	+
	(
	x_1-\ol y_1+(\ol x_2-x_2)_+
	)_+
	\notag\\
	&\qquad
	+
	(
	x_1-\ol y_1+\ol x_2-\ol y_2
	+
	(\ol x_3-x_3)_+
	-
	(
	x_1-\ol y_1+(\ol x_2-x_2)_+
	)_+
	)_+
	\notag\\
	&\qquad
	-
	y_1+\ol y_1
	-
	(
	\ol y_1-x_1+(y_2-\ol y_2)_+
	)_+
	\notag\\
	&\qquad
	-
	(
	\ol y_1-x_1+y_2-x_2
	+
	(y_3-\ol y_3)_+
	-
	(
	\ol y_1-x_1+(y_2-\ol y_2)_+
	)_+
	)_+
	|
	\bigr].
	\label{eq:A5twisted-energy-expanded}
\end{align}
\end{example}

\begin{example}[Energy function in type \(D_4^{(1)}\)]
	\label{ex:D4affine-energy-matrix}
	Let \(X=D_4^{(1)}\), corresponding to \(n=4\) in type
	\(D_n^{(1)}\). By \eqref{eq:Dn-perfect-set}, the level-\(l\)
	perfect crystal is
	\[
	\begin{aligned}
		B
		=
		\bigl\{
		(x_1,x_2,x_3,x_4\mid
		\ol x_4,\ol x_3,\ol x_2,\ol x_1)
		\ \big|\
		x_i,\ol x_i\in\mathbb Z_{\geq0}\ (1\leq i\leq4),
		\sum_{i=1}^{4}(x_i+\ol x_i)=l,\
		x_4=0\ \text{or}\ \ol x_4=0
		\bigr\}.
	\end{aligned}
	\]
	Hence $|B|
	=
	\binom{l+7}{7}
	-
	\binom{l+5}{7}$.
	Let $b_1=
	(x_1,x_2,x_3,x_4\mid
	\ol x_4,\ol x_3,\ol x_2,\ol x_1), b_2=
	(y_1,y_2,y_3,y_4\mid
	\ol y_4,\ol y_3,\ol y_2,\ol y_1)\in B$. By \eqref{eq:uniform-delta-c},
	\eqref{eq:uniform-AB}, and \eqref{eq:uniform-alpha-beta}, we have $\Delta_X=0$, $\Gamma_X=1$,
	and
	\begingroup
	\small
	\[
	\begin{aligned}
		\alpha_1^X
		&=
		\ol x_1-x_1,
		&
		\beta_1^X
		&=
		y_1-\ol y_1,
		\\
		\alpha_2^X
		&=
		\ol x_1-\ol y_1+(\ol x_2-x_2)_+,
		&
		\beta_2^X
		&=
		y_1-x_1+(y_2-\ol y_2)_+,
		\\
		\alpha_3^X
		&=
		\ol x_1-\ol y_1+\ol x_2-\ol y_2
		+(\ol x_3-x_3)_+,
		&
		\beta_3^X
		&=
		y_1-x_1+y_2-x_2+(y_3-\ol y_3)_+,
		\\
		\alpha_4^X
		&=
		\ol x_1-\ol y_1+\ol x_2-\ol y_2
		+\ol x_3-\ol y_3+x_4-\ol y_4,
		&
		\beta_4^X
		&=
		y_1-x_1+y_2-x_2+y_3-x_3+\ol y_4-x_4.
	\end{aligned}
	\]
	\endgroup	
	Hence, by \eqref{eq:M_r}, \eqref{eq:mathfrak_M}, and
	\eqref{eq:U_X_V_X_rec},
	\begingroup
	\small
	\begin{align*}
		\mathcal U_X^{\mathrm{rec}}
		&=
		\alpha_1^X+(\alpha_2^X-\alpha_1^X)_++
		(
		\alpha_3^X-\alpha_1^X-(\alpha_2^X-\alpha_1^X)_+
		)_+\\
		&\quad+
		(
		\alpha_4^X-\alpha_1^X-(\alpha_2^X-\alpha_1^X)_+
		-
		(
		\alpha_3^X-\alpha_1^X-(\alpha_2^X-\alpha_1^X)_+
		)_+
		)_+,
		\\
		\mathcal V_X^{\mathrm{rec}}
		&=
		\beta_1^X+(\beta_2^X-\beta_1^X)_++
		(
		\beta_3^X-\beta_1^X-(\beta_2^X-\beta_1^X)_+
		)_+\\
		&\quad+
		(
		\beta_4^X-\beta_1^X-(\beta_2^X-\beta_1^X)_+
		-
		(
		\beta_3^X-\beta_1^X-(\beta_2^X-\beta_1^X)_+
		)_+
		)_+.
	\end{align*}
	\endgroup
	Thus Theorem~
	\ref{thm:unified-max-free-recursive-closed-formulas} gives
\begin{align}
	H_X(b_1\otimes b_2)
	&=
	\frac12
	[
	\ol x_1-x_1
	+
	(
	x_1-\ol y_1+(\ol x_2-x_2)_+
	)_+
	\notag\\
	&\quad+
	(
	x_1-\ol y_1+\ol x_2-\ol y_2+(\ol x_3-x_3)_+
	-
	(
	x_1-\ol y_1+(\ol x_2-x_2)_+
	)_+
	)_+
	\notag\\
	&\quad+
	(
	x_1-\ol y_1+\ol x_2-\ol y_2+\ol x_3-\ol y_3+x_4-\ol y_4
	-
	(
	x_1-\ol y_1+(\ol x_2-x_2)_+
	)_+
	\notag\\
	&\qquad
	-
	(
	x_1-\ol y_1+\ol x_2-\ol y_2+(\ol x_3-x_3)_+
	-
	(
	x_1-\ol y_1+(\ol x_2-x_2)_+
	)_+
	)_+
	)_+
	\notag\\
	&\quad+
	y_1-\ol y_1
	+
	(
	\ol y_1-x_1+(y_2-\ol y_2)_+
	)_+
	\notag\\
	&\quad+
	(
	\ol y_1-x_1+y_2-x_2+(y_3-\ol y_3)_+
	-
	(
	\ol y_1-x_1+(y_2-\ol y_2)_+
	)_+
	)_+
	\notag\\
	&\quad+
	(
	\ol y_1-x_1+y_2-x_2+y_3-x_3+\ol y_4-x_4
	-
	(
	\ol y_1-x_1+(y_2-\ol y_2)_+
	)_+
	\notag\\
	&\qquad
	-
	(
	\ol y_1-x_1+y_2-x_2+(y_3-\ol y_3)_+
	-
	(
	\ol y_1-x_1+(y_2-\ol y_2)_+
	)_+
	)_+
	)_+
	\notag\\
	&\quad+
	|
	\ol x_1-x_1
	+
	(
	x_1-\ol y_1+(\ol x_2-x_2)_+
	)_+
	\notag\\
	&\qquad+
	(
	x_1-\ol y_1+\ol x_2-\ol y_2+(\ol x_3-x_3)_+
	-
	(
	x_1-\ol y_1+(\ol x_2-x_2)_+
	)_+
	)_+
	\notag\\
	&\qquad+
	(
	x_1-\ol y_1+\ol x_2-\ol y_2+\ol x_3-\ol y_3+x_4-\ol y_4
	-
	(
	x_1-\ol y_1+(\ol x_2-x_2)_+
	)_+
	\notag\\
	&\qquad
	-
	(
	x_1-\ol y_1+\ol x_2-\ol y_2+(\ol x_3-x_3)_+
	-
	(
	x_1-\ol y_1+(\ol x_2-x_2)_+
	)_+
	)_+
	)_+
	\notag\\
	&\qquad-
	y_1+\ol y_1
	-
	(
	\ol y_1-x_1+(y_2-\ol y_2)_+
	)_+
	\notag\\
	&\qquad-
	(
	\ol y_1-x_1+y_2-x_2+(y_3-\ol y_3)_+
	-
	(
	\ol y_1-x_1+(y_2-\ol y_2)_+
	)_+
	)_+
	\notag\\
	&\qquad-
	(
	\ol y_1-x_1+y_2-x_2+y_3-x_3+\ol y_4-x_4
	-
	(
	\ol y_1-x_1+(y_2-\ol y_2)_+
	)_+
	\notag\\
	&\qquad
	-
	(
	\ol y_1-x_1+y_2-x_2+(y_3-\ol y_3)_+
	-
	(
	\ol y_1-x_1+(y_2-\ol y_2)_+
	)_+
	)_+
	)_+
	|
	].
	\label{eq:D4affine-energy-matrix-formula}
\end{align}
\end{example}

\begin{example}[Energy function in type \(B_3^{(1)}\)]
	\label{ex:B3affine-energy-matrix}
	Let \(X=B_3^{(1)}\), corresponding to \(n=3\) in type
	\(B_n^{(1)}\). By \eqref{eq:Bn-perfect-set}, the level-\(l\)
	perfect crystal is
	\[
	\begin{aligned}
		B
		=
		\bigl\{
		(x_1,x_2,x_3\mid x_0\mid\ol x_3,\ol x_2,\ol x_1)
		\ \big|\
		x_0\in\{0,1\},\
		x_i,\ol x_i\in\mathbb Z_{\geq0}\ (1\leq i\leq3),\
		x_0+\sum_{i=1}^{3}(x_i+\ol x_i)=l
		\bigr\}.
	\end{aligned}
	\]
	Hence $|B|
	=
	\binom{l+5}{5}
	+
	\binom{l+4}{5}$.
	Let $b_1=
	(x_1,x_2,x_3\mid x_0\mid\ol x_3,\ol x_2,\ol x_1), b_2=
	(y_1,y_2,y_3\mid y_0\mid \ol y_3,\ol y_2,\ol y_1)\in B$. By \eqref{eq:uniform-delta-c},
	\eqref{eq:uniform-AB}, and \eqref{eq:uniform-alpha-beta}, we have $\Delta_X=0$, $\Gamma_X=1$,
	and
	\begingroup
	\small
	\[
	\begin{aligned}
		\alpha_1^X
		&=
		\ol x_1-x_1,
		&
		\beta_1^X
		&=
		y_1-\ol y_1,
		\\
		\alpha_2^X
		&=
		\ol x_1-\ol y_1+(\ol x_2-x_2)_+,
		&
		\beta_2^X
		&=
		y_1-x_1+(y_2-\ol y_2)_+,
		\\
		\alpha_3^X
		&=
		\ol x_1-\ol y_1+\ol x_2-\ol y_2
		+(\ol x_3-x_3)_+,
		&
		\beta_3^X
		&=
		y_1-x_1+y_2-x_2+(y_3-\ol y_3)_+.
	\end{aligned}
	\]
	\endgroup
	Hence, by \eqref{eq:M_r}, \eqref{eq:mathfrak_M}, and
	\eqref{eq:U_X_V_X_rec},
	\begin{align*}
		\mathcal U_X^{\mathrm{rec}}
		&=
		\alpha_1^X+(\alpha_2^X-\alpha_1^X)_+
		+
		(
		\alpha_3^X-\alpha_1^X-(\alpha_2^X-\alpha_1^X)_+
		)_+,\\
		\mathcal V_X^{\mathrm{rec}}
		&=
		\beta_1^X+(\beta_2^X-\beta_1^X)_+
		+
		(
		\beta_3^X-\beta_1^X-(\beta_2^X-\beta_1^X)_+
		)_+.
	\end{align*}
	Then Theorem~
	\ref{thm:unified-max-free-recursive-closed-formulas} gives
\begin{align}
	H_X(b_1\otimes b_2)
	&=
	\frac12
	[
	\ol x_1-x_1
	+
	(
	x_1-\ol y_1+(\ol x_2-x_2)_+
	)_+
	\notag\\
	&\quad+
	(
	x_1-\ol y_1+\ol x_2-\ol y_2+(\ol x_3-x_3)_+
	-
	(
	x_1-\ol y_1+(\ol x_2-x_2)_+
	)_+
	)_+
	\notag\\
	&\quad+
	y_1-\ol y_1
	+
	(
	\ol y_1-x_1+(y_2-\ol y_2)_+
	)_+
	\notag\\
	&\quad+
	(
	\ol y_1-x_1+y_2-x_2+(y_3-\ol y_3)_+
	-
	(
	\ol y_1-x_1+(y_2-\ol y_2)_+
	)_+
	)_+
	\notag\\
	&\quad+
	|
	\ol x_1-x_1
	+
	(
	x_1-\ol y_1+(\ol x_2-x_2)_+
	)_+
	\notag\\
	&\qquad+
	(
	x_1-\ol y_1+\ol x_2-\ol y_2+(\ol x_3-x_3)_+
	-
	(
	x_1-\ol y_1+(\ol x_2-x_2)_+
	)_+
	)_+
	\notag\\
	&\qquad-
	y_1+\ol y_1
	-
	(
	\ol y_1-x_1+(y_2-\ol y_2)_+
	)_+
	\notag\\
	&\qquad-
	(
	\ol y_1-x_1+y_2-x_2+(y_3-\ol y_3)_+
	-
	(
	\ol y_1-x_1+(y_2-\ol y_2)_+
	)_+
	)_+
	|
	].
	\label{eq:B3affine-energy-expanded}
\end{align}
\end{example}

\section{The character formulas from the local energies}
\label{sec:effective-character-algorithm}
In this section, we give closed formulas for the affine characters in terms of
the coordinates of perfect crystals by using the local energy functions.

\subsection{The KMN character formulas}

We recall the path realization of the highest weight crystal $B(\lambda)$.
\begin{theorem}[{\cite{KMN1}}]\label{thm:Blambda_isomorphism}
	Let \(B\) be a perfect crystal of level \(l\), and let
	\(\lambda\in(P_{\mathrm{cl}}^+)_l\). Then there exists a \(U_q'(\mathfrak g)\)-crystal isomorphism
	\begin{equation*}
		\begin{aligned}
			\Psi :B(\lambda) \stackrel{\sim}{\longrightarrow}
			B(\ep(b_\lambda)) \otimes B &
			&\text{given by}& &  u_\lambda  \longmapsto  u_{\ep(b_\lambda)} \otimes b_\lambda ,
		\end{aligned}
	\end{equation*}
	where $u_\lambda$ is the highest weight vector in $B(\lambda)$ and $b_\lambda$ is the unique vector in $B$ such that $\varphi(b_{\lambda})=\lambda $.
\end{theorem}
Set 
\begin{equation}\label{eq:inductive_path}
\lambda_0 =\lambda, \quad \lambda_{k+1}=\ep(b_{\lambda_k}), \quad b_0=b_{\lambda_0}, \quad b_{k+1}=b_{\lambda_{k+1}}.
\end{equation}
Applying the above crystal isomorphism repeatedly, we obtain a sequence of crystal isomorphisms
\begin{equation*}
	\begin{array}{ccccccc} B(\lambda) &
		\stackrel{\sim}{\longrightarrow} & B(\lambda_1) \otimes B &
		\stackrel{\sim}{\longrightarrow} & B(\lambda_2)\otimes B \otimes
		B & \stackrel{\sim}{\longrightarrow} & \cdots
		\\
		u_\lambda & \longmapsto & u_{\lambda_1} \otimes b_0 & \longmapsto &
		u_{\lambda_2} \otimes b_1 \otimes b_0 & \longmapsto & \cdots.
\end{array} \end{equation*}
In this process, we obtain an infinite sequence $\mathbf{p}_\lambda =\cdots b_2 \otimes b_1\otimes b_0 \in B^{\otimes \infty}
$, which is called the \textit{ground-state path of weight} $\lambda$.




\begin{lemma}\label{lem:AnCn-ground-state-path}
	Let \(\lambda=\sum_{i\in I}m_i\Lambda_i\in(P_{\mathrm{cl}}^+)_l\).	
	For type \(A_n^{(1)}\), the ground-state path associated with
	\(\lambda\) is \(\mathbf p_\lambda=(b_k)_{k\ge0}\), where $b_k=(m_{i+k})_{i\in I}$,
	and all indices are taken modulo \(n+1\). In particular,
	\(b_{k+n+1}=b_k\), \(b_0=b_\lambda\), and \(b_n=b^\lambda\).
	For type \(C_n^{(1)}\), the ground-state path associated with
	\(\lambda\) is constant: $b_k=(m_1,\dots,m_n\mid m_n,\dots,m_1)$ $(k\ge0)$.
\end{lemma}

\begin{proof}
It is enough to determine the infinite sequence in \eqref{eq:inductive_path}.

For type $A_n^{(1)}$, by Lemma~\ref{lem:minimal-elements-all-types}, applied to
	\(\lambda_0=\sum_{i\in I}m_i\Lambda_i\), we have $b_0=b_{\lambda_0}=(m_i)_{i\in I}$.
	Using \eqref{eq:An-perfect-eps-phi}, we get
	\[
	\lambda_1
	=
	\varepsilon(b_0)
	=
	\sum_{i\in I}\varepsilon_i(b_0)\Lambda_i
	=
	\sum_{i\in I}m_{i+1}\Lambda_i.
	\]
	Applying Lemma~\ref{lem:minimal-elements-all-types} again gives $b_1=b_{\lambda_1}=(m_{i+1})_{i\in I}$.
	Assume inductively that
	\[
	\lambda_k=\sum_{i\in I}m_{i+k}\Lambda_i,
	\qquad
	b_k=b_{\lambda_k}=(m_{i+k})_{i\in I}.
	\]
	Then, by \eqref{eq:An-perfect-eps-phi}, we obtain $\lambda_{k+1}
	=
	\varepsilon(b_k)
	=
	\sum_{i\in I}\varepsilon_i(b_k)\Lambda_i
	=
	\sum_{i\in I}m_{i+k+1}\Lambda_i$.

	Hence Lemma~\ref{lem:minimal-elements-all-types} gives $b_{k+1}=b_{\lambda_{k+1}}=(m_{i+k+1})_{i\in I}$.
	Therefore $b_k=(m_{i+k})_{i\in I}$
	for all \(k\ge0\). 
	
	Since all indices are taken modulo \(n+1\), one has $b_{k+n+1}=(m_{i+k+n+1})_{i\in I}=(m_{i+k})_{i\in I}=b_k$.
	Moreover, by Lemma~\ref{lem:minimal-elements-all-types}, we have $b_n=(m_{i+n})_{i\in I}=(m_{i-1})_{i\in I}=b^\lambda$.	
	
For type $C_n^{(1)}$, by \eqref{eq:inductive_path},
\(\lambda_0=\lambda\), \(b_0=b_{\lambda_0}\),
\(\lambda_{k+1}=\varepsilon(b_{\lambda_k})\), and
\(b_{k+1}=b_{\lambda_{k+1}}\). By
Lemma~\ref{lem:minimal-elements-all-types},
\[
b_0=b_{\lambda_0}=(m_1,\dots,m_n\mid m_n,\dots,m_1).
\]
For this element, \(x_i=\ol x_i=m_i\) for \(1\le i\le n\). Hence
\eqref{eq:common-internal-eps-phi} gives
\(\varepsilon_i(b_0)=\ol x_i+(x_{i+1}-\ol x_{i+1})_+=m_i\) for
\(1\le i\le n-1\), while \eqref{eq:Cn-perfect-eps-phi-n} gives
\(\varepsilon_n(b_0)=\ol x_n=m_n\). Moreover,
\(S(b_0)=2\sum_{i=1}^n m_i=2(l-m_0)\), and therefore
\eqref{eq:Cn-perfect-eps-phi-0} gives $\varepsilon_0(b_0)
=
l-\frac12S(b_0)+(x_1-\ol x_1)_+
=
m_0$.
Thus \(\lambda_1=\varepsilon(b_0)=m_0\Lambda_0+\sum_{i=1}^n m_i\Lambda_i
=\lambda\), and hence \(b_1=b_{\lambda_1}=b_\lambda=b_0\). Inductively,
\[
\lambda_k=\lambda,
\qquad
b_k=b_\lambda=(m_1,\dots,m_n\mid m_n,\dots,m_1)
\]
for all \(k\ge0\). 
\end{proof}

By \cite[Sections~6.7--6.11]{KMN2}, the ground-state path is $(\cdots,b_\lambda,b_\lambda)$ for \(A_{2n}^{(2)}\) and \(D_{n+1}^{(2)}\), and is $(\cdots,b^\lambda,b_\lambda,b^\lambda,b_\lambda)$
for \(A_{2n-1}^{(2)}\), \(D_n^{(1)}\) and \(B_n^{(1)}\).
The minimal elements are those listed in Lemma~\ref{lem:minimal-elements-all-types}.

\vskip 2mm

Let $\lambda\in(P_{\mathrm{cl}}^+)_l$,
and let $\mathbf g_\lambda
=
\cdots\otimes g_2\otimes g_1\otimes g_0$
be the corresponding ground-state path. For \(u,v\in B\) and
\(k\geq0\), define the ground-state corrected local energy by
\begin{equation}\label{eq:uniform-ground-state-corrected-energy}
	\mathcal H_{\lambda,k}^X(u,v):=H_X(u\otimes v)-H_X(g_{k+1}\otimes g_k).
\end{equation}

By Lemma~\ref{lem:energy-uniqueness-general}, any other energy
function \(H_X'\) differs from \(H_X\) by an additive constant:
\(H_X'=H_X+C\) for some \(C\in\mathbb Z\). Hence 
\[
\begin{aligned}
	H_X'(u\otimes v)
	-
	H_X'(g_{k+1}\otimes g_k)
	=
	H_X(u\otimes v)+C
	-
	H_X(g_{k+1}\otimes g_k)-C=
	\mathcal H_{\lambda,k}^X(u,v).
\end{aligned}
\]

Thus, \(\mathcal H_{\lambda,k}^X\) is independent of the additive
normalization chosen for \(H_X\).

Using Lemma~\ref{lem:AnCn-ground-state-path}, the explicit
ground-state paths, and the closed energy formula
\eqref{eq:uniform-energy-seven-types}, we obtain
\begin{equation}\label{eq:ground-state-local-energy-values}
	H_X(g_{k+1}\otimes g_k)
	=
	\begin{cases}
		0,
		&
		X=A_{2n}^{(2)},D_{n+1}^{(2)},C_n^{(1)},
		\\[1mm]
		m_{k+1},
		&
		X=A_n^{(1)},
		\\[1mm]
		(-1)^k(m_1-m_0),
		&
		X=A_{2n-1}^{(2)},D_n^{(1)},B_n^{(1)}.
	\end{cases}
\end{equation}

\vskip 2mm

More explicitly, for $X=A_{2n}^{(2)},D_{n+1}^{(2)},C_n^{(1)}$,
one has $\mathcal H_{\lambda,k}^X(b_\lambda,b_\lambda)=0$.

For $X=A_{2n-1}^{(2)},D_n^{(1)},B_n^{(1)}$, we have $\mathcal H_{\lambda,2s}^X(b^\lambda,b_\lambda)=0$, $\mathcal H_{\lambda,2s+1}^X(b_\lambda,b^\lambda)=0$.

For \(X=A_n^{(1)}\), writing \(g_k=(m_{i+k})_{i\in I}\), we have
\(\mathcal H_{\lambda,k}^{A_n^{(1)}}(g_{k+1},g_k)=0\).

\medskip
Set
\begin{equation}\label{eq:Plambda}
\mathcal P(\lambda)
:=
\left\{
\mathbf p
=
\cdots\otimes p_2\otimes p_1\otimes p_0
\in B^{\otimes\infty}
\ \middle|\
p_k\in B,\quad
p_k=b_k
\text{ for all sufficiently large }k
\right\}.
\end{equation}
The elements of \(\mathcal P(\lambda)\) are called
\emph{\(\lambda\)-paths}. When we wish to emphasize the dependence on
the affine type \(X\), and hence on the corresponding perfect crystal
\(B\), we write
\(\mathcal P_X(\lambda):=\mathcal P(\lambda)\).

The following result gives the {\it path realization} of $B(\lambda)$.

\begin{proposition}[{\cite{KMN1}}]\label{prop:path-realization}
	There exists a \(U_q'(\mathfrak g)\)-crystal isomorphism
	\begin{equation*}
		\Psi_{\lambda} :B(\lambda) \stackrel{\sim}{\longrightarrow}
		\mathcal{P}({\lambda})\ \
		\text{given by}\ \ u_\lambda \longmapsto\mathbf{p}_\lambda,
	\end{equation*}
	where $u_{\lambda}$ is the highest weight vector in $B(\lambda)$.
\end{proposition}

We have the following character formula.

\begin{theorem}[{\cite{KMN2}}]
	\label{theorem:wtchar}
	Let \(\lambda\in (P_{\mathrm{cl}}^+)_l\), and let $g=(g_k)_{k=0}^{\infty}$
	be the ground-state path associated with \(\lambda\). Let
	\(H: B\otimes B\to\Z\) be an energy function, and let $\mathbf{p}=(p_k)^\infty_{k=0}\in \mathcal P(\lambda)$
	be a \(\lambda\)-path. Then the affine weight of \(\mathbf{p}\) is given by
	\begin{align}
		\wt\,\mathbf{p}
		&=
		\lambda
		+
		\sum_{k=0}^{\infty}
		\bigl(\wt\,p_k-\wt\,g_k\bigr)
		-
		\frac{\delta}{d_0}
		\sum_{k=0}^{\infty}
		(k+1)
		\Bigl(
		H(p_{k+1}\otimes p_k)
		-
		H(g_{k+1}\otimes g_k)
		\Bigr)                                         
		\label{eq:firsteq}
	\end{align}
	Moreover, the character of the irreducible highest weight
	\(U_q(\mathfrak g)\)-module \(L(\lambda)\) is
	\begin{equation}
		\ch\,L(\lambda)
		=
		\sum_{\mathbf{p}\in\mathcal P(\lambda)}
		e^{\wt\,\mathbf{p}}.
		\label{eq:charweight}
	\end{equation}
\end{theorem}

For \(X=A_n^{(1)}\), write $p_k=(x_{k,i})_{i\in I}, g_k=(g_{k,i})_{i\in I}$.
We define $\theta_{k,s}^{A_n^{(1)}}:=x_{k,s}-g_{k,s}$ for $1\le s\le n$.
If $\lambda=\sum_{i\in I}m_i\Lambda_i,\
\sum_{i\in I}m_i=l$, then the ground-state path is given by $g_{k,i}=m_{i+k}\ (i\in I)$,
where the indices are read modulo \(n+1\), thus $\theta_{k,s}^{A_n^{(1)}}
=
x_{k,s}-m_{s+k}
\ (1\le s\le n)$.

For the barred-coordinate types, write
\[
p_k=(x_{k,1},\ldots,x_{k,n}\mid \ol x_{k,n},\ldots,\ol x_{k,1}),
\]
with the possible additional \(0\)-coordinate \(x_{k,0}\) in the types
\(D_{n+1}^{(2)}\) and \(B_n^{(1)}\). We also write
\[
g_k=(x_{g_k,1},\ldots,x_{g_k,n}\mid
\ol x_{g_k,n},\ldots,\ol x_{g_k,1})
\]
for the \(k\)-th element of the ground-state path. Thus
\(x_{g_k,s}\) and \(\ol x_{g_k,s}\) denote respectively the \(s\)-th
unbarred and barred coordinates of \(g_k\). Define
\[
\theta_{k,s}^X
:=
(x_{k,s}-\ol x_{k,s})
-
(x_{g_k,s}-\ol x_{g_k,s})
\qquad
(1\le s\le n).
\]

Using the minimal elements in
Lemma~\ref{lem:minimal-elements-all-types} and the explicit ground-state
paths, the above definition gives the following formulas.

Explicitly, we have $\theta_{k,s}^X
=
x_{k,s}-\ol x_{k,s}$ 
for $X=A_{2n}^{(2)},\ D_{n+1}^{(2)},\ C_n^{(1)}$, and 
\[
\theta_{k,s}^X
=
\begin{cases}
	x_{k,1}-\ol x_{k,1}-(-1)^k(m_1-m_0),
	& s=1,\\[1mm]
	x_{k,s}-\ol x_{k,s},
	& 2\le s\le n,
\end{cases}
\]
for $X=A_{2n-1}^{(2)},\ B_n^{(1)}$. For \(X=D_n^{(1)}\), one has
\[
\theta_{k,s}^{D_n^{(1)}}
=
\begin{cases}
	x_{k,1}-\ol x_{k,1}-(-1)^k(m_1-m_0),
	& s=1,\\[1mm]
	x_{k,s}-\ol x_{k,s},
	& 2\le s\le n-1,\\[1mm]
	x_{k,n}-\ol x_{k,n}-(-1)^k(m_n-m_{n-1}),
	& s=n.
\end{cases}
\]

Define $\Xi_X(p_k):=\wt(p_k)-\wt(g_k)$.
In terms of the above \(\theta\)-coordinates, \(\Xi_X(p_k)\) is given as follows.

For \(X=A_n^{(1)}\), whose classical type is \(A_n\), one has
\begin{equation}\label{eq:xiAn1}
\Xi_{A_n^{(1)}}(p_k)
=
\sum_{i=1}^{n}
\left(
\sum_{s=1}^{i}\theta_{k,s}^{A_n^{(1)}}
\right)\alpha_i.
\end{equation}

For \(X=D_{n+1}^{(2)},B_n^{(1)}\), whose classical type is
\(B_n\), one has
\begin{equation}\label{eq:XiDnplus12-Bn1}
\Xi_X(p_k)
=
\sum_{i=1}^{n}
\left(
\sum_{s=1}^{i}\theta_{k,s}^{X}
\right)\alpha_i.
\end{equation}

For \(X=A_{2n}^{(2)}, C_n^{(1)}, A_{2n-1}^{(2)}\), whose classical type is \(C_n\),
one has
\begin{equation}\label{eq:XiA2n2-Cn1-A2nminus12}
\Xi_X(p_k)
=
\sum_{i=1}^{n-1}
\left(
\sum_{s=1}^{i}\theta_{k,s}^{X}
\right)\alpha_i
+
\frac12
\left(
\sum_{s=1}^{n}\theta_{k,s}^{X}
\right)\alpha_n.
\end{equation}

If \(X=D_n^{(1)}\), then
\begin{equation}\label{eq:XiDn1}
\begin{aligned}
	\Xi_{D_n^{(1)}}(p_k)
	=
	\sum_{i=1}^{n-2}
	\left(
	\sum_{s=1}^{i}\theta_{k,s}^{D_n^{(1)}}
	\right)\alpha_i
	+
	\frac12
	\left(
	\sum_{s=1}^{n-1}\theta_{k,s}^{D_n^{(1)}}
	-
	\theta_{k,n}^{D_n^{(1)}}
	\right)\alpha_{n-1}
	+
	\frac12
	\left(
	\sum_{s=1}^{n}\theta_{k,s}^{D_n^{(1)}}
	\right)\alpha_n.
\end{aligned}
\end{equation}

Here \(\alpha_1,\ldots,\alpha_n\) are the simple roots of the corresponding
finite classical root system; the affine simple root \(\alpha_0\) does not
appear in \(\Xi_X(p_k)\).

For the seven types considered in this paper, we have
\[
d_0^X
=
\begin{cases}
	2, & X=A_{2n}^{(2)},\\
	1, & X=A_n^{(1)},D_{n+1}^{(2)},C_n^{(1)},
	A_{2n-1}^{(2)},D_n^{(1)},B_n^{(1)}.
\end{cases}
\]

Let \(\delta_X\) denote the null root of the affine root system of type \(X\). Let 
\begin{equation}\label{eq:Xi_z_q}
\Xi_X(p_k)=\sum_{i=1}^{n} c_{k,i}^X(p_k)\alpha_i,
\qquad
z_i=e^{\alpha_i}\quad(1\le i\le n),
\qquad
\qfrak:=e^{-\delta_X/d_0^X}.
\end{equation}
Then we have the following proposition.

\begin{proposition}\label{prop:closed-ch-formula}
The character of $L(\lambda)$ is given by
\begin{equation}
	\label{eq:uniform-coordinate-character-zq}
	\ch\,L(\lambda)
	=
	e^\lambda
	\sum_{\mathbf{p}\in\mathcal P_X(\lambda)}
	\prod_{i=1}^{n}
	z_i^{\sum_{k\ge0}c_{k,i}^X(p_k)}
	\,
	\qfrak^{
		\sum_{k\ge0}(k+1)
		\mathcal H_{\lambda,k}^X(p_{k+1},p_k)
	}.
\end{equation}
\end{proposition}
\begin{proof}
By the definitions of \(\Xi_X(p_k)\),
\(\mathcal H_{\lambda, k}^X(p_{k+1},p_k)\) and \eqref{eq:firsteq}--\eqref{eq:charweight}, we have 
\begin{equation}
	\label{eq:uniform-coordinate-character}
	\ch\,L(\lambda)
	=
	e^\lambda
	\sum_{\mathbf{p}\in\mathcal P_X(\lambda)}
	\exp\left(
	\sum_{k\ge0}\Xi_X(p_k)
	-
	\frac{\delta_X}{d_0^X}
	\sum_{k\ge0}(k+1)
	\mathcal H_{\lambda,k}^X(p_{k+1},p_k)
	\right).
\end{equation}

Substituting
\eqref{eq:Xi_z_q} into \eqref{eq:uniform-coordinate-character} gives the
desired formula.
\end{proof}

\subsection[The character data]
{The character data
	\(\mathcal P_X(\lambda)\), \(c_{k,i}^{X}(p_k)\), and
	\(\mathcal H_{\lambda, k}^{X}(p_{k+1},p_k)\)}
	\label{subsec:character-data}
In each case below, the parameters \(m_i\) are assumed to satisfy the
corresponding level condition stated in
Lemma~\ref{lem:minimal-elements-all-types}. Using the definitions of
\(\mathcal P_X(\lambda)\) in \eqref{eq:Plambda} and
\(\mathcal H_{\lambda,k}^X\) in
\eqref{eq:uniform-ground-state-corrected-energy}, together with
\eqref{eq:Xi_z_q} and the coordinate expansions of
\(\Xi_X(p_k)\) in \eqref{eq:xiAn1}--\eqref{eq:XiDn1}, we obtain the
following explicit descriptions of
\(\mathcal P_X(\lambda)\), \(c_{k,i}^X(p_k)\), and
\(\mathcal H_{\lambda,k}^X(p_{k+1},p_k)\).

\subsubsection{Type \(A_n^{(1)}\)}\label{subsec:cha_An1}\hfill

\noindent 
$\mathcal P_{A_n^{(1)}}(\lambda)
=
\{
\mathbf{p}=(p_k)_{k\ge0}
\mid
p_k=(x_{k,i})_{i\in I},\
x_{k,i}\in\Z_{\ge0},\
\sum_{i\in I}x_{k,i}=l,\
x_{k,i}=m_{i+k}\text{ for all }i\in I\text{ and }k\gg0
\}$,

\noindent 
$c_{k,i}^{A_n^{(1)}}(p_k)
=
\sum_{s=1}^{i}
\left(
x_{k,s}-m_{s+k}
\right),
\quad 1\le i\le n$,

\noindent 
$\mathcal H_{\lambda, k}^{A_n^{(1)}}(p_{k+1},p_k)
=
\max_{0\le j\le n}
\left(
\sum_{s=1}^{j}
(x_{k,s}-x_{k+1,s})
+
x_{k,j+1}
\right)
-
m_{k+1}$.

\noindent
All subscripts on \(m_i\) and all cyclic coordinate indices are
understood modulo \(n+1\).

\subsubsection{Type \(A_{2n}^{(2)}\)}
\label{subsec:cha_A2n2}\hfill

\vskip 2mm

\noindent
$\mathcal P_{A_{2n}^{(2)}}(\lambda)
=
\{
\mathbf{p}=(p_k)_{k\ge0}
\mid
p_k=(x_{k,i}\mid \ol x_{k,i}),\
x_{k,i},\ol x_{k,i}\in\Z_{\ge0},\
S(p_k)\le l,\
x_{k,i}=m_i,\ \ol x_{k,i}=m_i
\text{ for all }1\le i\le n\text{ and }k\gg0
\}$,

\vskip 2mm

\noindent
$c_{k,i}^{A_{2n}^{(2)}}(p_k)
=
\sum_{r=1}^{i}
\left(
x_{k,r}-\ol x_{k,r}
\right)
\ (1\le i\le n-1),\quad
c_{k,n}^{A_{2n}^{(2)}}(p_k)
=
\frac12
\sum_{r=1}^{n}
\left(
x_{k,r}-\ol x_{k,r}
\right)$,

\vskip 2mm

\noindent
$\mathcal H_{\lambda, k}^{A_{2n}^{(2)}}(p_{k+1},p_k)
=
\max_{1\le r\le n}
\{
\sum_{i=1}^{n}(x_{k,i}-x_{k+1,i})
+
\sum_{j=1}^{r-1}(\ol x_{k+1,j}-\ol x_{k,j})
+
\sum_{j=r+1}^{n}(\ol x_{k,j}-\ol x_{k+1,j})
+
\ol x_{k,r}-x_{k+1,r}
+
|\ol x_{k+1,r}-x_{k+1,r}|,\allowbreak\
\sum_{i=1}^{n}(\ol x_{k+1,i}-\ol x_{k,i})
+
\sum_{j=1}^{r-1}(x_{k,j}-x_{k+1,j})
+
\sum_{j=r+1}^{n}(x_{k+1,j}-x_{k,j})
+
x_{k+1,r}-\ol x_{k,r}
+
|x_{k,r}-\ol x_{k,r}|
\}$.

\subsubsection{Type \(D_{n+1}^{(2)}\)}
\label{subsec:cha_Dnplus12}

Let \(r\in\{0,1\}\) be determined by
\(r\equiv m_n\pmod 2\), and set \(m=(m_n-r)/2\).

\vskip 2mm

\noindent
$\mathcal P_{D_{n+1}^{(2)}}(\lambda)
=
\{
\mathbf{p}=(p_k)_{k\ge0}
\mid
p_k=(x_{k,i}\mid x_{k,0}\mid \ol x_{k,i}),\allowbreak\
x_{k,i},\ol x_{k,i}\in\Z_{\ge0},\allowbreak\
x_{k,0}\in\{0,1\},\allowbreak\
x_{k,0}+S(p_k)\le l,\allowbreak\
x_{k,i}=m_i,\ \ol x_{k,i}=m_i
\text{ for }1\le i\le n-1\text{ and }k\gg0,\allowbreak\
x_{k,n}=m,\ \ol x_{k,n}=m,\allowbreak\
x_{k,0}=r\text{ for }k\gg0
\}$,

\vskip 2mm

\noindent
$c_{k,i}^{D_{n+1}^{(2)}}(p_k)
=
\sum_{s=1}^{i}
\left(
x_{k,s}-\ol x_{k,s}
\right),
\quad
1\le i\le n$,

\vskip 2mm

\noindent
$\mathcal H_{\lambda, k}^{D_{n+1}^{(2)}}(p_{k+1},p_k)
=
\max_{1\le s\le n}
\{
x_{k,0}-x_{k+1,0}
+
\sum_{i=1}^{n}(x_{k,i}-x_{k+1,i})
+
\sum_{j=1}^{s-1}(\ol x_{k+1,j}-\ol x_{k,j})
+
\sum_{j=s+1}^{n}(\ol x_{k,j}-\ol x_{k+1,j})
+
\ol x_{k,s}-x_{k+1,s}
+
|\ol x_{k+1,s}-x_{k+1,s}|,\allowbreak\
x_{k+1,0}-x_{k,0}
+
\sum_{i=1}^{n}(\ol x_{k+1,i}-\ol x_{k,i})
+
\sum_{j=1}^{s-1}(x_{k,j}-x_{k+1,j})
+
\sum_{j=s+1}^{n}(x_{k+1,j}-x_{k,j})
+
x_{k+1,s}-\ol x_{k,s}
+
|x_{k,s}-\ol x_{k,s}|
\}$.

\subsubsection{Type \(C_n^{(1)}\)}
\label{subsec:cha_Cn1}\hfill

\vskip 2mm

\noindent
$\mathcal P_{C_n^{(1)}}(\lambda)
=
\{
\mathbf{p}=(p_k)_{k\ge0}
\mid
p_k=(x_{k,i}\mid \ol x_{k,i}),\
x_{k,i},\ol x_{k,i}\in\Z_{\ge0},\
S(p_k)\le 2l,\allowbreak\
S(p_k)\in2\Z,\
x_{k,i}=m_i,\
\ol x_{k,i}=m_i
\text{ for all }1\le i\le n\text{ and }k\gg0
\}$,

\vskip 2mm

\noindent
$c_{k,i}^{C_n^{(1)}}(p_k)
=
\sum_{r=1}^{i}
\left(
x_{k,r}-\ol x_{k,r}
\right)
\ (1\le i\le n-1),\quad
c_{k,n}^{C_n^{(1)}}(p_k)
=
\frac12
\sum_{r=1}^{n}
\left(
x_{k,r}-\ol x_{k,r}
\right)$,

\vskip 2mm

\noindent
$\mathcal H_{\lambda, k}^{C_n^{(1)}}(p_{k+1},p_k)
=
\frac12
\max_{1\le s\le n}
\{
\sum_{i=1}^{n}(x_{k,i}-x_{k+1,i})
+
\sum_{j=1}^{s-1}(\ol x_{k+1,j}-\ol x_{k,j})
+
\sum_{j=s+1}^{n}(\ol x_{k,j}-\ol x_{k+1,j})
+
\ol x_{k,s}-x_{k+1,s}
+
|\ol x_{k+1,s}-x_{k+1,s}|,\allowbreak\
\sum_{i=1}^{n}(\ol x_{k+1,i}-\ol x_{k,i})
+
\sum_{j=1}^{s-1}(x_{k,j}-x_{k+1,j})
+
\sum_{j=s+1}^{n}(x_{k+1,j}-x_{k,j})
+
x_{k+1,s}-\ol x_{k,s}
+
|x_{k,s}-\ol x_{k,s}|
\}$.

\subsubsection{Type \(A_{2n-1}^{(2)}\)}
\label{subsec:cha_A2nminus12}\hfill

\vskip 2mm

\noindent
$\mathcal P_{A_{2n-1}^{(2)}}(\lambda)
=
\{
\mathbf{p}=(p_k)_{k\ge0}
\mid
p_k=(x_{k,i}\mid \ol x_{k,i}),\
x_{k,i},\ol x_{k,i}\in\Z_{\ge0},\
S(p_k)=l,\allowbreak\
x_{k,i}=\ol x_{k,i}=m_i
\text{ for }2\le i\le n\text{ and }k\gg0,\allowbreak\
(x_{k,1},\ol x_{k,1})
=
\begin{cases}
	(m_1,m_0), & k\equiv0\pmod2,\\
	(m_0,m_1), & k\equiv1\pmod2
\end{cases}
\text{ for }k\gg0
\}$,

\vskip 2mm

\noindent
$c_{k,i}^{A_{2n-1}^{(2)}}(p_k)
=
\sum_{r=1}^{i}
\left(
x_{k,r}-\ol x_{k,r}
\right)
-
(-1)^k(m_1-m_0),
\quad
1\le i\le n-1$,

\vskip 2mm

\noindent
$c_{k,n}^{A_{2n-1}^{(2)}}(p_k)
=
\frac12
\left(
\sum_{r=1}^{n}
\left(
x_{k,r}-\ol x_{k,r}
\right)
-
(-1)^k(m_1-m_0)
\right)$,

\vskip 2mm

\noindent
$\mathcal H_{\lambda, k}^{A_{2n-1}^{(2)}}(p_{k+1},p_k)
=
\max (
\{
\ol x_{k+1,1}-x_{k+1,1},\
x_{k,1}-\ol x_{k,1}
\}
\cup
\{
\sum_{j=1}^{s-1}
(\ol x_{k+1,j}-\ol x_{k,j})
+
(\ol x_{k+1,s}-x_{k+1,s})_+,\allowbreak\
\sum_{j=1}^{s-1}
(x_{k,j}-x_{k+1,j})
+
(x_{k,s}-\ol x_{k,s})_+
\mid
2\le s\le n
\}
)
-
(-1)^k(m_1-m_0)$.

\subsubsection{Type \(D_n^{(1)}\)}
\label{subsec:cha_Dn1}

Retain the notation \(a,u,v\) from part \textup{(6)} of
Lemma~\ref{lem:minimal-elements-all-types}.

\vskip 2mm

\noindent
$\mathcal P_{D_n^{(1)}}(\lambda)
=
\{
\mathbf{p}=(p_k)_{k\ge0}
\mid
p_k=(x_{k,i}\mid \ol x_{k,i}),\
x_{k,i},\ol x_{k,i}\in\Z_{\ge0},\
S(p_k)=l,\allowbreak\
x_{k,n}=0\text{ or }\ol x_{k,n}=0,\allowbreak\
p_k=
\begin{cases}
	(m_1,m_2,\ldots,m_{n-2},a,u
	\mid v,a,m_{n-2},\ldots,m_2,m_0),
	& k\equiv0\pmod2,\\[1mm]
	(m_0,m_2,\ldots,m_{n-2},a,v
	\mid u,a,m_{n-2},\ldots,m_2,m_1),
	& k\equiv1\pmod2
\end{cases}
\text{ for }k\gg0
\}$,

\vskip 2mm

\noindent
$c_{k,i}^{D_n^{(1)}}(p_k)
=
\sum_{r=1}^{i}
\left(
x_{k,r}-\ol x_{k,r}
\right)
-
(-1)^k(m_1-m_0),
\qquad
1\le i\le n-2$,

\vskip 2mm

\noindent
$c_{k,n-1}^{D_n^{(1)}}(p_k)
=
\frac12
\left(
\sum_{r=1}^{n-1}
\left(
x_{k,r}-\ol x_{k,r}
\right)
-
(x_{k,n}-\ol x_{k,n})
-
(-1)^k(m_1-m_0)
+
(-1)^k(m_n-m_{n-1})
\right)$,

\vskip 2mm

\noindent
$c_{k,n}^{D_n^{(1)}}(p_k)
=
\frac12
\left(
\sum_{r=1}^{n}
\left(
x_{k,r}-\ol x_{k,r}
\right)
-
(-1)^k
\bigl((m_1-m_0)+(m_n-m_{n-1})\bigr)
\right)$,

\vskip 2mm

\noindent
$\mathcal H_{\lambda, k}^{D_n^{(1)}}(p_{k+1},p_k)
=
\max(
\{
\ol x_{k+1,1}-x_{k+1,1},\allowbreak\
x_{k,1}-\ol x_{k,1},\allowbreak\
\sum_{j=1}^{n-1}
(\ol x_{k+1,j}-\ol x_{k,j})
+
x_{k+1,n}-\ol x_{k,n},\allowbreak\
\sum_{j=1}^{n-1}
(x_{k,j}-x_{k+1,j})
+
\ol x_{k,n}-x_{k+1,n}
\}
\cup
\{
\sum_{j=1}^{s-1}
(\ol x_{k+1,j}-\ol x_{k,j})
+
(\ol x_{k+1,s}-x_{k+1,s})_+,\allowbreak\
\sum_{j=1}^{s-1}
(x_{k,j}-x_{k+1,j})
+
(x_{k,s}-\ol x_{k,s})_+
\mid
2\le s\le n-1
\}
)
-
(-1)^k(m_1-m_0)$.

\subsubsection{Type \(B_n^{(1)}\)}
\label{subsec:cha_Bn1}

Let \(r\in\{0,1\}\) be determined by
\(r\equiv m_n\pmod2\), and set \(m=(m_n-r)/2\).

\vskip 2mm

\noindent
$\mathcal P_{B_n^{(1)}}(\lambda)
=
\{
\mathbf{p}=(p_k)_{k\ge0}
\mid
p_k=(x_{k,i}\mid x_{k,0}\mid \ol x_{k,i}),\
x_{k,i},\ol x_{k,i}\in\Z_{\ge0},\
x_{k,0}\in\{0,1\},\allowbreak\
x_{k,0}+S(p_k)=l,\allowbreak\
p_k=
\begin{cases}
	(m_1,m_2,\ldots,m_{n-1},m\mid r
	\mid m,m_{n-1},\ldots,m_2,m_0),
	& k\equiv0\pmod2,\\[1mm]
	(m_0,m_2,\ldots,m_{n-1},m\mid r
	\mid m,m_{n-1},\ldots,m_2,m_1),
	& k\equiv1\pmod2
\end{cases}
\text{ for }k\gg0
\}$,

\vskip 2mm

\noindent
$c_{k,i}^{B_n^{(1)}}(p_k)
=
\sum_{s=1}^{i}
\left(
x_{k,s}-\ol x_{k,s}
\right)
-
(-1)^k(m_1-m_0),
\qquad
1\le i\le n$,

\vskip 2mm

\noindent
$\mathcal H_{\lambda, k}^{B_n^{(1)}}(p_{k+1},p_k)
=
\max(
\{
\ol x_{k+1,1}-x_{k+1,1},\allowbreak\
x_{k,1}-\ol x_{k,1}
\}
\cup
\{
\sum_{j=1}^{s-1}
(\ol x_{k+1,j}-\ol x_{k,j})
+
(\ol x_{k+1,s}-x_{k+1,s})_+,\allowbreak\
\sum_{j=1}^{s-1}
(x_{k,j}-x_{k+1,j})
+
(x_{k,s}-\ol x_{k,s})_+
\mid
2\le s\le n
\}
)
-
(-1)^k(m_1-m_0)$.

\vskip 2mm

Substituting the above expressions for
\(\mathcal P_X(\lambda)\), \(c_{k,i}^{X}(p_k)\), and
\(\mathcal H_{\lambda,k}^{X}(p_{k+1},p_k)\) into
\eqref{eq:uniform-coordinate-character-zq} yields a closed coordinate
formula for \(\ch L(\lambda)\) in each of the seven types.

\section{The specialized Rogers--Ramanujan-type identities}
\label{sec:specialized-rogers-ramanujan-type-identities}
In this section, we use the character formula in Proposition \ref{prop:closed-ch-formula} to obtain the specialized Rogers--Ramanujan-type identities.
We recall Kac's principal specialization formula.  

For
\(\mathbf{s}=(s_0,s_1,\ldots,s_n)\in\Z_{>0}^{n+1}\), define
\[
F_{\mathbf s}:
\Z\bigl[\bigl[
e^{-\frac12\alpha_0},e^{-\frac12\alpha_1},\ldots,e^{-\frac12\alpha_n}
\bigr]\bigr]
\longrightarrow \Z[[t^{1/2}]],
\qquad
e^{-\frac12\alpha_i}\longmapsto t^{s_i/2}
\quad(0\le i\le n).
\]

We use the same notation for the induced map on Laurent extensions,
with values in \(\Z((t^{1/2}))\), whenever only finitely many
negative powers occur. In particular,
\[
F_{\mathbf s}(e^{\frac12\alpha_i})=t^{-s_i/2},
\qquad
F_{\mathbf s}(e^{\alpha_i})=t^{-s_i}.
\]
We denote \(F_1:=F_{(1,\ldots,1)}\).

For an affine root system \(\Phi\), let \(\mathfrak g(\Phi)\) be the
corresponding affine Kac--Moody algebra, and set $\operatorname{mult}_{\Phi}(\alpha)
:=
\dim\mathfrak g(\Phi)_\alpha$ $(\alpha\in\Phi)$.
Define
\[
D(\Phi)
:=
\prod_{\alpha\in\Phi^+}
(1-e^{-\alpha})^{\operatorname{mult}_{\Phi}(\alpha)},
\]
where \(\Phi^+\) is the set of positive roots of \(\Phi\).

When \(F_{\mathbf s}\) is applied to 
\(D(\Phi^\vee)\) associated with the dual affine root system
\(\Phi^\vee\), we use the simple roots
\(\alpha_0^\vee,\alpha_1^\vee,\ldots,\alpha_n^\vee\) of
\(\Phi^\vee\) and set $F_{\mathbf s}(e^{-\alpha_i^\vee})=t^{s_i}$ for $0\le i\le n$.

\begin{theorem}[{\cite[Proposition~10.9]{Kac90}}]\label{thm:Kac-principal-specialization}
	For the integrable highest-weight module \(L(\lambda)\) over the affine Lie
	algebra \(\mathfrak g\), one has
	\[
	F_1\!\left(e^{-\lambda}\ch L(\lambda)\right)
	=
	\frac{
		F_{(\lambda(h_0)+1,\ldots,\lambda(h_n)+1)}D(\Phi^\vee)
	}{
		F_1D(\Phi^\vee)
	}.
	\]
\end{theorem}

Let \(X\) be one of the seven affine types considered in this paper, and let
\(\Phi_X\) denote its affine root system. Write $\delta_X=\sum_{i=0}^{n} d_i^X\alpha_i$.
Set $h_X:=\sum_{i=0}^{n} d_i^X$ and $\kappa_X:=\frac{h_X}{d_0^X}$.

Thus, under the principal specialization \(F_1\), one has
\begin{equation}\label{eq:F1eqz}
F_1(e^{-\delta_X})=t^{h_X},
\qquad
F_1(\mathfrak q)
=
F_1(e^{-\delta_X/d_0^X})
=
t^{\kappa_X},
\qquad
F_1(z_i)=t^{-1}
\quad(1\le i\le n).
\end{equation}

Let $\lambda=\sum_{i=0}^{n}m_i\Lambda_i\in (P_{\mathrm{cl}}^+)_l$. For
\(\mathbf{p}=(p_k)_{k\ge0}\in\mathcal P_X(\lambda)\), define
\[
\mathscr D_X(\mathbf{p})
:=
\kappa_X
\sum_{k\ge0}(k+1)\mathcal H_{\lambda, k}^X(p_{k+1},p_k)
-
\sum_{i=1}^{n}\sum_{k\ge0}c_{k,i}^X(p_k).
\]
The sums are finite because \(p_k=g_k\) for all sufficiently large \(k\).

\begin{proposition}\label{prop:nonnegative-path-degree}
	For every
	\(X\in\{A_n^{(1)},A_{2n}^{(2)},D_{n+1}^{(2)},C_n^{(1)},
	A_{2n-1}^{(2)},D_n^{(1)},B_n^{(1)}\}\)
	and every path
	\(\mathbf p\in\mathcal P_X(\lambda)\), one has $\mathscr D_X(\mathbf p)\in\mathbb Z_{\ge0}$.
\end{proposition}

\begin{proof}
	By the path weight formula \eqref{eq:firsteq}, the expansion
	\eqref{eq:Xi_z_q}, the specialization identities
	\eqref{eq:F1eqz}, and the definition of
	\(\mathscr D_X(\mathbf p)\), the principal specialization gives
	\begin{equation}\label{eq:F1ewtp}
		F_1\!\left(e^{\wt(\mathbf p)-\lambda}\right)
		=
		t^{\mathscr D_X(\mathbf p)}.
	\end{equation}
	
	On the other hand, under the path realization of \(B(\lambda)\) in
	Proposition~\ref{prop:path-realization}, the path \(\mathbf p\)
	corresponds to an element of the highest-weight crystal \(B(\lambda)\).
	Hence $\lambda-\wt(\mathbf p)\in Q_+$.
	Thus
	\[
	\lambda-\wt(\mathbf p)
	=
	\sum_{i=0}^{n}N_i\alpha_i,
	\quad
	N_i\in\mathbb Z_{\ge0},
	\]
	and therefore 
	\[
	\begin{aligned}
		F_1\!\left(e^{\wt(\mathbf p)-\lambda}\right)
		=
		F_1\!\left(
		e^{-\sum_{i=0}^{n}N_i\alpha_i}
		\right)=
		t^{\sum_{i=0}^{n}N_i}.
	\end{aligned}
	\]
	Comparing this identity with \eqref{eq:F1ewtp}, we obtain $\mathscr D_X(\mathbf p)
	=
	\sum_{i=0}^{n}N_i
	\in
	\mathbb Z_{\ge0}$.
\end{proof}

\begin{theorem}
	\label{thm:principal-specialized-path-product-identity}
For each of the seven affine types \(X\) and
every
\(\lambda=\sum_{i=0}^n m_i\Lambda_i\in(P_{\mathrm{cl}}^+)_l\),
the following identity holds.	
	\[
	\sum_{\mathbf{p}\in\mathcal P_X(\lambda)}
	t^{\mathscr D_X(\mathbf{p})}
	=
	\frac{
		F_{(m_0+1,\ldots,m_n+1)}D(\Phi_X^\vee)
	}{
		F_1D(\Phi_X^\vee)
	}.
	\]
\end{theorem}

\begin{proof}
Applying the principal specialization \(F_1\) to
\eqref{eq:uniform-coordinate-character-zq}, 
together with the definition of \(\mathscr D_X(\mathbf p)\), we obtain
\[
F_1\!\left(e^{-\lambda}\ch L(\lambda)\right)
=
\sum_{\mathbf p\in\mathcal P_X(\lambda)}
t^{\mathscr D_X(\mathbf p)}.
\]
By Proposition~\ref{prop:nonnegative-path-degree}, all the exponents
on the right-hand side are nonnegative integers. Moreover, for each
\(N\ge0\), only finitely many paths satisfy
\(\mathscr D_X(\mathbf p)=N\), since the corresponding principal
graded component of \(L(\lambda)\) is finite-dimensional. Hence the
right-hand side is a well-defined element of \(\mathbb Z[[t]]\).

On the other hand, since \(\lambda(h_i)=m_i\), 
Theorem~\ref{thm:Kac-principal-specialization} gives
\[
F_1\!\left(e^{-\lambda}\ch L(\lambda)\right)
=
\frac{
	F_{(m_0+1,\ldots,m_n+1)}D(\Phi_X^\vee)
}{
	F_1D(\Phi_X^\vee)
}.
\]
The desired identity follows.
\end{proof}

\subsection{The data for the sum side}
\label{subsec:The-sum-side}

\begin{proposition}
	\label{prop:adjacent-pair-degree-form}
	For each affine type \(X\), each dominant
	level-\(l\) weight \(\lambda\), and every path
	\(\mathbf p\in\mathcal P_X(\lambda)\), one has
	\[
	\mathscr D_X(\mathbf p)
	=
	\sum_{k\ge0}(k+1)E^X_{\lambda, k}(p_k,p_{k+1}).
	\]
	The adjacent contribution \(E^X_{\lambda, k}(p_k,p_{k+1})\) is given as follows.

	\begin{equation}\label{eq:E-An1}
	\begin{aligned}
		E_{\lambda, k}^{A_n^{(1)}}(p_k,p_{k+1})
		=
		(n+1)\mathcal H_{\lambda, k}^{A_n^{(1)}}(p_{k+1},p_k)-
		\sum_{s=1}^{n}
		(n-s+1)
		\Bigl[
		(x_{k,s}-m_{s+k})-(x_{k+1,s}-m_{s+k+1})
		\Bigr],
	\end{aligned}
	\end{equation}

	\begin{equation}\label{eq:E-A2n2}
	\begin{aligned}
		E_{\lambda, k}^{A_{2n}^{(2)}}(p_k,p_{k+1})
		=
		\frac{2n+1}{2}
		\mathcal H_{\lambda, k}^{A_{2n}^{(2)}}(p_{k+1},p_k) -
		\frac12
		\sum_{r=1}^{n}
		(2n-2r+1)
		\Bigl[
		(x_{k,r}-\ol x_{k,r})
		-
		(x_{k+1,r}-\ol x_{k+1,r})
		\Bigr].
	\end{aligned}
	\end{equation}

	\begin{equation}\label{eq:E-Dnplus12}
	\begin{aligned}
		E_{\lambda, k}^{D_{n+1}^{(2)}}(p_k,p_{k+1})
		=
		(n+1)\mathcal H_{\lambda, k}^{D_{n+1}^{(2)}}(p_{k+1},p_k) -
		\sum_{s=1}^{n}
		(n-s+1)
		\Bigl[
		(x_{k,s}-\ol x_{k,s})
		-
		(x_{k+1,s}-\ol x_{k+1,s})
		\Bigr].
	\end{aligned}
	\end{equation}

	\begin{equation}\label{eq:E-Cn1}
	\begin{aligned}
		E_{\lambda, k}^{C_n^{(1)}}(p_k,p_{k+1})
		=
		2n\mathcal H_{\lambda, k}^{C_n^{(1)}}(p_{k+1},p_k)-
		\frac12
		\sum_{r=1}^{n}
		(2n-2r+1)
		\Bigl[
		(x_{k,r}-\ol x_{k,r})
		-
		(x_{k+1,r}-\ol x_{k+1,r})
		\Bigr].
	\end{aligned}
	\end{equation}

	\begin{equation}\label{eq:E-A2nminus12}
	\begin{aligned}
		E_{\lambda, k}^{A_{2n-1}^{(2)}}(p_k,p_{k+1})
		&=
		(2n-1)\mathcal H_{\lambda, k}^{A_{2n-1}^{(2)}}(p_{k+1},p_k)  -
		\frac12(2n-1)
		\Bigl[
		x_{k,1}-\ol x_{k,1}-(-1)^k(m_1-m_0)
		\\
		&
		\quad-
		x_{k+1,1}+\ol x_{k+1,1}
		+(-1)^{k+1}(m_1-m_0)
		\Bigr]
		\\
		&\quad-
		\frac12
		\sum_{r=2}^{n}
		(2n-2r+1)
		\Bigl[
		(x_{k,r}-\ol x_{k,r})
		-
		(x_{k+1,r}-\ol x_{k+1,r})
		\Bigr].
	\end{aligned}
	\end{equation}

	\begin{equation}\label{eq:E-Dn1}
	\begin{aligned}
		E_{\lambda, k}^{D_n^{(1)}}(p_k,p_{k+1})
		&=
		(2n-2)\mathcal H_{\lambda, k}^{D_n^{(1)}}(p_{k+1},p_k) -
		(n-1)
		\Bigl[
		x_{k,1}-\ol x_{k,1}-(-1)^k(m_1-m_0)
		\\
		&\quad
		-
		x_{k+1,1}+\ol x_{k+1,1}
		+(-1)^{k+1}(m_1-m_0)
		\Bigr]
		\\
		&\quad
		-
		\sum_{r=2}^{n-1}
		(n-r)
		\Bigl[
		(x_{k,r}-\ol x_{k,r})
		-
		(x_{k+1,r}-\ol x_{k+1,r})
		\Bigr].
	\end{aligned}
	\end{equation}

	\begin{equation}\label{eq:E-Bn1}
	\begin{aligned}
		E_{\lambda, k}^{B_n^{(1)}}(p_k,p_{k+1})
		&=
		2n\mathcal H_{\lambda, k}^{B_n^{(1)}}(p_{k+1},p_k)-
		n
		\Bigl[
		x_{k,1}-\ol x_{k,1}-(-1)^k(m_1-m_0)
		\\
		&
		-
		x_{k+1,1}+\ol x_{k+1,1}
		+(-1)^{k+1}(m_1-m_0)
		\Bigr]
		\\
		&\quad
		-
		\sum_{s=2}^{n}
		(n-s+1)
		\Bigl[
		(x_{k,s}-\ol x_{k,s})
		-
		(x_{k+1,s}-\ol x_{k+1,s})
		\Bigr].
	\end{aligned}
	\end{equation}
\end{proposition}

\begin{proof}
For a path $\mathbf{p}=(p_k)_{k\ge0}\in \mathcal P_{X}(\lambda)$, by the descriptions of \(\mathcal P_X(\lambda)\) and \(c_{k,i}^X\) given in Subsection~\ref{subsec:character-data}, the statistic \(\mathscr D_X(\mathbf{p})\) is described explicitly as follows.

\begin{equation}\label{eq:disc-A-n-1}
	\begin{aligned}
		\mathscr D_{A_n^{(1)}}(\mathbf{p})
		=
		(n+1)
		\sum_{k\ge0}(k+1)
		\mathcal H_{\lambda, k}^{A_n^{(1)}}(p_{k+1},p_k)
		-
		\sum_{k\ge0}\sum_{s=1}^{n}
		(n-s+1)(x_{k,s}-m_{s+k}).
	\end{aligned}
\end{equation}

\begin{equation}\label{eq:disc-A-2n-2}
	\begin{aligned}
		\mathscr D_{A_{2n}^{(2)}}(\mathbf{p})=
		\frac{2n+1}{2}
		\sum_{k\ge0}(k+1)
		\mathcal H_{\lambda, k}^{A_{2n}^{(2)}}(p_{k+1},p_k)
		-
		\frac12
		\sum_{k\ge0}\sum_{r=1}^{n}
		(2n-2r+1)(x_{k,r}-\ol x_{k,r}).
	\end{aligned}
\end{equation}

\begin{equation}\label{eq:disc-D-nplus1-2}
	\begin{aligned}
		\mathscr D_{D_{n+1}^{(2)}}(\mathbf{p})=
		(n+1)
		\sum_{k\ge0}(k+1)
		\mathcal H_{\lambda, k}^{D_{n+1}^{(2)}}(p_{k+1},p_k)
		-
		\sum_{k\ge0}\sum_{s=1}^{n}
		(n-s+1)(x_{k,s}-\ol x_{k,s}).
	\end{aligned}
\end{equation}

\begin{equation}\label{eq:disc-C-n-1}
	\begin{aligned}
		\mathscr D_{C_n^{(1)}}(\mathbf{p})=
		2n
		\sum_{k\ge0}(k+1)
		\mathcal H_{\lambda, k}^{C_n^{(1)}}(p_{k+1},p_k)
		-
		\frac12
		\sum_{k\ge0}\sum_{r=1}^{n}
		(2n-2r+1)(x_{k,r}-\ol x_{k,r}).
	\end{aligned}
\end{equation}

\begin{equation}\label{eq:disc-A-2nminus1-2}
	\begin{aligned}
		\mathscr D_{A_{2n-1}^{(2)}}(\mathbf{p})
		&=
		(2n-1)
		\sum_{k\ge0}(k+1)
		\mathcal H_{\lambda, k}^{A_{2n-1}^{(2)}}(p_{k+1},p_k) \\
		&\quad
		-
		\frac12
		\sum_{k\ge0}
		\left[
		\sum_{r=1}^{n}
		(2n-2r+1)(x_{k,r}-\ol x_{k,r})
		-
		(2n-1)(-1)^k(m_1-m_0)
		\right].
	\end{aligned}
\end{equation}

\begin{equation}\label{eq:disc-D-n-1}
	\begin{aligned}
		\mathscr D_{D_n^{(1)}}(\mathbf{p})
		&=
		(2n-2)
		\sum_{k\ge0}(k+1)
		\mathcal H_{\lambda, k}^{D_n^{(1)}}(p_{k+1},p_k) \\
		&\quad
		-
		\sum_{k\ge0}
		\left[
		\sum_{r=1}^{n-1}(n-r)(x_{k,r}-\ol x_{k,r})
		-(n-1)(-1)^k(m_1-m_0)
		\right].
	\end{aligned}
\end{equation}

\begin{equation}\label{eq:disc-B-n-1}
	\begin{aligned}
		\mathscr D_{B_n^{(1)}}(\mathbf{p})
		&=
		2n
		\sum_{k\ge0}(k+1)
		\mathcal H_{\lambda, k}^{B_n^{(1)}}(p_{k+1},p_k) \\
		&\quad-
		\sum_{k\ge0}
		\left[
		\sum_{s=1}^{n}
		(n-s+1)(x_{k,s}-\ol x_{k,s})
		-
		n(-1)^k(m_1-m_0)
		\right].
	\end{aligned}
\end{equation}		
	
We use the telescoping identity
\begin{equation}\label{eq:telescoping-identity}
\sum_{k\ge0}U_k
=
\sum_{k\ge0}(k+1)(U_k-U_{k+1}),
\end{equation}
valid whenever \(U_k=0\) for \(k\gg0\).

For type \(A_n^{(1)}\), the ground-state tail gives $x_{k,s}-m_{s+k}=0$ $(k\gg0)$,
where the index is read modulo \(n+1\).  By \eqref{eq:telescoping-identity}, for each \(1\le s\le n\),
\[
\sum_{k\ge0}(x_{k,s}-m_{s+k})
=
\sum_{k\ge0}(k+1)
\Bigl[
(x_{k,s}-m_{s+k})-(x_{k+1,s}-m_{s+k+1})
\Bigr].
\]
Substituting these identities into \eqref{eq:disc-A-n-1} gives the displayed
formula for \(E_{\lambda, k}^{A_n^{(1)}}(p_k,p_{k+1})\).

For the types \(A_{2n}^{(2)}\), \(D_{n+1}^{(2)}\), and \(C_n^{(1)}\), the
ground-state tails satisfy $x_{k,r}-\ol x_{k,r}=0$ $(k\gg0)$.

From \eqref{eq:telescoping-identity}, for each relevant \(r\), we have
\[
\sum_{k\ge0}(x_{k,r}-\ol x_{k,r})
=
\sum_{k\ge0}(k+1)
\Bigl[
(x_{k,r}-\ol x_{k,r})
-
(x_{k+1,r}-\ol x_{k+1,r})
\Bigr].
\]
Substituting these identities into
\eqref{eq:disc-A-2n-2}, \eqref{eq:disc-D-nplus1-2}, and
\eqref{eq:disc-C-n-1} gives the displayed formulas for
\(E_{\lambda, k}^{A_{2n}^{(2)}}\), \(E_{\lambda, k}^{D_{n+1}^{(2)}}\), and \(E_{\lambda, k}^{C_n^{(1)}}\).

For the types \(A_{2n-1}^{(2)}\), \(D_n^{(1)}\), and \(B_n^{(1)}\), the
two-periodic ground-state tail gives
\[
x_{k,1}-\ol x_{k,1}-(-1)^k(m_1-m_0)=0
\qquad(k\gg0),
\]
and, for the remaining barred coordinates appearing in the correction term, we have $x_{k,r}-\ol x_{k,r}=0$ for $k\gg0$.

It follows from \eqref{eq:telescoping-identity} that
\[
\begin{aligned}
	&\sum_{k\ge0}
	\bigl(x_{k,1}-\ol x_{k,1}-(-1)^k(m_1-m_0)\bigr) \\
	=&
	\sum_{k\ge0}(k+1)
	\Bigl[
	x_{k,1}-\ol x_{k,1}-(-1)^k(m_1-m_0)
	-
	x_{k+1,1}+\ol x_{k+1,1}
	+(-1)^{k+1}(m_1-m_0)
	\Bigr],
\end{aligned}
\]
and, for the remaining relevant coordinates,
\[
\sum_{k\ge0}(x_{k,r}-\ol x_{k,r})
=
\sum_{k\ge0}(k+1)
\Bigl[
(x_{k,r}-\ol x_{k,r})
-
(x_{k+1,r}-\ol x_{k+1,r})
\Bigr].
\]
Substituting these identities into
\eqref{eq:disc-A-2nminus1-2}, \eqref{eq:disc-D-n-1}, and
\eqref{eq:disc-B-n-1} gives the displayed formulas for
\(E_{\lambda, k}^{A_{2n-1}^{(2)}}\), \(E_{\lambda, k}^{D_n^{(1)}}\) and \(E_{\lambda, k}^{B_n^{(1)}}\).  Therefore $\mathscr D_X(\mathbf p)
=
\sum_{k\ge0}(k+1)E_{\lambda, k}^X(p_k,p_{k+1})$ holds for all seven cases.
\end{proof}

\vskip 2mm

\begin{proposition}\label{prop:adjacent-energy-independent-of-k}
For each of the seven affine types \(X\), each
	\(\lambda\in(P_{\mathrm{cl}}^+)_l\), every path
	\(\mathbf p=(p_k)_{k\ge0}\in\mathcal P_X(\lambda)\), and every
	\(k\ge0\), the value
	\(E_{\lambda,k}^X(p_k,p_{k+1})\) depends only on the ordered pair
	\((p_k,p_{k+1})\).
\end{proposition}
\begin{proof}
	Substituting \eqref{eq:ground-state-local-energy-values} into
	\eqref{eq:uniform-ground-state-corrected-energy} gives
	\[
	\mathcal H_{\lambda,k}^X(p_{k+1},p_k)
	=
	\begin{cases}
		H_X(p_{k+1}\otimes p_k),
		&
		X=A_{2n}^{(2)},D_{n+1}^{(2)},C_n^{(1)},\\[1mm]
		H_X(p_{k+1}\otimes p_k)-m_{k+1},
		&
		X=A_n^{(1)},\\[1mm]
		H_X(p_{k+1}\otimes p_k)-(-1)^k(m_1-m_0),
		&
		X=A_{2n-1}^{(2)},D_n^{(1)},B_n^{(1)}.
	\end{cases}
	\]
	
	For
	\(X=A_{2n}^{(2)},D_{n+1}^{(2)},C_n^{(1)}\), substituting the first
	case into \eqref{eq:E-A2n2}, \eqref{eq:E-Dnplus12}, and
	\eqref{eq:E-Cn1}, respectively, shows immediately that the resulting
	expressions depend only on the coordinates of \(p_k,p_{k+1}\) and on
	\(H_X(p_{k+1}\otimes p_k)\). Hence the assertion holds for these three
	types.
	
	It remains to consider the other four types. For \(X=A_n^{(1)}\),
	all indices of the \(m_i\) are read modulo \(n+1\), we have
	\begin{equation}\label{eq:nplus1mkplus1}
	\begin{aligned}
		\sum_{s=1}^{n}(n-s+1)(m_{k+s}-m_{k+s+1})
		=
		n m_{k+1}-\sum_{s=2}^{n+1}m_{k+s}
		=
		(n+1)m_{k+1}-l.
	\end{aligned}
	\end{equation}
	Substituting the second case above into \eqref{eq:E-An1} and using
	the identity \eqref{eq:nplus1mkplus1}, we obtain
	\begin{equation}\label{eq:E-An1-pair-only}
		\begin{aligned}
			E_{\lambda,k}^{A_n^{(1)}}(p_k,p_{k+1})
			&=
			(n+1)
			\bigl(
			H_{A_n^{(1)}}(p_{k+1}\otimes p_k)-m_{k+1}
			\bigr)\\
			&\quad-
			\sum_{s=1}^{n}(n-s+1)
			\Bigl[
			(x_{k,s}-x_{k+1,s})
			-
			(m_{k+s}-m_{k+s+1})
			\Bigr]\\
			&=
			(n+1)H_{A_n^{(1)}}(p_{k+1}\otimes p_k)
			-
			\sum_{s=1}^{n}(n-s+1)(x_{k,s}-x_{k+1,s})-l.
		\end{aligned}
	\end{equation}
	
	For \(X=A_{2n-1}^{(2)}\), substituting the third case above into
	\eqref{eq:E-A2nminus12} and using
	\((-1)^{k+1}=-(-1)^k\), we obtain
	\begin{equation}\label{eq:E-A2nminus12-pair-only}
		\begin{aligned}
			E_{\lambda,k}^{A_{2n-1}^{(2)}}(p_k,p_{k+1})
			&=
			(2n-1)
			\Bigl(
			H_{A_{2n-1}^{(2)}}(p_{k+1}\otimes p_k)
			-
			(-1)^k(m_1-m_0)
			\Bigr)\\
			&\quad-
			\frac{2n-1}{2}
			\Bigl[
			x_{k,1}-\ol x_{k,1}
			-
			x_{k+1,1}+\ol x_{k+1,1}
			-
			2(-1)^k(m_1-m_0)
			\Bigr]\\
			&\quad-
			\frac12
			\sum_{r=2}^{n}(2n-2r+1)
			\Bigl[
			(x_{k,r}-\ol x_{k,r})
			-
			(x_{k+1,r}-\ol x_{k+1,r})
			\Bigr]\\
			&=
			(2n-1)H_{A_{2n-1}^{(2)}}(p_{k+1}\otimes p_k)\\
			&\quad-
			\frac12
			\sum_{r=1}^{n}(2n-2r+1)
			\Bigl[
			(x_{k,r}-\ol x_{k,r})
			-
			(x_{k+1,r}-\ol x_{k+1,r})
			\Bigr].
		\end{aligned}
	\end{equation}
	
	For \(X=D_n^{(1)}\), substituting the third case above into
	\eqref{eq:E-Dn1} gives
	\begin{equation}\label{eq:E-Dn1-pair-only}
		\begin{aligned}
			E_{\lambda,k}^{D_n^{(1)}}(p_k,p_{k+1})
			&=
			(2n-2)
			\Bigl(
			H_{D_n^{(1)}}(p_{k+1}\otimes p_k)
			-
			(-1)^k(m_1-m_0)
			\Bigr)\\
			&\quad-
			(n-1)
			\Bigl[
			x_{k,1}-\ol x_{k,1}
			-
			x_{k+1,1}+\ol x_{k+1,1}
			-
			2(-1)^k(m_1-m_0)
			\Bigr]\\
			&\quad-
			\sum_{r=2}^{n-1}(n-r)
			\Bigl[
			(x_{k,r}-\ol x_{k,r})
			-
			(x_{k+1,r}-\ol x_{k+1,r})
			\Bigr]\\
			&=
			(2n-2)H_{D_n^{(1)}}(p_{k+1}\otimes p_k)\\
			&\quad-
			\sum_{r=1}^{n-1}(n-r)
			\Bigl[
			(x_{k,r}-\ol x_{k,r})
			-
			(x_{k+1,r}-\ol x_{k+1,r})
			\Bigr].
		\end{aligned}
	\end{equation}
	
	Finally, for \(X=B_n^{(1)}\), substituting the third case above into
	\eqref{eq:E-Bn1} gives
	\begin{equation}\label{eq:E-Bn1-pair-only}
		\begin{aligned}
			E_{\lambda,k}^{B_n^{(1)}}(p_k,p_{k+1})
			&=
			2n
			\Bigl(
			H_{B_n^{(1)}}(p_{k+1}\otimes p_k)
			-
			(-1)^k(m_1-m_0)
			\Bigr)\\
			&\quad-
			n
			\Bigl[
			x_{k,1}-\ol x_{k,1}
			-
			x_{k+1,1}+\ol x_{k+1,1}
			-
			2(-1)^k(m_1-m_0)
			\Bigr]\\
			&\quad-
			\sum_{r=2}^{n}(n-r+1)
			\Bigl[
			(x_{k,r}-\ol x_{k,r})
			-
			(x_{k+1,r}-\ol x_{k+1,r})
			\Bigr]\\
			&=
			2nH_{B_n^{(1)}}(p_{k+1}\otimes p_k)
			-
			\sum_{r=1}^{n}(n-r+1)
			\Bigl[
			(x_{k,r}-\ol x_{k,r})
			-
			(x_{k+1,r}-\ol x_{k+1,r})
			\Bigr].
		\end{aligned}
	\end{equation}
	
	The right-hand sides of
	\eqref{eq:E-An1-pair-only},
	\eqref{eq:E-A2nminus12-pair-only},
	\eqref{eq:E-Dn1-pair-only}, and
	\eqref{eq:E-Bn1-pair-only} depend only on
	\(H_X(p_{k+1}\otimes p_k)\), the coordinates of \(p_k,p_{k+1}\),
	and, in type \(A_n^{(1)}\), the fixed level \(l\). Thus they depend
	on \(k\) only through the ordered pair \((p_k,p_{k+1})\). Together with the first three cases, this proves the proposition.
\end{proof}

By Proposition~\ref{prop:adjacent-energy-independent-of-k},
\(E_{\lambda,k}^X(p_k,p_{k+1})\) is determined solely by the ordered
pair \((p_k,p_{k+1})\). We may therefore suppress the index \(k\) and
define the function \(E_\lambda^X\) on \(B\times B\) by
\begin{equation}\label{eq:ElambdaX}
	E_\lambda^X(b,b')
	:=
	E_{\lambda,k}^X(b,b'),
	\qquad b,b'\in B,
\end{equation}
where the right-hand side is independent of \(k\).

By Proposition~\ref{prop:adjacent-pair-degree-form}, the degree \(\mathscr D_X(\mathbf p)\) is determined by
the adjacent-pair matrix
\begin{equation*}
\bigl(E_\lambda^X(b,b')\bigr)_{b,b'\in B},
\end{equation*}
whose entries are obtained explicitly from the formulas \eqref{eq:E-An1}--\eqref{eq:E-Bn1} in
Proposition~\ref{prop:adjacent-pair-degree-form}.

\subsection{The data for the product sides}
\label{subsec:uniform-root-data-product-sides}

In the product side of
Theorem~\ref{thm:principal-specialized-path-product-identity}, the affine
root system which occurs is the dual affine root system \(\Phi_X^\vee\).
We record here the notation and root data used in the product evaluations
below.

Throughout this subsection, let $s_i=m_i+1$ for $0\leq i\leq n$.
Set  $M_{a,b}:=\sum_{j=a}^{b}s_j$ for $1\le a\le b\le n$.

For each type \(X\), the integer \(L\) is defined by
\[
F_{\mathbf s}(e^{-\delta_X^\vee})=F_{(m_0+1,\ldots,m_n+1)}(e^{-\delta_X^\vee})=t^L,
\]
where \(\delta_X^\vee\) is the null root of \(\Phi_X^\vee\).

For each affine type \(X\), we write $\Phi_X^{\vee,+,\mathrm{im}}$ and $\Phi_X^{\vee,+,\mathrm{re}}$
for the sets of positive imaginary roots and positive real roots of the dual
affine root system \(\Phi_X^\vee\), respectively.

We compute the product side by separating
imaginary and real roots:
\[
\begin{aligned}
	F_{\mathbf{s}}D(\Phi_X^\vee)
	=
	\prod_{\alpha\in\Phi_X^{\vee,+,\mathrm{im}}}
	\left(1-F_{\mathbf{s}}(e^{-\alpha})\right)^{\operatorname{mult}\alpha}\times
	\prod_{\alpha\in\Phi_X^{\vee,+,\mathrm{re}}}
	\left(1-F_{\mathbf{s}}(e^{-\alpha})\right),
\end{aligned}
\]
\[
\begin{aligned}
	F_1D(\Phi_X^\vee)
	&=
	\prod_{\alpha\in\Phi_X^{\vee,+,\mathrm{im}}}
	\left(1-F_1(e^{-\alpha})\right)^{\operatorname{mult}\alpha}\times
	\prod_{\alpha\in\Phi_X^{\vee,+,\mathrm{re}}}
	\left(1-F_1(e^{-\alpha})\right).
\end{aligned}
\]

We use the standard \(q\)-Pochhammer notation $(a;u)_\infty:=\prod_{j=0}^{\infty}(1-au^j)$,
and, for several parameters, we set $(a_1,\ldots,a_r;u)_\infty
:=
(a_1;u)_\infty\cdots(a_r;u)_\infty$.

Using the root data in Appendix~\ref{app:root-data}, we have the following results.

\vskip 2mm

\paragraph{\bf Type \(A_n^{(1)}\).}
\begin{equation}\label{eq:product-side-An1}
	\frac{
		F_{\mathbf{s}}D(\Phi_{A_n^{(1)}}^\vee)
	}{
		F_1D(\Phi_{A_n^{(1)}}^\vee)
	}
	=
	\frac{(t^L;t^L)_\infty^n}
	{(t^{n+1};t^{n+1})_\infty^n}
	\prod_{1\le a\le b\le n}
	\frac{
		(t^{M_{a,b}},t^{L-M_{a,b}};t^L)_\infty
	}{
		(t^{b-a+1},t^{n-b+a};t^{n+1})_\infty
	}.
\end{equation}

\vskip 2mm

\paragraph{\bf Type \(A_{2n}^{(2)}\).}
\begin{equation}
	\label{eq:product-side-A2n2}
	\begin{aligned}
		\frac{
			F_{\mathbf{s}}D(\Phi_{A_{2n}^{(2)}}^\vee)
		}{
			F_1D(\Phi_{A_{2n}^{(2)}}^\vee)
		}
		&=
		\frac{(t^L;t^L)_\infty^n}
		{(t^{2n+1};t^{2n+1})_\infty^n}
		\prod_{a=1}^{n}
		\frac{
			(t^{M_{a,n}},t^{L-M_{a,n}};t^L)_\infty
		}{
			(t^{n-a+1},t^{n+a};t^{2n+1})_\infty
		}
		\prod_{1\le a<b\le n}
		\frac{
			(t^{M_{a,b-1}},t^{L-M_{a,b-1}};t^L)_\infty
		}{
			(t^{b-a},t^{2n+1-b+a};t^{2n+1})_\infty
		}
		\\
		&\quad\times
		\prod_{1\le a<b\le n}
		\frac{
			(t^{M_{a,b-1}+2M_{b,n}},
			t^{L-M_{a,b-1}-2M_{b,n}};t^L)_\infty
		}{
			(t^{2n-a-b+2},t^{a+b-1};t^{2n+1})_\infty
		}
		\prod_{a=1}^{n}
		\frac{
			(t^{L+2M_{a,n}},t^{L-2M_{a,n}};t^{2L})_\infty
		}{
			(t^{4n-2a+3},t^{2a-1};t^{4n+2})_\infty
		}.
	\end{aligned}
\end{equation}

\vskip 2mm

\paragraph{\bf Type \(D_{n+1}^{(2)}\).}
\begin{equation}
	\label{eq:product-side-Dn12}
	\begin{aligned}
		\frac{
			F_{\mathbf{s}}D(\Phi_{D_{n+1}^{(2)}}^\vee)
		}{
			F_1D(\Phi_{D_{n+1}^{(2)}}^\vee)
		}
		&=
		\frac{(t^L;t^L)_\infty^n}
		{(t^{2n};t^{2n})_\infty^n}
		\prod_{1\le a<b\le n}
		\frac{
			(t^{M_{a,b-1}},t^{L-M_{a,b-1}};t^L)_\infty
		}{
			(t^{b-a},t^{2n-b+a};t^{2n})_\infty
		}
		\\
		&\quad\times
		\prod_{1\le a<b\le n}
		\frac{
			(t^{M_{a,b-1}+2M_{b,n}-(m_n+1)},
			t^{L-M_{a,b-1}-2M_{b,n}+(m_n+1)};t^L)_\infty
		}{
			(t^{2n-a-b+1},t^{a+b-1};t^{2n})_\infty
		}
		\\
		&\quad\times
		\prod_{a=1}^{n}
		\frac{
			(t^{2M_{a,n}-(m_n+1)},
			t^{L-2M_{a,n}+(m_n+1)};t^L)_\infty
		}{
			(t^{2n-2a+1},t^{2a-1};t^{2n})_\infty
		}.
	\end{aligned}
\end{equation}

\vskip 2mm

\paragraph{\bf Type \(C_n^{(1)}\).}
\begin{equation}
	\label{eq:product-side-Cn1}
	\begin{aligned}
		\frac{
			F_{\mathbf{s}}D(\Phi_{C_n^{(1)}}^\vee)
		}{
			F_1D(\Phi_{C_n^{(1)}}^\vee)
		}
		&=
		\frac{
			(t^{2L};t^{2L})_\infty^n
			(t^L;t^{2L})_\infty
		}{
			(t^{2n+2};t^{2n+2})_\infty^n
			(t^{n+1};t^{2n+2})_\infty
		}
		\prod_{a=1}^{n}
		\frac{
			(t^{M_{a,n}},t^{L-M_{a,n}};t^L)_\infty
		}{
			(t^{n-a+1},t^a;t^{n+1})_\infty
		}
		\\
		&\quad\times
		\prod_{1\le a<b\le n}
		\frac{
			(t^{M_{a,b-1}},t^{2L-M_{a,b-1}};t^{2L})_\infty
		}{
			(t^{b-a},t^{2n+2-b+a};t^{2n+2})_\infty
		}
		\\
		&\quad\times
		\prod_{1\le a<b\le n}
		\frac{
			(t^{M_{a,b-1}+2M_{b,n}},
			t^{2L-M_{a,b-1}-2M_{b,n}};t^{2L})_\infty
		}{
			(t^{2n-a-b+2},t^{a+b};t^{2n+2})_\infty
		}.
	\end{aligned}
\end{equation}

\vskip 2mm

\paragraph{\bf Type \(A_{2n-1}^{(2)}\).}
\begin{equation}
	\label{eq:product-side-A2n12}
	\begin{aligned}
		\frac{
			F_{\mathbf{s}}D(\Phi_{A_{2n-1}^{(2)}}^\vee)
		}{
			F_1D(\Phi_{A_{2n-1}^{(2)}}^\vee)
		}
		&=
		\frac{(t^L;t^L)_\infty^n}
		{(t^{2n};t^{2n})_\infty^n}
		\prod_{a=1}^{n}
		\frac{
			(t^{M_{a,n}},t^{L-M_{a,n}};t^L)_\infty
		}{
			(t^{n-a+1},t^{n+a-1};t^{2n})_\infty
		}
		\prod_{1\le a<b\le n}
		\frac{
			(t^{M_{a,b-1}},t^{L-M_{a,b-1}};t^L)_\infty
		}{
			(t^{b-a},t^{2n-b+a};t^{2n})_\infty
		}
		\\
		&\quad\times
		\prod_{1\le a<b\le n}
		\frac{
			(t^{M_{a,b-1}+2M_{b,n}},
			t^{L-M_{a,b-1}-2M_{b,n}};t^L)_\infty
		}{
			(t^{2n-a-b+2},t^{a+b-2};t^{2n})_\infty
		}.
	\end{aligned}
\end{equation}

\vskip 2mm

\paragraph{\bf Type \(D_n^{(1)}\).}
\begin{equation}
	\label{eq:product-side-Dn1}
	\begin{aligned}
		\frac{
			F_{\mathbf{s}}D(\Phi_{D_n^{(1)}}^\vee)
		}{
			F_1D(\Phi_{D_n^{(1)}}^\vee)
		}
		&=
		\frac{(t^L;t^L)_\infty^n}
		{(t^{2n-2};t^{2n-2})_\infty^n}
		\prod_{1\le a<b\le n}
		\frac{
			(t^{M_{a,b-1}},t^{L-M_{a,b-1}};t^L)_\infty
		}{
			(t^{b-a},t^{2n-2-b+a};t^{2n-2})_\infty
		}
		\\
		&\quad\times
		\prod_{a=1}^{n-1}
		\frac{
			(t^{M_{a,n}-(m_{n-1}+1)},
			t^{L-M_{a,n}+(m_{n-1}+1)};t^L)_\infty
		}{
			(t^{n-a},t^{n+a-2};t^{2n-2})_\infty
		}
		\\
		&\quad\times
		\prod_{1\le a<b\le n-1}
		\frac{
			(t^{M_{a,b-1}+2M_{b,n}-m_{n-1}-m_n-2},
			t^{L-M_{a,b-1}-2M_{b,n}+m_{n-1}+m_n+2};t^L)_\infty
		}{
			(t^{2n-a-b},t^{a+b-2};t^{2n-2})_\infty
		}.
	\end{aligned}
\end{equation}

\vskip 2mm

\paragraph{\bf Type \(B_n^{(1)}\).}
\begin{equation}
	\label{eq:product-side-Bn1}
	\begin{aligned}
		\frac{
			F_{\mathbf{s}}D(\Phi_{B_n^{(1)}}^\vee)
		}{
			F_1D(\Phi_{B_n^{(1)}}^\vee)
		}
		&=
		\frac{
			(t^{2L};t^{2L})_\infty^n
			(t^L;t^{2L})_\infty^{n-1}
		}{
			(t^{4n-2};t^{4n-2})_\infty^n
			(t^{2n-1};t^{4n-2})_\infty^{n-1}
		}
		\prod_{1\le a<b\le n}
		\frac{
			(t^{M_{a,b-1}},t^{L-M_{a,b-1}};t^L)_\infty
		}{
			(t^{b-a},t^{2n-1-b+a};t^{2n-1})_\infty
		}
		\\
		&\quad\times
		\prod_{1\le a<b\le n}
		\frac{
			(t^{M_{a,b-1}+2M_{b,n}-(m_n+1)},
			t^{L-M_{a,b-1}-2M_{b,n}+(m_n+1)};t^L)_\infty
		}{
			(t^{2n-a-b+1},t^{a+b-2};t^{2n-1})_\infty
		}
		\\
		&\quad\times
		\prod_{a=1}^{n}
		\frac{
			(t^{2M_{a,n}-(m_n+1)},
			t^{2L-2M_{a,n}+(m_n+1)};t^{2L})_\infty
		}{
			(t^{2n-2a+1},t^{2n+2a-3};t^{4n-2})_\infty
		}.
	\end{aligned}
\end{equation}

By Theorem~\ref{thm:principal-specialized-path-product-identity}
and Propositions~\ref{prop:adjacent-pair-degree-form}
and~\ref{prop:adjacent-energy-independent-of-k}, we obtain the
following corollary.

\begin{corollary}
	\label{cor:level-l-specialization}
	Let \(X\) be one of the seven affine types, and let
	\(B\) be the corresponding level-\(l\) perfect crystal. For $\lambda=\sum_{i=0}^{n}m_i\Lambda_i
	\in(P_{\mathrm{cl}}^+)_l$ and $\mathbf s=(m_0+1,\ldots,m_n+1)$,
	let \(\mathbf g_\lambda=(g_k^\lambda)_{k\ge0}\) be the corresponding
	ground-state path. Then
	\begin{equation}\label{eq:level-l-specialization}
		\sum_{\substack{(p_k)_{k\ge0}\\
				p_k\in B\\
				p_k=g_k^\lambda\text{ for }k\gg0}}
		t^{\sum_{k\ge0}(k+1)E_\lambda^X(p_k,p_{k+1})}
		=
		\frac{F_{\mathbf s}D(\Phi_X^\vee)}
		{F_1D(\Phi_X^\vee)}.
	\end{equation}
	Here \(E_\lambda^X\) is the \(k\)-independent adjacent contribution
	defined by \eqref{eq:ElambdaX}. The right-hand
	side is given in
	\eqref{eq:product-side-An1}--\eqref{eq:product-side-Bn1}.
\end{corollary}

\subsection{Level-two examples}
\label{subsec:level-two-examples}

In the low-rank examples below, we specialize
Corollary~\ref{cor:level-l-specialization} to level \(l=2\). For a
level-\(2\) weight \(\lambda\), let \(Z_\lambda^X(t)\) denote the
left-hand side of \eqref{eq:level-l-specialization}. The data determining
its sum side: the level-\(2\) perfect crystal
\(B=\{b_0,\ldots,b_N\}\), the ground-state path
\(\mathbf g_\lambda=(g_k^\lambda)_{k\ge0}\), and the matrix
\(\bigl(E_\lambda^X(b_i,b_j)\bigr)_{b_i,b_j\in B}\) are recorded,
respectively, in
Appendices~\ref{app:level-two-perfect-crystal-elements},
\ref{app:level-two-ground-state-paths}, and
\ref{app:level-two-degree-matrices}. The corresponding product side is
obtained by substituting the value of \(L\) from
Table~\ref{tab:dual-affine-root-systems-and-L} and the coefficients
\(m_i\) from Appendix~\ref{app:level-two-ground-state-paths} into the
appropriate formula in
\eqref{eq:product-side-An1}--\eqref{eq:product-side-Bn1}. We list only
the resulting identities with distinct product sides, grouping together
the weights that yield the same product.

	\medskip
	\noindent\textup{(1) Type \(A_1^{(1)}\).}
	\[
	Z_{\Lambda_0+\Lambda_1}^{A_1^{(1)}}(t)
	=
	\frac{(t^2;t^4)_\infty}{(t;t^2)_\infty^2},\quad Z_{2\Lambda_0}^{A_1^{(1)}}(t)
	=
	Z_{2\Lambda_1}^{A_1^{(1)}}(t)
	=
	\frac{1}{(t^2;t^4)_\infty(t;t^2)_\infty}.
	\]

	\medskip
	\noindent\textup{(2) Type \(A_2^{(2)}\).}
	\[
	Z_{2\Lambda_0}^{A_2^{(2)}}(t)
	=
	\frac{1}{
		(t^2,t^8;t^{10})_\infty
		(t,t^5;t^6)_\infty
	},\quad 
	Z_{\Lambda_1}^{A_2^{(2)}}(t)
	=
	\frac{1}{
		(t^4,t^6;t^{10})_\infty
		(t,t^5;t^6)_\infty
	}.
	\]
	
	\medskip
	\noindent\textup{(3) Type \(D_3^{(2)}\).}
	\[
	Z_{2\Lambda_0}^{D_3^{(2)}}(t)
	=
	Z_{\Lambda_1}^{D_3^{(2)}}(t)
	=
	Z_{2\Lambda_2}^{D_3^{(2)}}(t)
	=
	\frac{1}{(t,t^2,t^3,t^4,t^5;t^6)_\infty},\quad 
	Z_{\Lambda_0+\Lambda_2}^{D_3^{(2)}}(t)
	=
	\frac{1}{(t,t^5;t^6)_\infty^2(t^3;t^6)_\infty}.
	\]
	
	\medskip
	\noindent\textup{(4) Type \(C_2^{(1)}\).}
	\[
	Z_{2\Lambda_0}^{C_2^{(1)}}(t)
	=
	Z_{2\Lambda_2}^{C_2^{(1)}}(t)
	=
	\frac{1}{
		(t,t^2,t^3,t^4,t^6,t^7,t^8,t^9;t^{10})_\infty
		(t^5;t^{10})_\infty^2
	},
	\]
	\[
	Z_{\Lambda_0+\Lambda_1}^{C_2^{(1)}}(t)
	=
	Z_{\Lambda_1+\Lambda_2}^{C_2^{(1)}}(t)
	=
	\frac{1}{
		(t,t^3,t^5,t^7,t^9;t^{10})_\infty^2
	},
	\]
	\[
	Z_{\Lambda_0+\Lambda_2}^{C_2^{(1)}}(t)
	=
	\frac{1}{
		(t,t^4,t^6,t^9;t^{10})_\infty^2
		(t^3,t^7;t^{10})_\infty
	},\quad
	Z_{2\Lambda_1}^{C_2^{(1)}}(t)
	=
	\frac{1}{
		(t^2,t^3,t^7,t^8;t^{10})_\infty^2
		(t,t^9;t^{10})_\infty
	}.
	\]
	
	\medskip
	\noindent\textup{(5) Type \(A_5^{(2)}\).}
	\[
	Z_{2\Lambda_0}^{A_5^{(2)}}(t)
	=
	Z_{2\Lambda_1}^{A_5^{(2)}}(t)
	=
	\frac{1}{
		(t,t^2,t^3,t^4,t^5,t^6,t^7,
		t^9,t^{10},t^{11},t^{12},t^{13},t^{14},t^{15},
		t^{17},t^{18},t^{19},t^{20},t^{21},t^{22},t^{23};t^{24})_\infty
	},
	\]
	\[
	Z_{\Lambda_0+\Lambda_1}^{A_5^{(2)}}(t)
	=
	\frac{1}{
		(t,t^7,t^9,t^{15},t^{17},t^{23};t^{24})_\infty^2
		(t^3,t^4,t^5,t^{11},t^{12},t^{13},
		t^{19},t^{20},t^{21};t^{24})_\infty
	},
	\]
	\[
	Z_{\Lambda_2}^{A_5^{(2)}}(t)
	=
	\frac{
		(t^4,t^{12},t^{20};t^{24})_\infty
	}{
		(t^2,t^6,t^{10},t^{14},t^{18},t^{22};t^{24})_\infty^2
		(t,t^3,t^5,t^7,t^9,t^{11},t^{13},t^{15},
		t^{17},t^{19},t^{21},t^{23};t^{24})_\infty
	},
	\]
	\[
	Z_{\Lambda_3}^{A_5^{(2)}}(t)
	=
	\frac{1}{
		(t^3,t^5,t^{11},t^{13},t^{19},t^{21};t^{24})_\infty^2
		(t,t^4,t^7,t^9,t^{12},t^{15},t^{17},t^{20},t^{23};t^{24})_\infty
	}.
	\]
	
	\medskip
	\noindent\textup{(6) Type \(D_4^{(1)}\).}
	\[
	\begin{aligned}
		Z_{2\Lambda_0}^{D_4^{(1)}}(t)
		&=
		Z_{2\Lambda_1}^{D_4^{(1)}}(t)
		=
		Z_{2\Lambda_3}^{D_4^{(1)}}(t)
		=
		Z_{2\Lambda_4}^{D_4^{(1)}}(t)\\
		&=
		\frac{1}{
			(t,t^2,t^5,t^6,t^7,t^{10},t^{11},
			t^{13},t^{14},t^{17},t^{18},t^{19},t^{22},t^{23};t^{24})_\infty
			(t^3,t^4,t^9,t^{12},t^{15},t^{20},t^{21};t^{24})_\infty^2
		}.
	\end{aligned}
	\]
	\[
	\begin{aligned}
		Z_{\Lambda_0+\Lambda_1}^{D_4^{(1)}}(t)
		&=
		Z_{\Lambda_0+\Lambda_3}^{D_4^{(1)}}(t)
		=
		Z_{\Lambda_0+\Lambda_4}^{D_4^{(1)}}(t)
		=
		Z_{\Lambda_1+\Lambda_3}^{D_4^{(1)}}(t)=
		Z_{\Lambda_1+\Lambda_4}^{D_4^{(1)}}(t)
		=
		Z_{\Lambda_3+\Lambda_4}^{D_4^{(1)}}(t)\\
		&=
		\frac{1}{
			(t,t^5,t^7,t^{11},t^{13},t^{17},t^{19},t^{23};t^{24})_\infty^2
			(t^3,t^9,t^{15},t^{21};t^{24})_\infty^3
		}.
	\end{aligned}
	\]
	\[
	Z_{\Lambda_2}^{D_4^{(1)}}(t)
	=
	\frac{
		(t^4,t^{12},t^{20};t^{24})_\infty^2
	}{
		(t,t^5,t^7,t^{11},t^{13},t^{17},t^{19},t^{23};t^{24})_\infty
		(t^3,t^9,t^{15},t^{21};t^{24})_\infty^2
		(t^2,t^6,t^{10},t^{14},t^{18},t^{22};t^{24})_\infty^3
	}.
	\]
	
	\medskip
	\noindent\textup{(7) Type \(B_3^{(1)}\).}
	\[
	Z_{2\Lambda_0}^{B_3^{(1)}}(t)
	=
	Z_{2\Lambda_1}^{B_3^{(1)}}(t)
	=
	\frac{1}{
		(t,t^2,t^3,t^4,t^5,t^6,t^8,t^9,t^{10},t^{11},t^{12},t^{13};t^{14})_\infty
		(t^7;t^{14})_\infty^2
	},
	\]
	\[
	Z_{\Lambda_0+\Lambda_1}^{B_3^{(1)}}(t)
	=
	\frac{1}{
		(t,t^6,t^8,t^{13};t^{14})_\infty^2
		(t^3,t^4,t^5,t^9,t^{10},t^{11};t^{14})_\infty
	},
	\]
	\[
	Z_{\Lambda_0+\Lambda_3}^{B_3^{(1)}}(t)
	=
	Z_{\Lambda_1+\Lambda_3}^{B_3^{(1)}}(t)
	=
	\frac{1}{
		(t,t^3,t^5,t^7,t^9,t^{11},t^{13};t^{14})_\infty^2
	},
	\]
	\[
	Z_{2\Lambda_3}^{B_3^{(1)}}(t)
	=
	\frac{1}{
		(t^3,t^4,t^{10},t^{11};t^{14})_\infty^2
		(t,t^2,t^5,t^9,t^{12},t^{13};t^{14})_\infty
	},
	\]
	\[
	Z_{\Lambda_2}^{B_3^{(1)}}(t)
	=
	\frac{1}{
		(t^2,t^5,t^9,t^{12};t^{14})_\infty^2
		(t,t^3,t^6,t^8,t^{11},t^{13};t^{14})_\infty
	}.
	\]

\appendix
\section{Root data}
\label{app:root-data}

\begin{table}[H]
\caption{Dual affine root systems, their finite types, null roots, and the values of \(L\).}
\label{tab:dual-affine-root-systems-and-L}	
	\centering
	\renewcommand{\arraystretch}{1.35}
	\small
	\begin{tabular}{|c|c|c|c|c|}
		\hline
		\(X\) & \(\Phi_X^\vee\) & \(\Phi_{0,X}^\vee\) & \(\delta^\vee\) & \(L\) \\
		\hline
		\(A_n^{(1)}\)
		&
		\(A_n^{(1)}\)
		&
		\(A_n\)
		&
		\(\displaystyle \sum_{i=0}^{n}\alpha_i^\vee\)
		&
		\(l+n+1\)
		\\
		\hline
		\(A_{2n}^{(2)}\)
		&
		\(A_{2n}^{(2),\mathrm{tr}}\)
		&
		\(B_n\)
		&
		\makecell{\(\displaystyle
			\alpha_0^\vee+2\sum_{i=1}^{n}\alpha_i^\vee
			\)}
		&
		\(l+2n+1\)
		\\
		\hline
		\(D_{n+1}^{(2)}\)
		&
		\(C_n^{(1)}\)
		&
		\(C_n\)
		&
		\makecell{\(\displaystyle
			\alpha_0^\vee
			+2\sum_{i=1}^{n-1}\alpha_i^\vee
			+\alpha_n^\vee
			\)}
		&
		\(l+2n\)
		\\
		\hline
		\(C_n^{(1)}\)
		&
		\(D_{n+1}^{(2)}\)
		&
		\(B_n\)
		&
		\(\displaystyle \sum_{i=0}^{n}\alpha_i^\vee\)
		&
		\(l+n+1\)
		\\
		\hline
		\(A_{2n-1}^{(2)}\)
		&
		\(B_n^{(1)}\)
		&
		\(B_n\)
		&
		\makecell{\(\displaystyle
			\alpha_0^\vee+\alpha_1^\vee
			+2\sum_{i=2}^{n}\alpha_i^\vee
			\)}
		&
		\(l+2n\)
		\\
		\hline
		\(D_n^{(1)}\)
		&
		\(D_n^{(1)}\)
		&
		\(D_n\)
		&
		\makecell{\(\displaystyle
			\alpha_0^\vee+\alpha_1^\vee
			+2\sum_{i=2}^{n-2}\alpha_i^\vee
			\)\\
			\(\displaystyle
			+\alpha_{n-1}^\vee+\alpha_n^\vee
			\)}
		&
		\(l+2n-2\)
		\\
		\hline
		\(B_n^{(1)}\)
		&
		\(A_{2n-1}^{(2)}\)
		&
		\(C_n\)
		&
		\makecell{\(\displaystyle
			\alpha_0^\vee+\alpha_1^\vee
			+2\sum_{i=2}^{n-1}\alpha_i^\vee
			\)\\
			\(\displaystyle
			+\alpha_n^\vee
			\)}
		&
		\(l+2n-1\)
		\\
		\hline
	\end{tabular}
\end{table}
Here \(A_{2n}^{(2),\mathrm{tr}}\) denotes the affine root system associated with the transpose of the generalized Cartan matrix of type \(A_{2n}^{(2)}\).

We use the standard realizations of the finite root systems. For \(B_n\),
\[
\alpha_i^\vee=\epsilon_i-\epsilon_{i+1}\quad(1\le i<n),
\qquad
\alpha_n^\vee=\epsilon_n.
\]
For \(C_n\),
\[
\alpha_i^\vee=\epsilon_i-\epsilon_{i+1}\quad(1\le i<n),
\qquad
\alpha_n^\vee=2\epsilon_n.
\]
For \(D_n\),
\[
\alpha_i^\vee=\epsilon_i-\epsilon_{i+1}\quad(1\le i<n),
\qquad
\alpha_n^\vee=\epsilon_{n-1}+\epsilon_n.
\]
Throughout, an empty sum is understood to be \(0\).
All real roots have multiplicity one; the multiplicities of imaginary roots
are stated explicitly.

\paragraph{\bf Type \(A_n^{(1)}\).}
The positive imaginary roots are
\[
\Phi_{A_n^{(1)}}^{\vee,+,\mathrm{im}}
=
\{j\delta^\vee\mid j\ge 1\},
\]
each with multiplicity \(n\). The positive real roots are
\[
\Phi_{A_n^{(1)}}^{\vee,+,\mathrm{re}}
=
\{
\beta+j\delta^\vee,\ (j+1)\delta^\vee-\beta
\mid
\beta\in\Phi_{A_n}^{0,+},\ j\ge0
\},
\]
where
\[
\Phi_{A_n}^{0,+}
=
\left\{
\alpha_a^\vee+\alpha_{a+1}^\vee+\cdots+\alpha_b^\vee
\ \middle|\
1\le a\le b\le n
\right\}.
\]

\paragraph{\bf Type \(A_{2n}^{(2)}\).}
The positive imaginary roots are
\[
\Phi_{A_{2n}^{(2)}}^{\vee,+,\mathrm{im}}
=
\{j\delta^\vee\mid j\ge1\},
\]
each with multiplicity \(n\).

For the finite root system \(B_n\), let
\[
\Phi_{B_n,s}^{0,+}
=
\{
\sum_{i=a}^{n}\alpha_i^\vee
\mid
1\le a\le n
\}.
\]

\[
\Phi_{B_n,l}^{0,+}
=
\{
\sum_{i=a}^{b-1}\alpha_i^\vee
\mid
1\le a<b\le n
\}
\cup
\{
\sum_{i=a}^{b-1}\alpha_i^\vee
+
2\sum_{i=b}^{n}\alpha_i^\vee
\mid
1\le a<b\le n
\}.
\]
Then the positive real roots are
\[
\begin{aligned}
	\Phi_{A_{2n}^{(2)}}^{\vee,+,\mathrm{re}}
	=
	&\{
	\beta+j\delta^\vee,\ (j+1)\delta^\vee-\beta
	\mid
	\beta\in
	\Phi_{B_n,s}^{0,+}\cup\Phi_{B_n,l}^{0,+},
	\ j\ge0
	\}
	\\
	&\cup
	\{
	2\beta+(2j+1)\delta^\vee,\ (2j+1)\delta^\vee-2\beta
	\mid
	\beta\in\Phi_{B_n,s}^{0,+},
	\ j\ge0
	\}.
\end{aligned}
\]

\paragraph{\bf Type \(D_{n+1}^{(2)}\).}
The positive imaginary roots are
\[
\Phi_{D_{n+1}^{(2)}}^{\vee,+,\mathrm{im}}
=
\{j\delta^\vee\mid j\ge1\},
\]
each with multiplicity \(n\). Since the dual affine root system is untwisted,
the positive real roots are
\[
\Phi_{D_{n+1}^{(2)}}^{\vee,+,\mathrm{re}}
=
\{
\beta+j\delta^\vee,\ (j+1)\delta^\vee-\beta
\mid
\beta\in\Phi_{C_n}^{0,+},\ j\ge0
\},
\]
where  $\Phi_{C_n}^{0,+}
=
\Phi_{C_n,s}^{0,+}\cup\Phi_{C_n,l}^{0,+}$ and

\[
\Phi_{C_n,s}^{0,+}
=
\{
\sum_{i=a}^{b-1}\alpha_i^\vee
\mid
1\le a<b\le n
\}
\cup
\{
\sum_{i=a}^{b-1}\alpha_i^\vee
+
2\sum_{i=b}^{n-1}\alpha_i^\vee
+
\alpha_n^\vee
\mid
1\le a<b\le n
\},
\]

\[
\Phi_{C_n,l}^{0,+}
=
\{
2\sum_{i=a}^{n-1}\alpha_i^\vee+\alpha_n^\vee
\mid
1\le a\le n
\}.
\]

\paragraph{\bf Type \(C_n^{(1)}\).}
The positive imaginary roots are
\[
\Phi_{C_n^{(1)}}^{\vee,+,\mathrm{im}}
=
\{2j\delta^\vee\mid j\ge1\}
\cup
\{(2j+1)\delta^\vee\mid j\ge0\},
\]
where the roots \(2j\delta^\vee\) have multiplicity \(n\), and the roots
\((2j+1)\delta^\vee\) have multiplicity \(1\).

Then the positive real roots are
\[
\begin{aligned}
	\Phi_{C_n^{(1)}}^{\vee,+,\mathrm{re}}
	=
	&\{
	\beta+j\delta^\vee,\ (j+1)\delta^\vee-\beta
	\mid
	\beta\in\Phi_{B_n,s}^{0,+},\ j\ge0
	\}
	\\
	&\cup
	\{
	\beta+2j\delta^\vee,\ 2(j+1)\delta^\vee-\beta
	\mid
	\beta\in\Phi_{B_n,l}^{0,+},\ j\ge0
	\}.
\end{aligned}
\]

\paragraph{\bf Type \(A_{2n-1}^{(2)}\).}
The positive imaginary roots are
\[
\Phi_{A_{2n-1}^{(2)}}^{\vee,+,\mathrm{im}}
=
\{j\delta^\vee\mid j\ge1\},
\]
each with multiplicity \(n\).

Let $\Phi_{B_n}^{0,+}
=
\Phi_{B_n,s}^{0,+}\cup\Phi_{B_n,l}^{0,+}$.
Since the dual affine root system is untwisted, the positive real roots are
\[
\Phi_{A_{2n-1}^{(2)}}^{\vee,+,\mathrm{re}}
=
\{
\beta+j\delta^\vee,\ (j+1)\delta^\vee-\beta
\mid
\beta\in\Phi_{B_n}^{0,+},\ j\ge0
\}.
\]

\paragraph{\bf Type \(D_n^{(1)}\).}
The positive imaginary roots are
\[
\Phi_{D_n^{(1)}}^{\vee,+,\mathrm{im}}
=
\{j\delta^\vee\mid j\ge1\},
\]
each with multiplicity \(n\).

For the finite root system \(D_n\), let

\[
\begin{aligned}
	\Phi_{D_n}^{0,+}
	=
	&\{
	\sum_{i=a}^{b-1}\alpha_i^\vee
	\mid
	1\le a<b\le n
	\}
	\cup
	\{
	\sum_{i=a}^{n-2}\alpha_i^\vee+\alpha_n^\vee
	\mid
	1\le a<n
	\}
	\\
	&\cup
	\{
	\sum_{i=a}^{b-1}\alpha_i^\vee
	+
	2\sum_{i=b}^{n-2}\alpha_i^\vee
	+
	\alpha_{n-1}^\vee+\alpha_n^\vee
	\mid
	1\le a<b\le n-1
	\}.
\end{aligned}
\]

Since the dual affine root system is untwisted, the positive real roots are
\[
\Phi_{D_n^{(1)}}^{\vee,+,\mathrm{re}}
=
\{
\beta+j\delta^\vee,\ (j+1)\delta^\vee-\beta
\mid
\beta\in\Phi_{D_n}^{0,+},\ j\ge0
\}.
\]

\paragraph{\bf Type \(B_n^{(1)}\).}
The positive imaginary roots are
\[
\Phi_{B_n^{(1)}}^{\vee,+,\mathrm{im}}
=
\{2j\delta^\vee\mid j\ge1\}
\cup
\{(2j+1)\delta^\vee\mid j\ge0\},
\]
where the roots \(2j\delta^\vee\) have multiplicity \(n\), and the roots
\((2j+1)\delta^\vee\) have multiplicity \(n-1\).
The positive real roots are
\[
\begin{aligned}
	\Phi_{B_n^{(1)}}^{\vee,+,\mathrm{re}}
	=
	&\{
	\beta+j\delta^\vee,\ (j+1)\delta^\vee-\beta
	\mid
	\beta\in\Phi_{C_n,s}^{0,+},\ j\ge0
	\}
	\\
	&\cup
	\{
	\beta+2j\delta^\vee,\ 2(j+1)\delta^\vee-\beta
	\mid
	\beta\in\Phi_{C_n,l}^{0,+},\ j\ge0
	\}.
\end{aligned}
\]

\section{Elements of level-two perfect crystals}
\label{app:level-two-perfect-crystal-elements}

For \(A_1^{(1)}\) at level \(l=2\), set $B=\{b_0,b_1,b_2\}$, where
$$
b_0=(2,0),\ b_1=(1,1),\ b_2=(0,2).
$$

For \(A_2^{(2)}\) at level \(l=2\), set $B=\{b_0,b_1,\cdots, b_{5}\}$, where 
$$
b_0=(0\mid0),\ b_1=(1\mid0),\ b_2=(0\mid1),\ b_3=(2\mid0),\ b_4=(1\mid1),\ b_5=(0\mid2).
$$

For \(D_3^{(2)}\) at level \(l=2\), set $B=\{b_0,b_1,\cdots, b_{19}\}$, where
\[
\begin{array}{llll}
	b_0=(0,0\mid0\mid0,0)
	& b_5=(0,0\mid0\mid0,1)
	& b_{10}=(1,0\mid0\mid0,1)
	& b_{15}=(0,0\mid1\mid1,0)\\
	b_1=(1,0\mid0\mid0,0)
	& b_6=(2,0\mid0\mid0,0)
	& b_{11}=(0,2\mid0\mid0,0)
	& b_{16}=(0,0\mid1\mid0,1)\\
	b_2=(0,1\mid0\mid0,0)
	& b_7=(1,1\mid0\mid0,0)
	& b_{12}=(0,1\mid1\mid0,0)
	& b_{17}=(0,0\mid0\mid2,0)\\
	b_3=(0,0\mid1\mid0,0)
	& b_8=(1,0\mid1\mid0,0)
	& b_{13}=(0,1\mid0\mid1,0)
	& b_{18}=(0,0\mid0\mid1,1)\\
	b_4=(0,0\mid0\mid1,0)
	& b_9=(1,0\mid0\mid1,0)
	& b_{14}=(0,1\mid0\mid0,1)
	& b_{19}=(0,0\mid0\mid0,2).
\end{array}
\]

For \(C_2^{(1)}\) at level \(l=2\), set $B=\{b_0, b_1, \cdots, b_{45}\}$, where
\[
\begin{array}{llll}
	b_{0}=(0,0\mid0,0)      & b_{12}=(3,1\mid0,0)     & b_{24}=(1,1\mid2,0)     & b_{36}=(0,2\mid0,2)\\
	b_{1}=(2,0\mid0,0)      & b_{13}=(3,0\mid1,0)     & b_{25}=(1,1\mid1,1)     & b_{37}=(0,1\mid3,0)\\
	b_{2}=(1,1\mid0,0)      & b_{14}=(3,0\mid0,1)     & b_{26}=(1,1\mid0,2)     & b_{38}=(0,1\mid2,1)\\
	b_{3}=(1,0\mid1,0)      & b_{15}=(2,2\mid0,0)     & b_{27}=(1,0\mid3,0)     & b_{39}=(0,1\mid1,2)\\
	b_{4}=(1,0\mid0,1)      & b_{16}=(2,1\mid1,0)     & b_{28}=(1,0\mid2,1)     & b_{40}=(0,1\mid0,3)\\
	b_{5}=(0,2\mid0,0)      & b_{17}=(2,1\mid0,1)     & b_{29}=(1,0\mid1,2)     & b_{41}=(0,0\mid4,0)\\
	b_{6}=(0,1\mid1,0)      & b_{18}=(2,0\mid2,0)     & b_{30}=(1,0\mid0,3)     & b_{42}=(0,0\mid3,1)\\
	b_{7}=(0,1\mid0,1)      & b_{19}=(2,0\mid1,1)     & b_{31}=(0,4\mid0,0)     & b_{43}=(0,0\mid2,2)\\
	b_{8}=(0,0\mid2,0)      & b_{20}=(2,0\mid0,2)     & b_{32}=(0,3\mid1,0)     & b_{44}=(0,0\mid1,3)\\
	b_{9}=(0,0\mid1,1)      & b_{21}=(1,3\mid0,0)     & b_{33}=(0,3\mid0,1)     & b_{45}=(0,0\mid0,4)\\
	b_{10}=(0,0\mid0,2)     & b_{22}=(1,2\mid1,0)     & b_{34}=(0,2\mid2,0)     & \\
	b_{11}=(4,0\mid0,0)     & b_{23}=(1,2\mid0,1)     & b_{35}=(0,2\mid1,1)     &
\end{array}
\]

For \(A_5^{(2)}\) at level \(l=2\), set $B=\{b_0, b_1, \cdots, b_{20}\}$, where
\[
\begin{array}{lll}
	b_0=(2,0,0\mid0,0,0)      & b_1=(1,1,0\mid0,0,0)      & b_2=(1,0,1\mid0,0,0)\\
	b_3=(1,0,0\mid1,0,0)      & b_4=(1,0,0\mid0,1,0)      & b_5=(1,0,0\mid0,0,1)\\
	b_6=(0,2,0\mid0,0,0)      & b_7=(0,1,1\mid0,0,0)      & b_8=(0,1,0\mid1,0,0)\\
	b_9=(0,1,0\mid0,1,0)      & b_{10}=(0,1,0\mid0,0,1)   & b_{11}=(0,0,2\mid0,0,0)\\
	b_{12}=(0,0,1\mid1,0,0)   & b_{13}=(0,0,1\mid0,1,0)   & b_{14}=(0,0,1\mid0,0,1)\\
	b_{15}=(0,0,0\mid2,0,0)   & b_{16}=(0,0,0\mid1,1,0)   & b_{17}=(0,0,0\mid1,0,1)\\
	b_{18}=(0,0,0\mid0,2,0)   & b_{19}=(0,0,0\mid0,1,1)   & b_{20}=(0,0,0\mid0,0,2).
\end{array}
\]

For \(D_4^{(1)}\) at level \(l=2\),
let $B=\{b_0, b_1, \cdots, b_{34}\}$, where
\[
\begin{array}{lll}
	b_{0}=(2,0,0,0\mid0,0,0,0)       & b_{1}=(1,1,0,0\mid0,0,0,0)       & b_{2}=(1,0,1,0\mid0,0,0,0)\\
	b_{3}=(1,0,0,1\mid0,0,0,0)       & b_{4}=(1,0,0,0\mid1,0,0,0)       & b_{5}=(1,0,0,0\mid0,1,0,0)\\
	b_{6}=(1,0,0,0\mid0,0,1,0)       & b_{7}=(1,0,0,0\mid0,0,0,1)       & b_{8}=(0,2,0,0\mid0,0,0,0)\\
	b_{9}=(0,1,1,0\mid0,0,0,0)       & b_{10}=(0,1,0,1\mid0,0,0,0)      & b_{11}=(0,1,0,0\mid1,0,0,0)\\
	b_{12}=(0,1,0,0\mid0,1,0,0)      & b_{13}=(0,1,0,0\mid0,0,1,0)      & b_{14}=(0,1,0,0\mid0,0,0,1)\\
	b_{15}=(0,0,2,0\mid0,0,0,0)      & b_{16}=(0,0,1,1\mid0,0,0,0)      & b_{17}=(0,0,1,0\mid1,0,0,0)\\
	b_{18}=(0,0,1,0\mid0,1,0,0)      & b_{19}=(0,0,1,0\mid0,0,1,0)      & b_{20}=(0,0,1,0\mid0,0,0,1)\\
	b_{21}=(0,0,0,2\mid0,0,0,0)      & b_{22}=(0,0,0,1\mid0,1,0,0)      & b_{23}=(0,0,0,1\mid0,0,1,0)\\
	b_{24}=(0,0,0,1\mid0,0,0,1)      & b_{25}=(0,0,0,0\mid2,0,0,0)      & b_{26}=(0,0,0,0\mid1,1,0,0)\\
	b_{27}=(0,0,0,0\mid1,0,1,0)      & b_{28}=(0,0,0,0\mid1,0,0,1)      & b_{29}=(0,0,0,0\mid0,2,0,0)\\
	b_{30}=(0,0,0,0\mid0,1,1,0)      & b_{31}=(0,0,0,0\mid0,1,0,1)      & b_{32}=(0,0,0,0\mid0,0,2,0)\\
	b_{33}=(0,0,0,0\mid0,0,1,1)      & b_{34}=(0,0,0,0\mid0,0,0,2)      &
\end{array}
\]

For \(B_3^{(1)}\) at level \(l=2\), let
\(B=\{b_0,b_1,\ldots,b_{26}\}\), where
\[
\begin{array}{lll}
	b_0=(2,0,0\mid0\mid0,0,0)
	& b_1=(1,1,0\mid0\mid0,0,0)
	& b_2=(1,0,1\mid0\mid0,0,0)\\
	b_3=(1,0,0\mid0\mid1,0,0)
	& b_4=(1,0,0\mid0\mid0,1,0)
	& b_5=(1,0,0\mid0\mid0,0,1)\\
	b_6=(0,2,0\mid0\mid0,0,0)
	& b_7=(0,1,1\mid0\mid0,0,0)
	& b_8=(0,1,0\mid0\mid1,0,0)\\
	b_9=(0,1,0\mid0\mid0,1,0)
	& b_{10}=(0,1,0\mid0\mid0,0,1)
	& b_{11}=(0,0,2\mid0\mid0,0,0)\\
	b_{12}=(0,0,1\mid0\mid1,0,0)
	& b_{13}=(0,0,1\mid0\mid0,1,0)
	& b_{14}=(0,0,1\mid0\mid0,0,1)\\
	b_{15}=(0,0,0\mid0\mid2,0,0)
	& b_{16}=(0,0,0\mid0\mid1,1,0)
	& b_{17}=(0,0,0\mid0\mid1,0,1)\\
	b_{18}=(0,0,0\mid0\mid0,2,0)
	& b_{19}=(0,0,0\mid0\mid0,1,1)
	& b_{20}=(0,0,0\mid0\mid0,0,2)\\
	b_{21}=(1,0,0\mid1\mid0,0,0)
	& b_{22}=(0,1,0\mid1\mid0,0,0)
	& b_{23}=(0,0,1\mid1\mid0,0,0)\\
	b_{24}=(0,0,0\mid1\mid1,0,0)
	& b_{25}=(0,0,0\mid1\mid0,1,0)
	& b_{26}=(0,0,0\mid1\mid0,0,1).
\end{array}
\]

\vspace{-5pt}

\section{The level-two ground-state paths}
\label{app:level-two-ground-state-paths}
In the constant cases occurring below, the ground-state path is given by
\[
g_k^\lambda=b_\lambda=b^\lambda
\qquad(k\geq0).
\]
In the two-periodic cases, since \(g_0^\lambda=b_\lambda\), the ground-state path is given by
\[
g_{2r}^\lambda=b_\lambda,
\qquad
g_{2r+1}^\lambda=b^\lambda
\qquad(r\geq0).
\]

\begin{table}[H]
	\centering
	\caption{Ground-state conditions for the level-2 perfect crystal of type \(A_1^{(1)}\).}
	\label{tab:A11-level-two-ground-state-conditions}
	\begin{tabular}{|c|c|c|}
		\hline
		\(\lambda\) & \((m_0,m_1)\) & ground-state condition \\
		\hline
		\(2\Lambda_0\) & \((2,0)\) &
		\(g_{2r}^\lambda=b_0,\ g_{2r+1}^\lambda=b_2\) \\
		\hline
		\(\Lambda_0+\Lambda_1\) & \((1,1)\) &
		\(g_k^\lambda=b_1\) \\
		\hline
		\(2\Lambda_1\) & \((0,2)\) &
		\(g_{2r}^\lambda=b_2,\ g_{2r+1}^\lambda=b_0\) \\
		\hline
	\end{tabular}
\end{table}

\begin{table}[H]
	\centering	
	\caption{Ground-state conditions for the level-2 perfect crystal of type
		\(A_2^{(2)}\).}
	\label{tab:A22-level-two-ground-state-conditions}
	\begin{tabular}{|c|c|c|c|}
		\hline
		\(\lambda\) & \((m_0,m_1)\) & \(b_\lambda=b^\lambda\) &
		ground-state condition \\
		\hline
		\(2\Lambda_0\) & \((2,0)\) &
		\((0\mid0)=b_0\) &
		\(g_k^\lambda=b_0\) \\
		\hline
		\(\Lambda_1\) & \((0,1)\) &
		\((1\mid1)=b_4\) &
		\(g_k^\lambda=b_4\) \\
		\hline
	\end{tabular}
\end{table}

\begin{table}[H]
	\centering	
	\caption{Ground-state conditions for the level-2 perfect crystal of type
		\(D_3^{(2)}\).}
	\label{tab:D32-level-two-ground-state-conditions}
	\begin{tabular}{|c|c|c|c|}
		\hline
		\(\lambda\) & \((m_0,m_1,m_2)\) & \(b_\lambda=b^\lambda\) &
		ground-state condition \\
		\hline
		\(2\Lambda_0\) & \((2,0,0)\) &
		\((0,0\mid0\mid0,0)=b_0\) &
		\(g_k^\lambda=b_0\) \\
		\hline
		\(\Lambda_1\) & \((0,1,0)\) &
		\((1,0\mid0\mid0,1)=b_{10}\) &
		\(g_k^\lambda=b_{10}\) \\
		\hline
		\(\Lambda_0+\Lambda_2\) & \((1,0,1)\) &
		\((0,0\mid1\mid0,0)=b_3\) &
		\(g_k^\lambda=b_3\) \\
		\hline
		\(2\Lambda_2\) & \((0,0,2)\) &
		\((0,1\mid0\mid1,0)=b_{13}\) &
		\(g_k^\lambda=b_{13}\) \\
		\hline
	\end{tabular}
\end{table}

\begin{table}[H]
	\centering
\caption{Ground-state conditions for the level-2 perfect crystal of type
	\(C_2^{(1)}\).}
\label{tab:C21-level-two-ground-state-conditions}	
	\begin{tabular}{|c|c|c|c|}
		\hline
		\(\lambda\) & \((m_0,m_1,m_2)\) & \(b_\lambda=b^\lambda\) &
		ground-state condition \\
		\hline
		\(2\Lambda_0\) & \((2,0,0)\) &
		\((0,0\mid0,0)=b_0\) &
		\(g_k^\lambda=b_0\) \\
		\hline
		\(\Lambda_0+\Lambda_1\) & \((1,1,0)\) &
		\((1,0\mid0,1)=b_4\) &
		\(g_k^\lambda=b_4\) \\
		\hline
		\(\Lambda_0+\Lambda_2\) & \((1,0,1)\) &
		\((0,1\mid1,0)=b_6\) &
		\(g_k^\lambda=b_6\) \\
		\hline
		\(2\Lambda_1\) & \((0,2,0)\) &
		\((2,0\mid0,2)=b_{20}\) &
		\(g_k^\lambda=b_{20}\) \\
		\hline
		\(\Lambda_1+\Lambda_2\) & \((0,1,1)\) &
		\((1,1\mid1,1)=b_{25}\) &
		\(g_k^\lambda=b_{25}\) \\
		\hline
		\(2\Lambda_2\) & \((0,0,2)\) &
		\((0,2\mid2,0)=b_{34}\) &
		\(g_k^\lambda=b_{34}\) \\
		\hline
	\end{tabular}
\end{table}

\begin{table}[H]
	\centering	
	\caption{Ground-state conditions for the level-2 perfect crystal of type
		\(A_5^{(2)}\).}
	\label{tab:A52-level-two-ground-state-conditions}
	\resizebox{\textwidth}{!}{%
		\begin{tabular}{|c|c|c|c|c|}
			\hline
			\(\lambda\) & \((m_0,m_1,m_2,m_3)\) & \(b_\lambda\) & \(b^\lambda\) &
			ground-state condition \\
			\hline
			\(2\Lambda_0\) & \((2,0,0,0)\) &
			\((0,0,0\mid0,0,2)=b_{20}\) &
			\((2,0,0\mid0,0,0)=b_0\) &
			\(g_{2r}^\lambda=b_{20},\ g_{2r+1}^\lambda=b_0\) \\
			\hline
			\(\Lambda_0+\Lambda_1\) & \((1,1,0,0)\) &
			\((1,0,0\mid0,0,1)=b_5\) &
			\((1,0,0\mid0,0,1)=b_5\) &
			\(g_k^\lambda=b_5\) \\
			\hline
			\(2\Lambda_1\) & \((0,2,0,0)\) &
			\((2,0,0\mid0,0,0)=b_0\) &
			\((0,0,0\mid0,0,2)=b_{20}\) &
			\(g_{2r}^\lambda=b_0,\ g_{2r+1}^\lambda=b_{20}\) \\
			\hline
			\(\Lambda_2\) & \((0,0,1,0)\) &
			\((0,1,0\mid0,1,0)=b_9\) &
			\((0,1,0\mid0,1,0)=b_9\) &
			\(g_k^\lambda=b_9\) \\
			\hline
			\(\Lambda_3\) & \((0,0,0,1)\) &
			\((0,0,1\mid1,0,0)=b_{12}\) &
			\((0,0,1\mid1,0,0)=b_{12}\) &
			\(g_k^\lambda=b_{12}\) \\
			\hline
		\end{tabular}%
	}
\end{table}

\begin{table}[H]
	\centering
\caption{Ground-state conditions for the level-\(2\) perfect crystal of type \(D_4^{(1)}\).}
\label{tab:D41-level-two-ground-state-conditions}	
	\resizebox{\textwidth}{!}{%
		\begin{tabular}{|c|c|c|c|c|}
			\hline
			\(\lambda\) & \((m_0,m_1,m_2,m_3,m_4)\) &
			\(b_\lambda\) & \(b^\lambda\) &
			ground-state condition \\
			\hline
			\(2\Lambda_0\) & \((2,0,0,0,0)\) &
			\((0,0,0,0\mid0,0,0,2)=b_{34}\) &
			\((2,0,0,0\mid0,0,0,0)=b_0\) &
			\(g_{2r}^\lambda=b_{34},\ g_{2r+1}^\lambda=b_0\) \\
			\hline
			\(\Lambda_0+\Lambda_1\) & \((1,1,0,0,0)\) &
			\((1,0,0,0\mid0,0,0,1)=b_7\) &
			\((1,0,0,0\mid0,0,0,1)=b_7\) &
			\(g_k^\lambda=b_7\) \\
			\hline
			\(2\Lambda_1\) & \((0,2,0,0,0)\) &
			\((2,0,0,0\mid0,0,0,0)=b_0\) &
			\((0,0,0,0\mid0,0,0,2)=b_{34}\) &
			\(g_{2r}^\lambda=b_0,\ g_{2r+1}^\lambda=b_{34}\) \\
			\hline
			\(\Lambda_0+\Lambda_3\) & \((1,0,0,1,0)\) &
			\((0,0,0,0\mid1,0,0,1)=b_{28}\) &
			\((1,0,0,1\mid0,0,0,0)=b_3\) &
			\(g_{2r}^\lambda=b_{28},\ g_{2r+1}^\lambda=b_3\) \\
			\hline
			\(\Lambda_0+\Lambda_4\) & \((1,0,0,0,1)\) &
			\((0,0,0,1\mid0,0,0,1)=b_{24}\) &
			\((1,0,0,0\mid1,0,0,0)=b_4\) &
			\(g_{2r}^\lambda=b_{24},\ g_{2r+1}^\lambda=b_4\) \\
			\hline
			\(\Lambda_1+\Lambda_3\) & \((0,1,0,1,0)\) &
			\((1,0,0,0\mid1,0,0,0)=b_4\) &
			\((0,0,0,1\mid0,0,0,1)=b_{24}\) &
			\(g_{2r}^\lambda=b_4,\ g_{2r+1}^\lambda=b_{24}\) \\
			\hline
			\(\Lambda_1+\Lambda_4\) & \((0,1,0,0,1)\) &
			\((1,0,0,1\mid0,0,0,0)=b_3\) &
			\((0,0,0,0\mid1,0,0,1)=b_{28}\) &
			\(g_{2r}^\lambda=b_3,\ g_{2r+1}^\lambda=b_{28}\) \\
			\hline
			\(2\Lambda_3\) & \((0,0,0,2,0)\) &
			\((0,0,0,0\mid2,0,0,0)=b_{25}\) &
			\((0,0,0,2\mid0,0,0,0)=b_{21}\) &
			\(g_{2r}^\lambda=b_{25},\ g_{2r+1}^\lambda=b_{21}\) \\
			\hline
			\(\Lambda_3+\Lambda_4\) & \((0,0,0,1,1)\) &
			\((0,0,1,0\mid0,1,0,0)=b_{18}\) &
			\((0,0,1,0\mid0,1,0,0)=b_{18}\) &
			\(g_k^\lambda=b_{18}\) \\
			\hline
			\(2\Lambda_4\) & \((0,0,0,0,2)\) &
			\((0,0,0,2\mid0,0,0,0)=b_{21}\) &
			\((0,0,0,0\mid2,0,0,0)=b_{25}\) &
			\(g_{2r}^\lambda=b_{21},\ g_{2r+1}^\lambda=b_{25}\) \\
			\hline
			\(\Lambda_2\) & \((0,0,1,0,0)\) &
			\((0,1,0,0\mid0,0,1,0)=b_{13}\) &
			\((0,1,0,0\mid0,0,1,0)=b_{13}\) &
			\(g_k^\lambda=b_{13}\) \\
			\hline
		\end{tabular}%
	}
\end{table}

\begin{table}[H]
	\centering
\caption{Ground-state conditions for the level-2 perfect crystal of type \(B_3^{(1)}\).}
\label{tab:B31-level-two-ground-state-conditions}	
	\resizebox{\textwidth}{!}{%
		\begin{tabular}{|c|c|c|c|c|}
			\hline
			\(\lambda\) & \((m_0,m_1,m_2,m_3)\) &
			\(b_\lambda\) & \(b^\lambda\) &
			ground-state condition \\
			\hline
			\(2\Lambda_0\) & \((2,0,0,0)\) &
			\((0,0,0\mid0\mid0,0,2)=b_{20}\) &
			\((2,0,0\mid0\mid0,0,0)=b_0\) &
			\(g_{2r}^\lambda=b_{20},\
			g_{2r+1}^\lambda=b_0\) \\
			\hline
			\(\Lambda_0+\Lambda_1\) & \((1,1,0,0)\) &
			\((1,0,0\mid0\mid0,0,1)=b_5\) &
			\((1,0,0\mid0\mid0,0,1)=b_5\) &
			\(g_k^\lambda=b_5\) \\
			\hline
			\(2\Lambda_1\) & \((0,2,0,0)\) &
			\((2,0,0\mid0\mid0,0,0)=b_0\) &
			\((0,0,0\mid0\mid0,0,2)=b_{20}\) &
			\(g_{2r}^\lambda=b_0,\
			g_{2r+1}^\lambda=b_{20}\) \\
			\hline
			\(\Lambda_0+\Lambda_3\) & \((1,0,0,1)\) &
			\((0,0,0\mid1\mid0,0,1)=b_{26}\) &
			\((1,0,0\mid1\mid0,0,0)=b_{21}\) &
			\(g_{2r}^\lambda=b_{26},\
			g_{2r+1}^\lambda=b_{21}\) \\
			\hline
			\(\Lambda_1+\Lambda_3\) & \((0,1,0,1)\) &
			\((1,0,0\mid1\mid0,0,0)=b_{21}\) &
			\((0,0,0\mid1\mid0,0,1)=b_{26}\) &
			\(g_{2r}^\lambda=b_{21},\
			g_{2r+1}^\lambda=b_{26}\) \\
			\hline
			\(2\Lambda_3\) & \((0,0,0,2)\) &
			\((0,0,1\mid0\mid1,0,0)=b_{12}\) &
			\((0,0,1\mid0\mid1,0,0)=b_{12}\) &
			\(g_k^\lambda=b_{12}\) \\
			\hline
			\(\Lambda_2\) & \((0,0,1,0)\) &
			\((0,1,0\mid0\mid0,1,0)=b_9\) &
			\((0,1,0\mid0\mid0,1,0)=b_9\) &
			\(g_k^\lambda=b_9\) \\
			\hline
		\end{tabular}%
	}
\end{table}

\vspace{-7pt}

\section{Degree matrices for level-two examples}
\label{app:level-two-degree-matrices}

\begin{table}[H]
	\centering
	\caption{The adjacent-pair degree matrices \(E(b_i,b_j)\) at level \(l=2\).}
	\label{tab:level-two-degree-matrices}
	
	\resizebox{0.82\textwidth}{!}{%
		\begin{minipage}{\textwidth}
			\centering
			
			\begin{subtable}[b]{0.25\textwidth}
				\centering
				\begingroup
				\small
				\setlength{\tabcolsep}{6pt}
				\renewcommand{\arraystretch}{1.05}
				\begin{tabular}{|c|*{3}{c|}}
					\hline
					\diagbox{\(b_i\)}{\(b_j\)}
					& \(b_0\)& \(b_1\)& \(b_2\)
					\\
					\hline
					\(b_0\)&2&1&0\\
					\hline
					\(b_1\)&1&0&1\\
					\hline
					\(b_2\)&0&1&2\\
					\hline
				\end{tabular}
				\endgroup
				\caption{Type \(A_1^{(1)}\).}
				\label{tab:A11-level-two-degree-matrix}
			\end{subtable}
			\hspace{-0.5em}%
			\begin{subtable}[b]{0.56\textwidth}
				\centering
				\begingroup
				\small
				\setlength{\tabcolsep}{5pt}
				\renewcommand{\arraystretch}{1.05}
				\begin{tabular}{|c|*{6}{c|}}
					\hline
					\diagbox{\(b_i\)}{\(b_j\)}
					& \(b_0\)& \(b_1\)& \(b_2\)& \(b_3\)& \(b_4\)& \(b_5\)
					\\
					\hline
					\(b_0\)&0&2&1&4&3&2\\
					\hline
					\(b_1\)&1&3&2&5&4&3\\
					\hline
					\(b_2\)&2&1&3&3&2&4\\
					\hline
					\(b_3\)&2&4&3&6&5&4\\
					\hline
					\(b_4\)&3&2&4&1&0&5\\
					\hline
					\(b_5\)&4&3&5&2&1&6\\
					\hline
				\end{tabular}
				\endgroup
				\caption{Type \(A_2^{(2)}\).}
				\label{tab:A22-level-two-degree-matrix}
			\end{subtable}
			
		\end{minipage}%
	}
\end{table}

\begin{table}[H]
	\centering
	\caption{The adjacent-pair degree matrix \(E(b_i,b_j)\) for \(D_3^{(2)}\) at level \(l=2\).}
	\label{tab:D32-level-two-degree-matrix}
	\begingroup
	\setlength{\tabcolsep}{2pt}
	\renewcommand{\arraystretch}{0.8}
	\resizebox{0.6\textwidth}{!}{%
		\begin{tabular}{|c|*{20}{c|}}
			\hline
			\diagbox{\(b_i\)}{\(b_j\)}
			& \(b_{0}\)& \(b_{1}\)& \(b_{2}\)& \(b_{3}\)& \(b_{4}\)
			& \(b_{5}\)& \(b_{6}\)& \(b_{7}\)& \(b_{8}\)& \(b_{9}\)
			& \(b_{10}\)& \(b_{11}\)& \(b_{12}\)& \(b_{13}\)& \(b_{14}\)
			& \(b_{15}\)& \(b_{16}\)& \(b_{17}\)& \(b_{18}\)& \(b_{19}\)
			\\
			\hline
			\(b_{0}\)&0&5&4&3&2&1&10&9&8&7&6&8&7&6&5&5&4&4&3&2\\
			\hline
			\(b_{1}\)&1&6&5&4&3&2&11&10&9&8&7&9&8&7&6&6&5&5&4&3\\
			\hline
			\(b_{2}\)&2&1&6&5&4&3&6&5&4&3&2&10&9&8&7&7&6&6&5&4\\
			\hline
			\(b_{3}\)&3&2&1&0&5&4&7&6&5&4&3&5&4&3&2&2&1&7&6&5\\
			\hline
			\(b_{4}\)&4&3&2&1&6&5&8&7&6&5&4&6&5&4&3&3&2&8&7&6\\
			\hline
			\(b_{5}\)&5&4&3&2&1&6&9&8&7&6&5&7&6&5&4&4&3&3&2&7\\
			\hline
			\(b_{6}\)&2&7&6&5&4&3&12&11&10&9&8&10&9&8&7&7&6&6&5&4\\
			\hline
			\(b_{7}\)&3&2&7&6&5&4&7&6&5&4&3&11&10&9&8&8&7&7&6&5\\
			\hline
			\(b_{8}\)&4&3&2&1&6&5&8&7&6&5&4&6&5&4&3&3&2&8&7&6\\
			\hline
			\(b_{9}\)&5&4&3&2&7&6&9&8&7&6&5&7&6&5&4&4&3&9&8&7\\
			\hline
			\(b_{10}\)&6&5&4&3&2&7&4&3&2&1&0&8&7&6&5&5&4&4&3&8\\
			\hline
			\(b_{11}\)&4&3&8&7&6&5&2&7&6&5&4&12&11&10&9&9&8&8&7&6\\
			\hline
			\(b_{12}\)&5&4&3&2&7&6&3&2&1&6&5&7&6&5&4&4&3&9&8&7\\
			\hline
			\(b_{13}\)&6&5&4&3&8&7&4&3&2&7&6&2&1&0&5&5&4&10&9&8\\
			\hline
			\(b_{14}\)&7&6&5&4&3&8&5&4&3&2&1&9&8&7&6&6&5&5&4&9\\
			\hline
			\(b_{15}\)&7&6&5&4&9&8&5&4&3&8&7&3&2&1&6&6&5&11&10&9\\
			\hline
			\(b_{16}\)&8&7&6&5&4&9&6&5&4&3&2&4&3&2&7&1&6&6&5&10\\
			\hline
			\(b_{17}\)&8&7&6&5&10&9&6&5&4&9&8&4&3&2&7&7&6&12&11&10\\
			\hline
			\(b_{18}\)&9&8&7&6&5&10&7&6&5&4&3&5&4&3&8&2&7&7&6&11\\
			\hline
			\(b_{19}\)&10&9&8&7&6&11&8&7&6&5&4&6&5&4&9&3&8&2&7&12\\
			\hline
		\end{tabular}%
	}
	\endgroup
\end{table}

\begin{table}[H]
	\centering
	\caption{The adjacent-pair degree matrix \(E(b_i,b_j)\) for \(C_2^{(1)}\) at level \(2\).}
	\label{tab:C2-1-level-two-energy-matrix}
	\begingroup
	\tiny
	\setlength{\tabcolsep}{1.2pt}
	\renewcommand{\arraystretch}{0.75}
	\resizebox{\textwidth}{!}{%
		\begin{tabular}{|c|*{46}{c|}}
			\hline
			\diagbox{\(b_i\)}{\(b_j\)}
			& \(b_{0}\)& \(b_{1}\)& \(b_{2}\)& \(b_{3}\)& \(b_{4}\)& \(b_{5}\)& \(b_{6}\)& \(b_{7}\)& \(b_{8}\)& \(b_{9}\)
			& \(b_{10}\)& \(b_{11}\)& \(b_{12}\)& \(b_{13}\)& \(b_{14}\)& \(b_{15}\)& \(b_{16}\)& \(b_{17}\)& \(b_{18}\)& \(b_{19}\)
			& \(b_{20}\)& \(b_{21}\)& \(b_{22}\)& \(b_{23}\)& \(b_{24}\)& \(b_{25}\)& \(b_{26}\)& \(b_{27}\)& \(b_{28}\)& \(b_{29}\)
			& \(b_{30}\)& \(b_{31}\)& \(b_{32}\)& \(b_{33}\)& \(b_{34}\)& \(b_{35}\)& \(b_{36}\)& \(b_{37}\)& \(b_{38}\)& \(b_{39}\)
			& \(b_{40}\)& \(b_{41}\)& \(b_{42}\)& \(b_{43}\)& \(b_{44}\)& \(b_{45}\)\\
			\hline
			\(b_{0}\)&0&7&6&5&4&5&4&3&3&2&1&14&13&12&11&12&11&10&10&9&8&11&10&9&9&8&7&8&7&6&5&10&9&8&8&7&6&7&6&5&4&6&5&4&3&2\\
			\hline
			\(b_{1}\)&1&8&7&6&5&6&5&4&4&3&2&15&14&13&12&13&12&11&11&10&9&12&11&10&10&9&8&9&8&7&6&11&10&9&9&8&7&8&7&6&5&7&6&5&4&3\\
			\hline
			\(b_{2}\)&2&5&4&3&2&7&6&5&5&4&3&12&11&10&9&10&9&8&8&7&6&9&8&7&7&6&5&6&5&4&3&12&11&10&10&9&8&9&8&7&6&8&7&6&5&4\\
			\hline
			\(b_{3}\)&3&6&5&4&3&4&3&2&6&5&4&13&12&11&10&11&10&9&9&8&7&10&9&8&8&7&6&7&6&5&4&9&8&7&7&6&5&6&5&4&3&9&8&7&6&5\\
			\hline
			\(b_{4}\)&4&3&2&1&0&5&4&3&3&2&5&10&9&8&7&8&7&6&6&5&4&7&6&5&5&4&3&4&3&2&1&10&9&8&8&7&6&7&6&5&4&6&5&4&3&6\\
			\hline
			\(b_{5}\)&3&2&5&4&3&8&7&6&6&5&4&9&8&7&6&7&6&5&5&4&3&10&9&8&8&7&6&7&6&5&4&13&12&11&11&10&9&10&9&8&7&9&8&7&6&5\\
			\hline
			\(b_{6}\)&4&3&2&5&4&1&0&3&7&6&5&10&9&8&7&8&7&6&6&5&4&7&6&5&5&4&3&8&7&6&5&6&5&4&4&3&2&3&2&1&4&10&9&8&7&6\\
			\hline
			\(b_{7}\)&5&4&3&2&1&6&5&4&4&3&6&11&10&9&8&9&8&7&7&6&5&8&7&6&6&5&4&5&4&3&2&11&10&9&9&8&7&8&7&6&5&7&6&5&4&7\\
			\hline
			\(b_{8}\)&5&4&3&6&5&2&1&4&8&7&6&11&10&9&8&9&8&7&7&6&5&8&7&6&6&5&4&9&8&7&6&7&6&5&5&4&3&4&3&2&5&11&10&9&8&7\\
			\hline
			\(b_{9}\)&6&5&4&3&2&3&2&5&5&4&7&12&11&10&9&10&9&8&8&7&6&9&8&7&7&6&5&6&5&4&3&8&7&6&6&5&4&5&4&3&6&8&7&6&5&8\\
			\hline
			\(b_{10}\)&7&6&5&4&3&4&3&6&2&5&8&13&12&11&10&11&10&9&9&8&7&10&9&8&8&7&6&7&6&5&4&9&8&7&7&6&5&6&5&4&7&5&4&3&6&9\\
			\hline
			\(b_{11}\)&2&9&8&7&6&7&6&5&5&4&3&16&15&14&13&14&13&12&12&11&10&13&12&11&11&10&9&10&9&8&7&12&11&10&10&9&8&9&8&7&6&8&7&6&5&4\\
			\hline
			\(b_{12}\)&3&6&5&4&3&8&7&6&6&5&4&13&12&11&10&11&10&9&9&8&7&10&9&8&8&7&6&7&6&5&4&13&12&11&11&10&9&10&9&8&7&9&8&7&6&5\\
			\hline
			\(b_{13}\)&4&7&6&5&4&5&4&3&7&6&5&14&13&12&11&12&11&10&10&9&8&11&10&9&9&8&7&8&7&6&5&10&9&8&8&7&6&7&6&5&4&10&9&8&7&6\\
			\hline
			\(b_{14}\)&5&4&3&2&1&6&5&4&4&3&6&11&10&9&8&9&8&7&7&6&5&8&7&6&6&5&4&5&4&3&2&11&10&9&9&8&7&8&7&6&5&7&6&5&4&7\\
			\hline
			\(b_{15}\)&4&3&6&5&4&9&8&7&7&6&5&10&9&8&7&8&7&6&6&5&4&11&10&9&9&8&7&8&7&6&5&14&13&12&12&11&10&11&10&9&8&10&9&8&7&6\\
			\hline
			\(b_{16}\)&5&4&3&6&5&2&1&4&8&7&6&11&10&9&8&9&8&7&7&6&5&8&7&6&6&5&4&9&8&7&6&7&6&5&5&4&3&4&3&2&5&11&10&9&8&7\\
			\hline
			\(b_{17}\)&6&5&4&3&2&7&6&5&5&4&7&8&7&6&5&6&5&4&4&3&2&9&8&7&7&6&5&6&5&4&3&12&11&10&10&9&8&9&8&7&6&8&7&6&5&8\\
			\hline
			\(b_{18}\)&6&5&4&7&6&3&2&5&9&8&7&12&11&10&9&10&9&8&8&7&6&9&8&7&7&6&5&10&9&8&7&8&7&6&6&5&4&5&4&3&6&12&11&10&9&8\\
			\hline
			\(b_{19}\)&7&6&5&4&3&4&3&6&6&5&8&9&8&7&6&7&6&5&5&4&3&6&5&4&4&3&2&7&6&5&4&9&8&7&7&6&5&6&5&4&7&9&8&7&6&9\\
			\hline
			\(b_{20}\)&8&7&6&5&4&5&4&7&3&6&9&6&5&4&3&4&3&2&2&1&0&7&6&5&5&4&3&4&3&2&5&10&9&8&8&7&6&7&6&5&8&6&5&4&7&10\\
			\hline
			\(b_{21}\)&5&4&7&6&5&10&9&8&8&7&6&7&6&5&4&9&8&7&7&6&5&12&11&10&10&9&8&9&8&7&6&15&14&13&13&12&11&12&11&10&9&11&10&9&8&7\\
			\hline
			\(b_{22}\)&6&5&4&7&6&3&2&5&9&8&7&8&7&6&5&6&5&4&8&7&6&5&4&3&3&2&5&10&9&8&7&8&7&6&6&5&4&5&4&3&6&12&11&10&9&8\\
			\hline
			\(b_{23}\)&7&6&5&4&3&8&7&6&6&5&8&5&4&3&2&7&6&5&5&4&3&10&9&8&8&7&6&7&6&5&4&13&12&11&11&10&9&10&9&8&7&9&8&7&6&9\\
			\hline
			\(b_{24}\)&7&6&5&8&7&4&3&6&10&9&8&9&8&7&6&7&6&5&9&8&7&6&5&4&4&3&6&11&10&9&8&5&4&3&3&2&5&6&5&4&7&13&12&11&10&9\\
			\hline
			\(b_{25}\)&8&7&6&5&4&5&4&7&7&6&9&6&5&4&3&4&3&2&6&5&4&3&2&1&1&0&3&8&7&6&5&6&5&4&4&3&6&3&2&5&8&10&9&8&7&10\\
			\hline
			\(b_{26}\)&9&8&7&6&5&6&5&8&4&7&10&7&6&5&4&5&4&3&3&2&1&8&7&6&6&5&4&5&4&3&6&11&10&9&9&8&7&8&7&6&9&7&6&5&8&11\\
			\hline
			\(b_{27}\)&8&7&6&9&8&5&4&7&11&10&9&10&9&8&7&8&7&6&10&9&8&7&6&5&5&4&7&12&11&10&9&6&5&4&4&3&6&7&6&5&8&14&13&12&11&10\\
			\hline
			\(b_{28}\)&9&8&7&6&5&6&5&8&8&7&10&7&6&5&4&5&4&3&7&6&5&4&3&2&2&1&4&9&8&7&6&7&6&5&5&4&7&4&3&6&9&11&10&9&8&11\\
			\hline
			\(b_{29}\)&10&9&8&7&6&7&6&9&5&8&11&8&7&6&5&6&5&4&4&3&2&5&4&3&3&2&5&6&5&4&7&8&7&6&6&5&8&5&4&7&10&8&7&6&9&12\\
			\hline
			\(b_{30}\)&11&10&9&8&7&8&7&10&6&9&12&9&8&7&6&7&6&5&5&4&3&6&5&4&4&3&6&3&2&5&8&9&8&7&7&6&9&6&5&8&11&5&4&7&10&13\\
			\hline
			\(b_{31}\)&6&5&8&7&6&11&10&9&9&8&7&4&7&6&5&10&9&8&8&7&6&13&12&11&11&10&9&10&9&8&7&16&15&14&14&13&12&13&12&11&10&12&11&10&9&8\\
			\hline
			\(b_{32}\)&7&6&5&8&7&4&3&6&10&9&8&5&4&7&6&3&2&5&9&8&7&6&5&4&4&3&6&11&10&9&8&9&8&7&7&6&5&6&5&4&7&13&12&11&10&9\\
			\hline
			\(b_{33}\)&8&7&6&5&4&9&8&7&7&6&9&6&5&4&3&8&7&6&6&5&4&11&10&9&9&8&7&8&7&6&5&14&13&12&12&11&10&11&10&9&8&10&9&8&7&10\\
			\hline
			\(b_{34}\)&8&7&6&9&8&5&4&7&11&10&9&6&5&8&7&4&3&6&10&9&8&3&2&5&5&4&7&12&11&10&9&2&1&4&0&3&6&7&6&5&8&14&13&12&11&10\\
			\hline
			\(b_{35}\)&9&8&7&6&5&6&5&8&8&7&10&7&6&5&4&5&4&3&7&6&5&4&3&2&2&1&4&9&8&7&6&7&6&5&5&4&7&4&3&6&9&11&10&9&8&11\\
			\hline
			\(b_{36}\)&10&9&8&7&6&7&6&9&5&8&11&8&7&6&5&6&5&4&4&3&2&9&8&7&7&6&5&6&5&4&7&12&11&10&10&9&8&9&8&7&10&8&7&6&9&12\\
			\hline
			\(b_{37}\)&9&8&7&10&9&6&5&8&12&11&10&7&6&9&8&5&4&7&11&10&9&4&3&6&6&5&8&13&12&11&10&3&2&5&1&4&7&8&7&6&9&15&14&13&12&11\\
			\hline
			\(b_{38}\)&10&9&8&7&6&7&6&9&9&8&11&8&7&6&5&6&5&4&8&7&6&5&4&3&3&2&5&10&9&8&7&4&3&6&2&5&8&5&4&7&10&12&11&10&9&12\\
			\hline
			\(b_{39}\)&11&10&9&8&7&8&7&10&6&9&12&9&8&7&6&7&6&5&5&4&3&6&5&4&4&3&6&7&6&5&8&5&4&7&3&6&9&2&5&8&11&9&8&7&10&13\\
			\hline
			\(b_{40}\)&12&11&10&9&8&9&8&11&7&10&13&10&9&8&7&8&7&6&6&5&4&7&6&5&5&4&7&4&3&6&9&10&9&8&8&7&10&7&6&9&12&6&5&8&11&14\\
			\hline
			\(b_{41}\)&10&9&8&11&10&7&6&9&13&12&11&8&7&10&9&6&5&8&12&11&10&5&4&7&7&6&9&14&13&12&11&4&3&6&2&5&8&9&8&7&10&16&15&14&13&12\\
			\hline
			\(b_{42}\)&11&10&9&8&7&8&7&10&10&9&12&9&8&7&6&7&6&5&9&8&7&6&5&4&4&3&6&11&10&9&8&5&4&7&3&6&9&6&5&8&11&13&12&11&10&13\\
			\hline
			\(b_{43}\)&12&11&10&9&8&9&8&11&7&10&13&10&9&8&7&8&7&6&6&5&4&7&6&5&5&4&7&8&7&6&9&6&5&8&4&7&10&3&6&9&12&10&9&8&11&14\\
			\hline
			\(b_{44}\)&13&12&11&10&9&10&9&12&8&11&14&11&10&9&8&9&8&7&7&6&5&8&7&6&6&5&8&5&4&7&10&7&6&9&5&8&11&4&7&10&13&7&6&9&12&15\\
			\hline
			\(b_{45}\)&14&13&12&11&10&11&10&13&9&12&15&12&11&10&9&10&9&8&8&7&6&9&8&7&7&6&9&6&5&8&11&8&7&10&6&9&12&5&8&11&14&4&7&10&13&16\\
			\hline
		\end{tabular}%
	}
	\endgroup
\end{table}

\begin{table}[H]
	\centering
	\caption{The adjacent-pair degree matrix \(E(b_i,b_j)\) for \(A_5^{(2)}\) at level \(2\).}
	\label{tab:A5-2-level-two-energy-matrix}
	\begingroup
	\setlength{\tabcolsep}{2pt}
	\renewcommand{\arraystretch}{0.8}
	\resizebox{0.5\textwidth}{!}{%
		\begin{tabular}{|c|*{21}{c|}}
			\hline
			\diagbox{\(b_i\)}{\(b_j\)}
			& \(b_{0}\)& \(b_{1}\)& \(b_{2}\)& \(b_{3}\)& \(b_{4}\)& \(b_{5}\)
			& \(b_{6}\)& \(b_{7}\)& \(b_{8}\)& \(b_{9}\)& \(b_{10}\)
			& \(b_{11}\)& \(b_{12}\)& \(b_{13}\)& \(b_{14}\)& \(b_{15}\)
			& \(b_{16}\)& \(b_{17}\)& \(b_{18}\)& \(b_{19}\)& \(b_{20}\)
			\\
			\hline
			\(b_{0}\)&10&9&8&7&6&5&8&7&6&5&4&6&5&4&3&4&3&2&2&1&0\\
			\hline
			\(b_{1}\)&6&5&4&3&2&1&9&8&7&6&5&7&6&5&4&5&4&3&3&2&1\\
			\hline
			\(b_{2}\)&7&6&5&4&3&2&5&4&3&2&1&8&7&6&5&6&5&4&4&3&2\\
			\hline
			\(b_{3}\)&8&7&6&5&4&3&6&5&4&3&2&4&3&2&1&7&6&5&5&4&3\\
			\hline
			\(b_{4}\)&9&8&7&6&5&4&7&6&5&4&3&5&4&3&2&3&2&1&6&5&4\\
			\hline
			\(b_{5}\)&5&4&3&2&1&0&8&7&6&5&4&6&5&4&3&4&3&2&2&1&5\\
			\hline
			\(b_{6}\)&2&6&5&4&3&2&10&9&8&7&6&8&7&6&5&6&5&4&4&3&2\\
			\hline
			\(b_{7}\)&3&2&6&5&4&3&6&5&4&3&2&9&8&7&6&7&6&5&5&4&3\\
			\hline
			\(b_{8}\)&4&3&2&6&5&4&7&6&5&4&3&5&4&3&2&8&7&6&6&5&4\\
			\hline
			\(b_{9}\)&5&4&3&2&6&5&3&2&1&0&4&6&5&4&3&4&3&2&7&6&5\\
			\hline
			\(b_{10}\)&1&5&4&3&2&1&9&8&7&6&5&7&6&5&4&5&4&3&3&2&6\\
			\hline
			\(b_{11}\)&4&3&7&6&5&4&2&6&5&4&3&10&9&8&7&8&7&6&6&5&4\\
			\hline
			\(b_{12}\)&5&4&3&7&6&5&3&2&6&5&4&1&0&4&3&9&8&7&7&6&5\\
			\hline
			\(b_{13}\)&6&5&4&3&7&6&4&3&2&1&5&7&6&5&4&5&4&3&8&7&6\\
			\hline
			\(b_{14}\)&2&1&5&4&3&2&5&4&3&2&6&8&7&6&5&6&5&4&4&3&7\\
			\hline
			\(b_{15}\)&6&5&4&8&7&6&4&3&7&6&5&2&1&5&4&10&9&8&8&7&6\\
			\hline
			\(b_{16}\)&7&6&5&4&8&7&5&4&3&2&6&3&2&6&5&6&5&4&9&8&7\\
			\hline
			\(b_{17}\)&3&2&1&5&4&3&6&5&4&3&7&4&3&2&6&7&6&5&5&4&8\\
			\hline
			\(b_{18}\)&8&7&6&5&9&8&6&5&4&3&7&4&3&7&6&2&6&5&10&9&8\\
			\hline
			\(b_{19}\)&4&3&2&1&5&4&7&6&5&4&8&5&4&3&7&3&2&6&6&5&9\\
			\hline
			\(b_{20}\)&0&4&3&2&1&5&8&7&6&5&9&6&5&4&8&4&3&7&2&6&10\\
			\hline
		\end{tabular}%
	}
	\endgroup
\end{table}

\begin{table}[H]
	\centering
	\caption{The adjacent-pair degree matrix \(E(b_i,b_j)\) for \(D_4^{(1)}\) at level \(2\).}
	\label{tab:D4-1-level-two-energy-matrix}
	\begingroup
	\tiny
	\setlength{\tabcolsep}{1.2pt}
	\renewcommand{\arraystretch}{0.75}
	\resizebox{0.8\textwidth}{!}{%
		\begin{tabular}{|c|*{35}{c|}}
			\hline
			\diagbox{\(b_i\)}{\(b_j\)}
			& \(b_{0}\)& \(b_{1}\)& \(b_{2}\)& \(b_{3}\)& \(b_{4}\)
			& \(b_{5}\)& \(b_{6}\)& \(b_{7}\)& \(b_{8}\)& \(b_{9}\)
			& \(b_{10}\)& \(b_{11}\)& \(b_{12}\)& \(b_{13}\)& \(b_{14}\)
			& \(b_{15}\)& \(b_{16}\)& \(b_{17}\)& \(b_{18}\)& \(b_{19}\)
			& \(b_{20}\)& \(b_{21}\)& \(b_{22}\)& \(b_{23}\)& \(b_{24}\)
			& \(b_{25}\)& \(b_{26}\)& \(b_{27}\)& \(b_{28}\)& \(b_{29}\)
			& \(b_{30}\)& \(b_{31}\)& \(b_{32}\)& \(b_{33}\)& \(b_{34}\)
			\\
			\hline
			\(b_{0}\)&12&11&10&9&9&8&7&6&10&9&8&8&7&6&5&8&7&7&6&5&4&6&5&4&3&6&5&4&3&4&3&2&2&1&0\\
			\hline
			\(b_{1}\)&7&6&5&4&4&3&2&1&11&10&9&9&8&7&6&9&8&8&7&6&5&7&6&5&4&7&6&5&4&5&4&3&3&2&1\\
			\hline
			\(b_{2}\)&8&7&6&5&5&4&3&2&6&5&4&4&3&2&1&10&9&9&8&7&6&8&7&6&5&8&7&6&5&6&5&4&4&3&2\\
			\hline
			\(b_{3}\)&9&8&7&6&6&5&4&3&7&6&5&5&4&3&2&5&4&4&3&2&1&9&8&7&6&3&2&1&0&7&6&5&5&4&3\\
			\hline
			\(b_{4}\)&9&8&7&6&6&5&4&3&7&6&5&5&4&3&2&5&4&4&3&2&1&3&2&1&0&9&8&7&6&7&6&5&5&4&3\\
			\hline
			\(b_{5}\)&10&9&8&7&7&6&5&4&8&7&6&6&5&4&3&6&5&5&4&3&2&4&3&2&1&4&3&2&1&8&7&6&6&5&4\\
			\hline
			\(b_{6}\)&11&10&9&8&8&7&6&5&9&8&7&7&6&5&4&7&6&6&5&4&3&5&4&3&2&5&4&3&2&3&2&1&7&6&5\\
			\hline
			\(b_{7}\)&6&5&4&3&3&2&1&0&10&9&8&8&7&6&5&8&7&7&6&5&4&6&5&4&3&6&5&4&3&4&3&2&2&1&6\\
			\hline
			\(b_{8}\)&2&7&6&5&5&4&3&2&12&11&10&10&9&8&7&10&9&9&8&7&6&8&7&6&5&8&7&6&5&6&5&4&4&3&2\\
			\hline
			\(b_{9}\)&3&2&7&6&6&5&4&3&7&6&5&5&4&3&2&11&10&10&9&8&7&9&8&7&6&9&8&7&6&7&6&5&5&4&3\\
			\hline
			\(b_{10}\)&4&3&2&7&1&6&5&4&8&7&6&6&5&4&3&6&5&5&4&3&2&10&9&8&7&4&3&2&1&8&7&6&6&5&4\\
			\hline
			\(b_{11}\)&4&3&2&1&7&6&5&4&8&7&6&6&5&4&3&6&5&5&4&3&2&4&3&2&1&10&9&8&7&8&7&6&6&5&4\\
			\hline
			\(b_{12}\)&5&4&3&2&2&7&6&5&9&8&7&7&6&5&4&7&6&6&5&4&3&5&4&3&2&5&4&3&2&9&8&7&7&6&5\\
			\hline
			\(b_{13}\)&6&5&4&3&3&2&7&6&4&3&2&2&1&0&5&8&7&7&6&5&4&6&5&4&3&6&5&4&3&4&3&2&8&7&6\\
			\hline
			\(b_{14}\)&1&6&5&4&4&3&2&1&11&10&9&9&8&7&6&9&8&8&7&6&5&7&6&5&4&7&6&5&4&5&4&3&3&2&7\\
			\hline
			\(b_{15}\)&4&3&8&7&7&6&5&4&2&7&6&6&5&4&3&12&11&11&10&9&8&10&9&8&7&10&9&8&7&8&7&6&6&5&4\\
			\hline
			\(b_{16}\)&5&4&3&8&2&7&6&5&3&2&7&1&6&5&4&7&6&6&5&4&3&11&10&9&8&5&4&3&2&9&8&7&7&6&5\\
			\hline
			\(b_{17}\)&5&4&3&2&8&7&6&5&3&2&1&7&6&5&4&7&6&6&5&4&3&5&4&3&2&11&10&9&8&9&8&7&7&6&5\\
			\hline
			\(b_{18}\)&6&5&4&3&3&8&7&6&4&3&2&2&7&6&5&2&1&1&0&5&4&6&5&4&3&6&5&4&3&10&9&8&8&7&6\\
			\hline
			\(b_{19}\)&7&6&5&4&4&3&8&7&5&4&3&3&2&1&6&9&8&8&7&6&5&7&6&5&4&7&6&5&4&5&4&3&9&8&7\\
			\hline
			\(b_{20}\)&2&1&6&5&5&4&3&2&6&5&4&4&3&2&7&10&9&9&8&7&6&8&7&6&5&8&7&6&5&6&5&4&4&3&8\\
			\hline
			\(b_{21}\)&6&5&4&9&3&8&7&6&4&3&8&2&7&6&5&2&7&1&6&5&4&12&11&10&9&0&5&4&3&10&9&8&8&7&6\\
			\hline
			\(b_{22}\)&7&6&5&4&4&9&8&7&5&4&3&3&8&7&6&3&2&2&1&6&5&7&6&5&4&1&6&5&4&11&10&9&9&8&7\\
			\hline
			\(b_{23}\)&8&7&6&5&5&4&9&8&6&5&4&4&3&2&7&4&3&3&2&7&6&8&7&6&5&2&1&6&5&6&5&4&10&9&8\\
			\hline
			\(b_{24}\)&3&2&1&6&0&5&4&3&7&6&5&5&4&3&8&5&4&4&3&2&7&9&8&7&6&3&2&1&6&7&6&5&5&4&9\\
			\hline
			\(b_{25}\)&6&5&4&3&9&8&7&6&4&3&2&8&7&6&5&2&1&7&6&5&4&0&5&4&3&12&11&10&9&10&9&8&8&7&6\\
			\hline
			\(b_{26}\)&7&6&5&4&4&9&8&7&5&4&3&3&8&7&6&3&2&2&1&6&5&1&6&5&4&7&6&5&4&11&10&9&9&8&7\\
			\hline
			\(b_{27}\)&8&7&6&5&5&4&9&8&6&5&4&4&3&2&7&4&3&3&2&7&6&2&1&6&5&8&7&6&5&6&5&4&10&9&8\\
			\hline
			\(b_{28}\)&3&2&1&0&6&5&4&3&7&6&5&5&4&3&8&5&4&4&3&2&7&3&2&1&6&9&8&7&6&7&6&5&5&4&9\\
			\hline
			\(b_{29}\)&8&7&6&5&5&10&9&8&6&5&4&4&9&8&7&4&3&3&2&7&6&2&7&6&5&2&7&6&5&12&11&10&10&9&8\\
			\hline
			\(b_{30}\)&9&8&7&6&6&5&10&9&7&6&5&5&4&3&8&5&4&4&3&8&7&3&2&7&6&3&2&7&6&7&6&5&11&10&9\\
			\hline
			\(b_{31}\)&4&3&2&1&1&6&5&4&8&7&6&6&5&4&9&6&5&5&4&3&8&4&3&2&7&4&3&2&7&8&7&6&6&5&10\\
			\hline
			\(b_{32}\)&10&9&8&7&7&6&11&10&8&7&6&6&5&4&9&6&5&5&4&9&8&4&3&8&7&4&3&8&7&2&7&6&12&11&10\\
			\hline
			\(b_{33}\)&5&4&3&2&2&1&6&5&9&8&7&7&6&5&10&7&6&6&5&4&9&5&4&3&8&5&4&3&8&3&2&7&7&6&11\\
			\hline
			\(b_{34}\)&0&5&4&3&3&2&1&6&10&9&8&8&7&6&11&8&7&7&6&5&10&6&5&4&9&6&5&4&9&4&3&8&2&7&12\\
			\hline
		\end{tabular}%
	}
	\endgroup
\end{table}

\begin{table}[H]
	\centering
	\caption{The adjacent-pair degree matrix \(E(b_i,b_j)\) for \(B_3^{(1)}\) at level \(2\).}
	\label{tab:B3-1-level-two-energy-matrix}
	\begingroup
	\setlength{\tabcolsep}{1.6pt}
	\renewcommand{\arraystretch}{0.75}
	\resizebox{0.7\textwidth}{!}{%
		\begin{tabular}{|c|*{27}{c|}}
			\hline
			\diagbox{\(b_i\)}{\(b_j\)}
			& \(b_0\)& \(b_1\)& \(b_2\)& \(b_3\)& \(b_4\)& \(b_5\)
			& \(b_6\)& \(b_7\)& \(b_8\)& \(b_9\)& \(b_{10}\)
			& \(b_{11}\)& \(b_{12}\)& \(b_{13}\)& \(b_{14}\)
			& \(b_{15}\)& \(b_{16}\)& \(b_{17}\)& \(b_{18}\)& \(b_{19}\)& \(b_{20}\)
			& \(b_{21}\)& \(b_{22}\)& \(b_{23}\)& \(b_{24}\)& \(b_{25}\)& \(b_{26}\)
			\\
			\hline
			\(b_0\)&12&11&10&8&7&6&10&9&7&6&5&8&6&5&4&4&3&2&2&1&0&9&8&7&5&4&3\\
			\hline
			\(b_1\)&7&6&5&3&2&1&11&10&8&7&6&9&7&6&5&5&4&3&3&2&1&4&9&8&6&5&4\\
			\hline
			\(b_2\)&8&7&6&4&3&2&6&5&3&2&1&10&8&7&6&6&5&4&4&3&2&5&4&9&7&6&5\\
			\hline
			\(b_3\)&10&9&8&6&5&4&8&7&5&4&3&6&4&3&2&8&7&6&6&5&4&7&6&5&3&2&1\\
			\hline
			\(b_4\)&11&10&9&7&6&5&9&8&6&5&4&7&5&4&3&3&2&1&7&6&5&8&7&6&4&3&2\\
			\hline
			\(b_5\)&6&5&4&2&1&0&10&9&7&6&5&8&6&5&4&4&3&2&2&1&6&3&8&7&5&4&3\\
			\hline
			\(b_6\)&2&7&6&4&3&2&12&11&9&8&7&10&8&7&6&6&5&4&4&3&2&5&10&9&7&6&5\\
			\hline
			\(b_7\)&3&2&7&5&4&3&7&6&4&3&2&11&9&8&7&7&6&5&5&4&3&6&5&10&8&7&6\\
			\hline
			\(b_8\)&5&4&3&7&6&5&9&8&6&5&4&7&5&4&3&9&8&7&7&6&5&2&7&6&4&3&2\\
			\hline
			\(b_9\)&6&5&4&2&7&6&4&3&1&0&5&8&6&5&4&4&3&2&8&7&6&3&2&7&5&4&3\\
			\hline
			\(b_{10}\)&1&6&5&3&2&1&11&10&8&7&6&9&7&6&5&5&4&3&3&2&7&4&9&8&6&5&4\\
			\hline
			\(b_{11}\)&4&3&8&6&5&4&2&7&5&4&3&12&10&9&8&8&7&6&6&5&4&7&6&11&9&8&7\\
			\hline
			\(b_{12}\)&6&5&4&8&7&6&4&3&7&6&5&2&0&5&4&10&9&8&8&7&6&3&2&1&5&4&3\\
			\hline
			\(b_{13}\)&7&6&5&3&8&7&5&4&2&1&6&9&7&6&5&5&4&3&9&8&7&4&3&8&6&5&4\\
			\hline
			\(b_{14}\)&2&1&6&4&3&2&6&5&3&2&7&10&8&7&6&6&5&4&4&3&8&5&4&9&7&6&5\\
			\hline
			\(b_{15}\)&8&7&6&10&9&8&6&5&9&8&7&4&2&7&6&12&11&10&10&9&8&5&4&3&7&6&5\\
			\hline
			\(b_{16}\)&9&8&7&5&10&9&7&6&4&3&8&5&3&8&7&7&6&5&11&10&9&6&5&4&2&7&6\\
			\hline
			\(b_{17}\)&4&3&2&6&5&4&8&7&5&4&9&6&4&3&8&8&7&6&6&5&10&1&6&5&3&2&7\\
			\hline
			\(b_{18}\)&10&9&8&6&11&10&8&7&5&4&9&6&4&9&8&2&7&6&12&11&10&7&6&5&3&8&7\\
			\hline
			\(b_{19}\)&5&4&3&1&6&5&9&8&6&5&10&7&5&4&9&3&2&7&7&6&11&2&7&6&4&3&8\\
			\hline
			\(b_{20}\)&0&5&4&2&1&6&10&9&7&6&11&8&6&5&10&4&3&8&2&7&12&3&8&7&5&4&9\\
			\hline
			\(b_{21}\)&9&8&7&5&4&3&7&6&4&3&2&5&3&2&1&7&6&5&5&4&3&6&5&4&2&1&0\\
			\hline
			\(b_{22}\)&4&3&2&6&5&4&8&7&5&4&3&6&4&3&2&8&7&6&6&5&4&1&6&5&3&2&1\\
			\hline
			\(b_{23}\)&5&4&3&7&6&5&3&2&6&5&4&7&5&4&3&9&8&7&7&6&5&2&1&6&4&3&2\\
			\hline
			\(b_{24}\)&7&6&5&9&8&7&5&4&8&7&6&3&1&6&5&11&10&9&9&8&7&4&3&2&6&5&4\\
			\hline
			\(b_{25}\)&8&7&6&4&9&8&6&5&3&2&7&4&2&7&6&6&5&4&10&9&8&5&4&3&1&6&5\\
			\hline
			\(b_{26}\)&3&2&1&5&4&3&7&6&4&3&8&5&3&2&7&7&6&5&5&4&9&0&5&4&2&1&6\\
			\hline
		\end{tabular}%
	}
	\endgroup
\end{table}

\section*{Statements and Declarations}

\noindent\textbf{Competing interests.}
The author declares that there are no competing interests.



\end{document}